\documentclass[12pt]{amsart}
\usepackage{amssymb}
\usepackage{bm}
\usepackage{booktabs}
\usepackage{hyphenat}
\usepackage{mathtools}
\usepackage{enumitem}
\usepackage{xcolor}
\usepackage{graphicx}
\usepackage[lmargin=1in,rmargin=1in,tmargin=1in,bmargin=1in]{geometry}

\definecolor{darkblue}{rgb}{0,0,0.4}
\definecolor{mediumblue}{rgb}{0,0,0.6}
\definecolor{darkred}{rgb}{0.7,0,0.0}
\usepackage[colorlinks=true, citecolor=darkblue, filecolor=darkblue, linkcolor=darkblue,urlcolor=darkblue]{hyperref} 

\allowdisplaybreaks

\usepackage{tikz}
\usetikzlibrary{matrix,arrows,cd}
\tikzcdset{scale cd/.style={every label/.append style={scale=#1},
    cells={nodes={scale=#1}}}}
\definecolor{darkgreen}{rgb}{.35,.51,.3} 

\graphicspath{{draws/}{}}

\newcommand{\RR}{\mathbb R}

\newcommand{\ZZ}{\mathbb Z}
\newcommand{\QQ}{\mathbb Q}

\newcommand{\FF}{\mathbb F}
\newcommand{\NN}{\mathbb N}

\newcommand{\lra}{\longrightarrow}

\newcommand{\comma}{\mathbin ,}
\newcommand{\co}{\nobreak\mskip2mu\mathpunct{}\nonscript
  \mkern-\thinmuskip{:}\penalty300\mskip6muplus1mu\relax}

\newcommand{\bdy}{\partial}

\DeclareMathOperator{\Hom}{Hom}

\DeclareMathOperator{\id}{id}

\DeclareMathOperator{\im}{im}

\DeclareMathOperator{\Pin}{\mathit{Pin}}

\DeclareMathOperator{\Split}{split}

\DeclareMathOperator{\gr}{gr}

\theoremstyle{plain}

\numberwithin{equation}{section}
\newtheorem{theorem}[equation]{Theorem}

\newtheorem{proposition}[equation]{Proposition}
\newtheorem{lemma}[equation]{Lemma}
\newtheorem{corollary}[equation]{Corollary}

\newtheorem{definition}[equation]{Definition}

\theoremstyle{definition}

\newtheorem{question}[equation]{Question}

\theoremstyle{remark}
\newtheorem{example}[equation]{Example}
\newtheorem{remark}[equation]{Remark}

\newcommand{\tild}[1]{{\widetilde{#1}}}
\newcommand{\Kh}{\mathit{Kh}}
\newcommand{\KhCx}{\mathit{CKh}}
\newcommand{\oKhCx}{\mathit{CKh}_o}
\newcommand{\oKh}{\mathit{Kh}_o}
\newcommand{\roKh}{\widetilde{\mathit{Kh}}_o}
\newcommand{\KhCxBN}{\KhCx_h} 
\newcommand{\rKhCx}{\tild{\mathit{CKh}}}
\newcommand{\roKhCx}{\tild{\mathit{CKh}}_o}
\newcommand{\rKhCxBN}{\tild{\KhCx}_h} 

\newcommand{\CCat}[1]{\underline{2}^{#1}}

\newcommand{\BNring}{\mathcal{R}}

\newcommand{\rKh}{\widetilde{\Kh}}

\newcommand\HH{\mathit{HH}}

\newcommand\Hochschild\HH

\DeclareMathOperator{\sgn}{sgn}

\newcommand{\ol}[1]{\overline{#1}{}}
\newcommand{\wt}[1]{\widetilde{#1}{}}

\newcommand{\onto}{\twoheadrightarrow}

\newcommand{\Sq}{\mathrm{Sq}}

\newcommand{\LEO}{\mathit{LEO}}
\newcommand{\LEE}{\mathit{LEE}}
\newcommand{\CLEO}{\mathcal{C}_\LEO}
\newcommand{\rCLEO}{\wt{\mathcal{C}}_\LEO}

\newcommand{\kernoverline}[3]{{\mkern #1mu\overline{\mkern-#1mu #2\mkern-#3mu}\mkern#3mu}}

\newcommand{\Kbar}{{\kernoverline{3.5}{K}{1}}}
\newcommand{\fbar}{{\kernoverline{2.5}{f}{0}}}

\newcommand{\stil}{\tilde{s}\vphantom{s}}

\newcommand{\svec}{\bm{s}}
\newcommand{\svectil}{\tilde{\bm{s}}\vphantom{\bm{s}}}
\renewcommand{\sgn}{\sigma}

\begin{document}
\title{Local equivalence and refinements of Rasmussen's s-invariant}

\author{Nathan M. Dunfield}
 \address{Dept.~of Mathematics, University of Illinois Urbana-Champaign\\
   Urbana, IL 61801, USA}
\thanks{NMD was partially supported by US NSF grants DMS-1811156 and
    DMS-2303572 and by a fellowship from the Simons Foundation (673067, Dunfield).}
\email{\href{mailto:nathan@dunfield.info}{nathan@dunfield.info}}
\urladdr{https://dunfield.info}

\author{Robert Lipshitz}
 \address{Department of Mathematics, University of Oregon\\
   Eugene, OR 97403, USA}
\thanks{RL was supported by NSF grant DMS-2204214 and a grant from the Simons Foundation (899815, Lipshitz).}
\email{\href{mailto:lipshitz@uoregon.edu}{lipshitz@uoregon.edu}}

\date{\today}

\author{Dirk Sch\"utz}
 \address{Department of Mathematical Sciences, Durham University\\
   Durham, DH1 3LE, UK}
\email{\href{mailto:dirk.schuetz@durham.ac.uk}{dirk.schuetz@durham.ac.uk}}

\begin{abstract}
  Inspired by the notions of local equivalence in monopole and
  Heegaard Floer homology, we introduce a version of local equivalence
  that combines odd Khovanov homology with equivariant even Khovanov
  homology into an algebraic package called a local even-odd (LEO)
  triple.  We get a homomorphism from the smooth concordance group
  $\mathcal{C}$ to the resulting local equivalence group
  $\mathcal{C}_{\mathit{LEO}}$ of such triples.  We give several
  versions of the $s$-invariant that descend to
  $\mathcal{C}_{\mathit{LEO}}$, including one that completely
  determines whether the image of a knot $K$ in
  $\mathcal{C}_{\mathit{LEO}}$ is trivial.  We discuss computer
  experiments illustrating the power of these invariants in
  obstructing sliceness, both statistically and for some interesting
  knots studied by Manolescu-Piccirillo.  Along the way, we explore
  several variants of this local equivalence group, including one that
  is totally ordered.

 
\end{abstract}

\leavevmode
\vspace{-0.25cm}

\maketitle

\vspace{-0.25cm}

\tableofcontents



\section{Introduction}

Over the last two decades, knot homologies---knot Floer homology,
Khovanov homology, and related invariants---have led to a wealth of
new concordance invariants. One of the first of these was
Ozsv\'ath-Szab\'o's homomorphism $\tau\co \mathcal{C}\to\ZZ$ from the
smooth concordance group, defined using knot Floer
homology~\cite{OSz03:tau}. Rasmussen imitated the construction of
$\tau$ using the Lee deformation of Khovanov homology to obtain
another concordance homomorphism $s\co \mathcal{C}\to\ZZ$, and he used
it to give the first combinatorial proof of Milnor's conjecture on the
unknotting number of torus knots~\cite{Rasmussen10:s}. The homomorphism $s$
was soon shown to be different from
$\tau$~\cite{HeddenOrding08:tau-neq-s} and continues to play a
somewhat special role in the subject
(e.g.,~\cite{Piccirillo20:Conway,HMP21:Mazur}). Other examples include
Manolescu-Owens's homomorphism $\delta$ coming from the correction
term for the branched double cover~\cite{ManolescuOwens07:delta},
Hom's invariant $\varepsilon\in\{-1,0,1\}$~\cite{Hom14:epsilon},
Ozsv\'ath-Stipsicz-Szab\'o's $\Upsilon$~\cite{OSSz17:Upsilon}, and
many others.

Another class of concordance invariants was inspired by Manolescu's
disproof of the Triangulation Conjecture using $\Pin(2)$-equivariant
Seiberg-Witten theory~\cite{Manolescu16:triangulation}. There, Manolescu
introduced homology cobordism invariants $\alpha$, $\beta$, and
$\gamma$ of homology 3-spheres. While studying them,
Stoffregen~\cite{Stoffregen20:Seifert} abstracted the key features of
Seiberg-Witten theory underlying their existence: for all homology
3-spheres, the Seiberg-Witten Floer spectrum has the same
$\Pin(2)$-fixed set (corresponding to the reducible connections), and
the maps associated to homology cobordisms induce homotopy
equivalences of these fixed sets. So, the local behavior near the
fixed set gives a homology cobordism invariant. This behavior can be
captured at the level of cochain complexes by using the Atiyah-Bott
localization theorem, leading to Stoffregen's notion of ``chain local
equivalence.''

Partial analogues of Stoffregen's construction work for Heegaard Floer
homology of closed 3-manifolds, using Hendricks-Manolescu's involutive
Heegaard Floer homology~\cite{HM17:involutive}.  For example,
Hendricks, Manolescu, and Zemke built a group whose elements are chain
complexes over a particular ring, modulo an equivalence relation they
called local equivalence, so that involutive Heegaard Floer homology
induces a homomorphism from the 3-dimensional homology cobordism group
to this group~\cite{HMZ18:inv-sum}. Using a variant of their definition,
whose group structure is easier to analyze, Dai, Hom, Stoffregen, and Truong
were able to prove that the homology cobordism group has a
$\ZZ^\infty$ summand~\cite{DHST:homol-cob}. Analogous constructions
also work for the concordance group of knots, using knot Floer
homology~\cite{Zemke19:inv-HFK-sum,DHST21:more-conc}.

The first goal of this paper is to give an analogous construction
using Khovanov homology. Given a knot $K$, we consider the pair
$(\oKhCx(K),\KhCxBN(K))$ of the odd Khovanov complex of $K$ and the
Bar-Natan deformation of the even Khovanov complex of $K$. These
complexes come with quotient maps to $\KhCx(K;\FF_2)$, the even
Khovanov complex with coefficients in $\FF_2$. We call the
data of (complexes like) $\oKhCx(K)$ and $\KhCxBN(K)$, and a homotopy
equivalence between their mod-2 reductions, a \emph{local even-odd
  triple} or \emph{LEO triple} (Definition~\ref{def:local}).  Using
the fact that knot cobordisms induce isomorphisms on $\KhCxBN$ after
suitably localizing, we define a notion of local equivalence for LEO
triples (Definition~\ref{def:local-map}) and prove:
\begin{theorem}
  The local equivalence classes of LEO triples form an abelian group
  $\CLEO$, and the Khovanov complexes induce a group homomorphism from
  the smooth concordance group to this local equivalence group.
\end{theorem}
\noindent(This is stated more precisely as Theorem~\ref{thm:local-equiv-group}.)

We also give several versions of the construction modeled on reduced
Khovanov homology. 
A \emph{reduced LEO triple} $(C, D, f)$ consists of
finitely generated, free, bigraded, cochain complexes $C$ over $\ZZ$ and
$D$ over $\ZZ[h]$, and a bigraded homotopy equivalence
\[
  f\co C\otimes_\ZZ\ZZ/(2)\to D\otimes_{\ZZ[h]}\ZZ[h]/(2,h),
\]
so that $D \otimes_{\ZZ[h]} \ZZ[h, h^{-1}]$ is homotopy equivalent to
a free graded module of rank 1 over $\ZZ[h,h^{-1}]$. Local equivalence
classes of these form an abelian group $\rCLEO$.  There is also a
variant where $D$ is merely defined over $\FF_2[h]$, giving a group
$\rCLEO^o$, and here the analogy with Heegaard Floer
homology goes farther:

\begin{theorem}\label{thm:intro-totally-ordered}
  There is an invariant $\stil_o\in\ZZ$ of reduced LEO triples so that
  $[(C,D,f)]$ is trivial in $\rCLEO^o$ if and only if
  $\stil_o(C,D,f)=\stil_o((C,D,f)^*)=0$, where $(C,D,f)^*$ is the
  inverse to $(C,D,f)$ in $\rCLEO^o$. Further, the relation that
  $[(C,D,f)]\geq0$ if and only if $\stil_o(C,D,f)\geq0$ makes
  $\rCLEO^o$ into a totally ordered abelian group.
\end{theorem}
\noindent
This is re-stated and proved as Theorems~\ref{thm:reduced-local-is-e}
and~\ref{thm:o-total-order}. There are partial analogues for some
other variants of local equivalence as well:
Theorem~\ref{thm:reduced-local-is-e} for $\rCLEO$ and
Theorem~\ref{thm:local-is-e} for the subgroup of $\CLEO$ coming from
knots. 
Theorem~\ref{thm:intro-totally-ordered} is reminiscent of Hom's work
on the $\varepsilon$-invariant, with $\rCLEO^o$ being analogous to the group
$\mathcal{CFK}$ of knot Floer-like complexes
modulo those  with $\varepsilon=0$ \cite{Hom15:inf-rank}.

The groups $\CLEO$ and $\rCLEO$ contain $\ZZ^\infty$-summands, though
we do not know if these are realized by knots
(Section~\ref{subsec:s-invt}); $\rCLEO^o$ has at least a
$\ZZ$-summand and is not isomorphic to $\ZZ$. Beyond this, the
structure of these groups is open; see Remark~\ref{remark:nothing}. We
do show that the images of the concordance group in $\CLEO$ and
$\rCLEO$ are isomorphic (Corollary~\ref{cor:CLEO-to-rCLEO-iso}). (The
image in $\rCLEO^o$ is smaller.) Understanding the structure of these
groups further might lead to new concordance information.

The local equivalence class associated to a knot determines its
$s$-invariant, as well as the refinements using the even and odd
$\Sq^1$-operations (cases of refinements studied
earlier~\cite{LS14:refine-s,SSS20:odd-htpy}). We give further integer
invariants of knots that are determined by the local equivalence class
of their Khovanov complexes. Certain of these invariants completely
determine whether a knot $K$ gives rise to the trivial element of
$\CLEO$, $\rCLEO$, and, $\rCLEO^o$, see
Theorems~\ref{thm:reduced-local-is-e} and~\ref{thm:local-is-e}.  Even
the most subtle of our invariants can be computed in practice for a
given knot $K$ with 20 crossings using \texttt{KnotJob}
\cite{knotjob}, allowing us to determine when such
$K$ gives rise to the trivial element of $\CLEO$.  We demonstrate
their efficacy by computing them for many knots where the
$s$-invariant and its even and odd $\Sq^1$-refinements vanish.  First,
we looked at the roughly 18,000 prime knots with at most 19 crossings
whose (smooth) slice status has not been resolved by Dunfield and
Gong~\cite{DunfieldGong2023} using a wide range of techniques; as
there are 352 million prime knots in this range~\cite{Burton2020},
these 18,000 knots are very unusual in their difficulty to analyze.
The integer invariants of Section~\ref{sec:knot-case} obstruct
sliceness for 890 of these knots, reducing the number of mystery knots
by 5\%.  Second, we use these invariants to give an alternate proof
that the five intriguing knots of
Manolescu-Piccirillo~\cite{ManolescuPiccirillo2023} are not slice;
this was originally shown by Nakamura~\cite{Nakamura2022} using
0-surgery homeomorphisms to stably relate slice properties of two
knots.

After posting this paper, we learned that Lewark has independently studied a notion of local equivalence for Khovanov homology \cite{Lewark23:Oberwolfach,Lewark24:local}, although without odd Khovanov homology.

This paper is organized as follows. Section~\ref{sec:background}
collects the results on Khovanov homology needed for the rest of the
paper. In particular, Section~\ref{sec:Kunneth} recalls Putyra's
K\"unneth theorem for the odd Khovanov homology of connected sums
(which is more subtle than the even
case). Section~\ref{sec:local-equiv} introduces our notion of local
equivalence and develops its basic
properties. Section~\ref{sec:refined-s} gives integer invariants of
local equivalence classes coming from refinements of $s$, and uses
them to further study the local equivalence group. The ties between
these invariants and the structure of the local equivalence groups is
studied in Section~\ref{sec:LEO-structure}; in particular,
Theorem~\ref{thm:intro-totally-ordered} is proved there. These
invariants have somewhat simpler behavior for LEO triples coming from
knots, and we study properties specific to knots in
Section~\ref{sec:knot-case}. We conclude, in
Section~\ref{sec:computations}, with results of computer computations
of these invariants, including the applications to certain difficult
knots mentioned above.

\subsection*{Acknowledgments} We thank Ciprian Manolescu, Sucharit
Sarkar, and the anonymous referee for helpful conversations and
comments.



\section{Background}\label{sec:background}

\subsection{The even and odd Khovanov complexes and the Bar-Natan
  deformation}\label{sec:review}
In this section, we introduce the notation we will use for the
Khovanov complex and its variants, and specify our grading
conventions.

Given a commutative ring $R$, we will denote the Khovanov complex for
a knot $K$ with coefficients in $R$ by $\KhCx(K;R)$.
We largely follow the conventions for Khovanov homology
from the first papers in the subject~\cite{Khovanov00:CatJones,BarNatan02:on-khovanovs,Rasmussen10:s}.
The Khovanov complex is constructed by applying the Frobenius algebra
$R[X]/(X^2)$ with comultiplication
\[
  \Delta(1)=1\otimes X+X\otimes 1\qquad \Delta(X)=X\otimes X
\]
(over $R$) to the cube of resolutions for $K$, and then taking either an
iterated mapping cone or the total complex.  Specifically, in
$\KhCx(K;R)$ the differential goes from the $0$-resolution of a
crossing to the $1$-resolution, and the homological and quantum
gradings of a generator $x$ lying over the vertex $v\in\{0,1\}^n$ are
given by:
\begin{align*}
  \gr_h(x) &= |v| - n_-\\
  \gr_q(x) &= |v| + \#\{x(C)=1\} - \#\{x(C)=X\} + n_+ - 2 n_-,
\end{align*}
where $n_\pm$ are the number of positive and negative crossings of
$K$, $|v|$ is the number of entries of $v$ which are $1$, and where
$\#\{x(C)=1\}$ and $\#\{x(C)=X\}$ are the number of circles
in the resolution $K_v$ that $x$ labels $1$ and $X$ respectively.
We will write $\KhCx^{i,j}(K;R)$ for the summand of $\KhCx(K;R)$
spanned by generators $x$ with $(\gr_h(x),\gr_q(x))=(i,j)$; the
differential is a map $\KhCx^{i,j}(K;R)\to \KhCx^{i+1,j}(K;R)$. The
homology of $\KhCx(K;R)$ is $\Kh(K;R)=\bigoplus_{i,j}\Kh^{i,j}(K;R)$.

We will use $\{n\}$ to denote an upwards shift in the quantum grading
by $n$, so
\[
  \bigl(\KhCx(K;R)\{1\}\bigr)^{i,j}=\KhCx^{i,j-1}(K;R).
\]
At the level of graded Euler characteristic, $\{1\}$ corresponds to
multiplying by $q$.

Choosing a basepoint $p$ on $K$, not at any of the crossings, makes
$\KhCx(K;R)$ into a module over $R[X]/(X^2)$: $X$ acts by multiplying
the label of the circle containing $p$ in each resolution by
$X$. Equivalently, this module structure corresponds to merging a
0-crossing unknot at $p$.

The Bar-Natan deformation of the Khovanov Frobenius algebra is the
Frobenius algebra $\BNring=R[h,X]/(X^2-hX)$ over $R[h]$ with
comultiplication
\[
  \Delta(1)=1\otimes X+X\otimes 1 - h \otimes 1\qquad
  \Delta(X)=X\otimes X.
\]
(This is a signed version of a construction of Bar-Natan's~\cite[Section
9.3]{BarNatan05:Kh-tangle-cob}. The tensor products are over $R[h]$,
so $h\otimes 1 = 1\otimes h$.)  We view $h$ as having quantum
grading $-2$. The unit is still $1$, and the counit to $R[h]$ sends
$1\mapsto 0$ and $X\mapsto 1$. Applying this Frobenius algebra to the
cube of resolutions gives the Bar-Natan complex $\KhCxBN^{i,j}(K;R)$,
which is again a bigraded complex, but now over the graded ring
$R[h]$.

Fixing a basepoint $p$ on $K$ makes $\KhCxBN(K;R)$ into a module over
$\BNring$, where $X$ acts on a vertex of the cube by multiplying the
label of the circle containing $p$ by $X$. Again, this can also be
thought of as merging in an unknot at $p$.

A fundamental property of the Bar-Natan complex is that
\[
h^{-1}\Kh_h(K;R)\coloneqq\Kh_h(K;R)\otimes_{R[h]}R[h,h^{-1}] \cong
R[h,h^{-1}]\oplus R[h,h^{-1}]
\]
(e.g.,~\cite[Proposition 2.1]{LS22:mixed}); the two summands
correspond to the two orientations of $K$, and both lie in homological
grading $0$. Since localization is exact, we could equivalently
localize $\KhCxBN^{i,j}(K;R)$ and then take homology. In fact, this
identification holds at the chain level, and respects the action of
$X$ as well: there is a chain homotopy equivalence, over $\BNring$,
\begin{equation}\label{eq:BN-cx-localized}
  h^{-1}\KhCxBN(K;R) \simeq h^{-1}\BNring=\BNring\otimes_{R[h]}R[h,h^{-1}]
\end{equation}
\cite[Proof of Proposition 2.1]{LS22:mixed}.
We will often abuse notation and write $h^{-1}\Kh_h^{i,j}(K;R)$ to
mean the part of $h^{-1}\Kh_h(K;R)$ in bigrading $(i,j)$.

Ozsv\'ath, Rasmussen, and Szab\'o constructed another variant of
Khovanov homology, odd Khovanov homology
$\Kh_o(K;R)=\bigoplus_{i,j}\Kh_o^{i,j}(K;R)$~\cite{ORSz13:odd}. (To
make the distinction clear, we henceforth refer to the original
Khovanov homology as the \emph{even Khovanov homology}.)  The
underlying odd Khovanov complex $\oKhCx(K;R)$ is defined similarly to
the even Khovanov complex, except that instead of viewing the space
associated to a vertex $K_v$ in the cube of resolutions as a tensor
product of copies of $R[X]/(X^2)$ over the circles in $K_v$, it is
viewed as the exterior algebra on the set of circles in $K_v$. As an
$R$-module, if one chooses an ordering of the circles in $K_v$, there
is an identification between the groups, sending
$C_{i_1}\wedge\cdots\wedge C_{i_k}\in\oKhCx(K;R)$ to the generator
$x\in \KhCx(K;R)$ which labels each $C_{i_j}$ by $X$ and the other
circles by $1$. The exterior algebra leads to different signs, and
hence a different invariant if $\mathrm{char}(R)\neq 2$. On the other
hand, $\Kh_o(K;\FF_2)$ and $\Kh(K;\FF_2)$ are canonically isomorphic.

Putyra observed that the complex $\oKhCx(K;R)$ associated to a based
knot (or link) is a bimodule over
$R[X]/(X^2)$~\cite{Putyra16:triply}. We review the construction of
this action in Section~\ref{sec:Kunneth}. The left and right actions
on $\oKhCx(K;\ZZ)$ have the same mod-2 reduction, and agree with the
usual action on $\KhCx(K;\FF_2)$. (In fact, they are related by an
automorphism of the complex; see Lemma~\ref{lem:left-is-right}.)

An oriented cobordism $\Sigma\subset [0,1]\times \RR^3$ from $K_0\subset
\{0\}\times \RR^3$ to $K_1\subset \{1\}\times\RR^3$, decomposed as a
movie of elementary cobordisms (Reidemeister moves, births, saddles,
and deaths), induces a chain map
\[
  \Sigma_*\co \KhCx^{i,j}(K_0;R)\to \KhCx^{i,j+\chi(\Sigma)}(K_1;R),
\]
as well as chain maps $\Sigma_*$ on $\KhCxBN$ and $\oKhCx$, with the
same grading shift by $\chi(\Sigma)$.  (The constructions of the maps
$\Sigma_*$ are given
by Khovanov, Bar-Natan, and Putyra~\cite{Khovanov00:CatJones,BarNatan05:Kh-tangle-cob,Putyra14:chron}.)
Up to homotopy and an overall sign, the maps on $\KhCx$ and $\KhCxBN$
are independent of the decomposition of $\Sigma$ as a movie (or,
equivalently, of isotopies of $\Sigma$), but the maps on $\oKhCx$ are
not known to be; we will never use below that any of the maps are
independent of the choice of movie. If none of the Reidemeister
moves in the movie cross the basepoint (i.e., it is a \emph{based
  movie}), then the map $\Sigma_*$ commute with the $X$-action on the
two sides.

For all of the variants of Khovanov homology, the complexes and maps
with $R$ coefficients are obtained from the corresponding complexes
over $\ZZ$ by tensoring with $R$. So, for instance, given a ring
homomorphism $f\co R\to S$, the square
\[
  \begin{tikzcd}
    \KhCx(K_0;R)\arrow[r,"f_*"]\arrow[d,"\Sigma_*"] &
    \KhCx(K_0;S)\arrow[d,"\Sigma_*"] \\
    \KhCx(K_1;R)\arrow[r,"f_*"] &
    \KhCx(K_1;S)
  \end{tikzcd}
\]
commutes. The complex $\KhCx(K;R)$ is obtained from $\KhCxBN(K;R)$ by
setting $h=0$, i.e.,
\[
  \KhCx(K;R)=\KhCxBN(K;R)\otimes_{R[h]}R
\]
where $R=R[h]/(h)$. Again, this identification is functorial,
in the sense that
\[
  \begin{tikzcd}
    \KhCxBN(K_0;R)\arrow[r,"/(h)"]\arrow[d,"\Sigma_*"] &
    \KhCx(K_0;R)\arrow[d,"\Sigma_*"] \\
    \KhCxBN(K_1;R)\arrow[r,"/(h)"] &
    \KhCx(K_1;R)
  \end{tikzcd}
\]
commutes.

There are reduced versions of the Khovanov complex, lifting the
reduced Jones polynomial. If we view $R$ as an $R[X]/(X^2)$-module
where $X$ acts by $0$ then the reduced Khovanov complex is
\[
  \rKhCx(K;R)=\KhCx(K;R)\otimes_{R[X]/(X^2)}R\{-1\}.
\]
Similarly, we can make $R[h]$ into a module over $\BNring$ by letting
$X$ act by $0$, and 
\[
  \rKhCxBN(K;R)=\KhCxBN(K;R)\otimes_{\BNring}R[h]\{-1\}.
\]
Finally, using either the left or right module structure on
$\oKhCx(K;R)$, we have
\[
  \roKhCx(K;R)=\oKhCx(K;R)\otimes_{R[X]/(X^2)}R\{-1\}.
\]
By construction, these are all natural with respect to changes of ring
$R$, and based movies induce maps on the reduced complexes.

\subsection{K\"unneth theorems}\label{sec:Kunneth}
The main reason we record the action of $R[X]/(X^2)$ on Khovanov
homology is the well-known K\"unneth theorem for connected sums:

\begin{lemma}\label{lem:Kh-kunneth}
  Given based knot (or link) diagrams $K_1$ and $K_2$, there is an isomorphism
  of chain complexes
  \begin{equation}\label{eq:Kh-kunneth}
    \KhCx(K_1\#K_2;R)\cong \KhCx(K_1;R)\otimes_{R[X]/(X^2)}\KhCx(K_2;R),
  \end{equation}
  where $K_1\#K_2$ is the result of taking the connected sum at the
  basepoints. Moreover, if we choose a basepoint on $K_1\#K_2$ to be
  one of the two arcs that connect $K_1$ and $K_2$, then the
  isomorphism~\eqref{eq:Kh-kunneth} intertwines the actions of
  $R[X]/(X^2)$ on the two sides. Further, given a ring homomorphism
  $R\to R'$, the isomorphism~\eqref{eq:Kh-kunneth} commutes with the
  change-of-ring maps $\KhCx(K;R)\to\KhCx(K;R')$ in the obvious sense.
\end{lemma}

\begin{proof}
  It is immediate from the construction of Khovanov homology that 
  \begin{equation}
    \label{eq:Kh-kunneth-amalg}
    \KhCx(K_1\amalg K_2;R)\cong \KhCx(K_1;R)\otimes_{R}\KhCx(K_2;R).
  \end{equation}
  Let $U$ be a 0-crossing unknot which is adjacent to the basepoints
  in $K_1$ and $K_2$.  The two cobordisms
  \[
    K_1\amalg U\amalg K_2\rightrightarrows K_1\amalg K_2\to K_1\#K_2,
  \]
  gotten by first merging $U$ to $K_1$ and then to $K_2$, or first to
  $K_2$ and then to $K_1$, induce the same map of Khovanov
  complexes. Using the isomorphism~\eqref{eq:Kh-kunneth-amalg} and the
  fact that $\KhCx(U;R)=R[X]/(X^2)$, we get a commutative diagram
  \begin{align*}
\KhCx(K_1;R)\otimes_R R[X]/(X^2)\otimes_R\KhCx(K_2;R)&\rightrightarrows \KhCx(K_1;R)\otimes_R \KhCx(K_2;R)\\&\to \KhCx(K_1\#K_2;R),
  \end{align*}
  and hence a chain map
  $\KhCx(K_1;R)\otimes_{R[X]/(X^2)} \KhCx(K_2;R)\to
  \KhCx(K_1\#K_2;R)$. It is easy to see this map induces a bijection
  on generators, and hence is a chain isomorphism.

  All of these maps commute with merging on an unknot somewhere else
  to $K_1$ or $K_2$, and hence respect the $R[X]/(X^2)$-module
  structure. It is also immediate from the construction that the maps
  are natural with respect to change of rings.
\end{proof}

\begin{remark}
Lemma~\ref{lem:Kh-kunneth} can be seen as a special case of Khovanov's
tangle invariant~\cite{Khovanov02:Tangles} for the 2-ended tangles
obtained by deleting neighborhoods of the basepoints.
\end{remark}

\begin{lemma}\label{lem:Kh-kunneth-BN}
  With notation as in Lemma~\ref{lem:Kh-kunneth}, there is an
  isomorphism of chain complexes
  \begin{equation}
    \KhCxBN(K_1\#K_2;R)\cong \KhCxBN(K_1;R)\otimes_{\BNring}\KhCxBN(K_2;R)\label{eq:BN-Kunneth}
  \end{equation}
  respecting the action of $\BNring$ and natural with respect to
  changes in ground ring. Moreover, reducing the
  isomorphism~\eqref{eq:BN-Kunneth} modulo $h$ gives the
  isomorphism~\eqref{eq:Kh-kunneth}.
\end{lemma}
\begin{proof}
  The proof is the same as for Lemma~\ref{lem:Kh-kunneth},
  using the isomorphism
  \[
    \KhCxBN(L_1\amalg L_2;R)\cong \KhCxBN(L_1;R)\otimes_{R[h]}\KhCxBN(L_2;R)
  \]
  in place of Formula~\eqref{eq:Kh-kunneth-amalg}.
\end{proof}

For odd Khovanov homology, the K\"unneth theorems for disjoint unions
and connected sums were proved by Putyra~\cite{Putyra16:triply}, and
are substantially more complicated. Putyra in fact proves more general
results about unified Khovanov homology. For the reader's
convenience, we give a streamlined explanation for just the case of
odd Khovanov homology. To keep notation short, we state the results
only for $R=\ZZ$; the general case follows by tensoring all of the
complexes with $R$.

To define the odd Khovanov complex, one first picks \emph{crossing
  orientations}. A generator of the complex $\oKhCx(K)$ at a vertex
$v$ of the cube is an element of the exterior algebra
$\Lambda(\pi_0K_v)$ on the components of the complete resolution
$K_v$. Reducing modulo 2, the generator
$Z_{i_1}\wedge\cdots\wedge Z_{i_k}$ of $\oKhCx(K;\FF_2)$, where the
$Z_{i_j}$ are circles in the $v$-resolution of $K$, corresponds to the
generator of $\KhCx(K;\FF_2)$ which labels the circles $Z_{i_j}$ by
$X$ and the other circles by $1$. To get the differential, one
also uses an \emph{edge assignment}~\cite[Definition 1.2]{ORSz13:odd},
which assigns $\pm1$ to each edge
in a way that makes each face anti-commute, and subject to an extra
condition on so-called \emph{type $X$} and \emph{type $Y$} faces.

A key observation of Putyra's is that for each vertex $v$ in the cube
of resolutions of $K$, there is an integer $\Split(v)$, the
\emph{number of splits before $v$}. Explicitly,
\[
\Split(v)=\frac{1}{2}\bigl(|\pi_0 K_v|-|\pi_0 K_{\vec{0}}|+|v|\bigr),
\]
the difference in the number of circles in $K_v$ and the number of circles in $K_{\vec{0}}$, plus the sum of the entries of $v$, divided by $2$.

Fix a basepoint $p$ on $K$. Given a complete resolution $K_v$ of $K$,
let $Z_p\subset K_v$ be the circle containing $p$. There is a right
action of
$\ZZ[X]/(X^2)$ on $\oKhCx(K)$ by
\begin{equation}
  \label{eq:right-act}
(Z_{i_1}\wedge\cdots\wedge Z_{i_j})\cdot X=Z_{i_1}\wedge\cdots\wedge
Z_{i_j}\wedge Z_p.
\end{equation}
Since the merge map merging circles $Z_1$ and
$Z_2$ to a new circle $Z$ is induced by the projection
$Z_1,Z_2\mapsto Z$, $Z_i\mapsto Z_i$ for $i\neq 1,2$, it is immediate
that this commutes with the merge maps in the odd Khovanov cube; it also
commutes with the split maps since splitting a circle $Z$ into $Z_1$ and
$Z_2$ corresponds to multiplying on the left by $Z_1-Z_2$, which
commutes with multiplying on the right by $Z_p$. Thus, $\oKhCx(K)$ is a
cochain complex over $\ZZ[X]/(X^2)$. There is also a left action of
$\ZZ[X]/(X^2)$ defined by
\begin{equation}
  X\cdot(Z_{i_1}\wedge\cdots\wedge Z_{i_j})=(-1)^{\Split(v)}Z_p\wedge Z_{i_1}\wedge\cdots\wedge Z_{i_j}.
\end{equation}
where $Z_{i_1},\dots,Z_{i_j}$ are circles in the resolution
$K_v$. Again, this evidently commutes with all merge maps; it is
straightforward to check that it commutes with all split maps as well.

Here is the special case of Putyra's result~\cite[Theorem
6.7]{Putyra16:triply} that we use in this paper:
\begin{theorem}\label{thm:odd-Kunneth}
  Given link diagrams $K_1$ and $K_2$, there is an isomorphism of chain complexes
  \begin{equation}
    \label{eq:odd-disjoint-union}
    \oKhCx(K_1\amalg K_2)\cong \oKhCx(K_1)\otimes_{\ZZ}\oKhCx(K_2).
  \end{equation}
  Given based link diagrams $K_1$ and $K_2$, there is an isomorphism of chain complexes
  \begin{equation}
    \label{eq:odd-con-sum}
    \oKhCx(K_1\#K_2)\cong \oKhCx(K_1)\otimes_{\ZZ[X]/(X^2)}\oKhCx(K_2).
  \end{equation}
  Further, if we choose the basepoint for $\oKhCx(K_1\#K_2)$ to be on
  one of the two arcs where the connected sum occurred, then the chain
  isomorphism~\eqref{eq:odd-con-sum} respects the
  $\ZZ[X]/(X^2)$ bimodule structure. Finally, if we reduce the
  isomorphisms~\eqref{eq:odd-disjoint-union}
  and~\eqref{eq:odd-con-sum} modulo 2, we obtain the
  isomorphisms~\eqref{eq:Kh-kunneth-amalg} and~\eqref{eq:Kh-kunneth}
  with $R=\FF_2$.
\end{theorem}
\begin{proof}
We start with the result about disjoint unions. Suppose $K_i$ has
$n_i$ crossings. For convenience, fix an ordering of the crossings
of each $K_i$, so we can identify the cube of resolutions of $K_i$
with $\CCat{n_i}$, where $\CCat{1}=(0\to 1)$. Fix also an edge
assignment $\varepsilon_i$ for $K_i$.

Let $v=(v_1,v_2) \in \{0,1\}^{n_1+n_2}$, and let $\oKhCx^{v_i}(K_i)$
denote the summand of $\oKhCx(K_i)$ corresponding to the
$v_i$-resolution. Given a basis element $w_i\in\oKhCx^{v_i}(K_i)$, so
$w_i$ is a wedge product of circles in $(K_i)_{v_i}$, let $|w_i|$
denote the number of factors in the wedge product.
Define $f\colon \oKhCx^{v_1}(K_1)\otimes_\ZZ \oKhCx^{v_2}(K_2) \to \oKhCx^{v}(K_1\sqcup K_2)$ by
\[
f(w_1\otimes w_2) = (-1)^{|w_1|\cdot \Split(v_2)}w_1\wedge w_2.
\]
Clearly this induces an isomorphism on cochain groups
\[
f\colon  \oKhCx^{i}(K_1)\otimes_\ZZ \oKhCx^{j}(K_2) \to \oKhCx^{i+j}(K_1\sqcup K_2),
\]
and thus we can define a differential on $\oKhCx(K_1\sqcup K_2)$ by
$\partial^f = f\circ \partial^\otimes\circ f^{-1}$. We need to check
that this differential is the differential of an odd Khovanov complex
for $K_1\sqcup K_2$; that is, we need to check that it differs from
the usual merge and split maps by an edge assignment.

We first show that $\partial^f$ restricted to an edge $v\to v'$ in the
cube of resolutions differs from the merge or split map by a sign only
depending on the edge, and then show that this sign function is indeed
an edge assignment. To see the first part, we need to check the
various cases with $v_1\to v_1'$ a merge or split, and $v_2\to v_2'$ a
merge or split. The most interesting case is the case when
$v_2\to v_2'$ is a split, and this case will also make clear how the
other easier cases work. We have
\begin{align*}
\partial^f_{v\to v'}(w_1\wedge w_2) &= f\partial^\otimes_{v\to v'}((-1)^{|w_1|\cdot \Split(v_2)} w_1\otimes w_2) \\
&= f((-1)^{\gr_h(w_1)+\varepsilon_2(v_2\to v_2')+|w_1|\cdot \Split(v_2)}w_1\otimes (A-B)\wedge w_2)\\
&=(-1)^{\gr_h(w_1)+\varepsilon_2(v_2\to v_2')+|w_1|}w_1\wedge (A-B)\wedge w_2 \\
&=(-1)^{\gr_h(w_1)+\varepsilon_2(v_2\to v_2')} (A-B)\wedge w_1\wedge w_2,
\end{align*}
where $A$ and $B$ are the two circles created by the split, and the
orientation of the handle (crossing orientation) points from $A$ to $B$.
Since $\gr_h(w_1)$ only depends on $v_1$, the sign difference $(-1)^{\gr_h(w_1)+\varepsilon_2(v_2\to v_2')}$ between $\partial^f_{v\to v'}$ and the split map depends only on the edge $v\to v'$, and not on the particular generator involved, as desired. The other three cases are similar, but easier: for a $v_2\to v'_2$ merge, the difference is $(-1)^{\gr_h(w_1)+\varepsilon_2(v_2\to v'_2)}$, for a $v_1\to v'_1$ split the difference is $(-1)^{\Split(v_2)+\varepsilon_1(v_1\to v'_1)}$, and for a $v_1\to v'_1$ merge, the difference is $(-1)^{\varepsilon_1(v_1\to v'_1)}$.

It remains to show that the sign function is indeed an edge
assignment. For faces in $\CCat{n_1+n_2}$ of type $A$ or $C$, this
follows from the fact that $\partial^f\circ \partial^f =
0$. For faces of type $X$ and $Y$, observe that such faces have to be
either in $\CCat{n_1}$ or in $\CCat{n_2}$. In either case, the product
of the signs around the face is the same as in $(K_i,\varepsilon_i)$,
because there is either an extra contribution of $(-1)^{4\gr_h(w_1)}$
or of $(-1)^{2\Split(v_2)}$.

Turning to the result for connected sums, note that there is a natural surjection on cochain groups
\[
m\colon \oKhCx(K_1\sqcup K_2) \to \oKhCx(K_1\# K_2)
\]
by merging the appropriate circles. The edge assignment on the disjoint union descends to an edge assignment on the connected sum; again, this follows from $\partial\circ \partial = 0$, since $X$- and $Y$-type faces have to be in $\CCat{n_1}$ or $\CCat{n_2}$ only. In particular, $m$ is a cochain map with this choice.

Now,
\[
g = m\circ f\colon \oKhCx(K_1)\otimes_\ZZ \oKhCx(K_2) \to \oKhCx(K_1\sqcup K_2) \to \oKhCx(K_1\# K_2).
\]
is a cochain map. We claim that $g(w_1 X \otimes w_2) = g(w_1 \otimes X w_2)$. Indeed, if we
let $Z_p$ and $Z_{p'}$ be the components of $(K_1)_{v_1}$ and $(K_2)_{v_2}$ containing the basepoints, then
\[
g(w_1X\otimes w_2) = m(f(w_1\wedge Z_p\otimes w_2) = m((-1)^{(|w_1|+1)\cdot \Split(v_2)}w_1\wedge Z_p\wedge w_2),
\]
while
\[
g(w_1\otimes X w_2) = m(f((-1)^{\Split(v_2)}w_1\otimes Z_{p'} \wedge w_2) = m((-1)^{\Split(v_2)+|w_1|\cdot \Split(v_2)}w_1\wedge Z_{p'}\wedge w_2).
\]
After the merge, $Z_p$ and $Z_{p'}$ represent the same element, so these expressions agree. Thus, $g$ descends to a cochain map
\[
\overline{g}\colon \oKhCx(K_1) \otimes_{\ZZ[X]/(X^2)}\oKhCx(K_2) \to \oKhCx(K_1\# K_2).
\]
We also get that $\overline{g}$ preserves the left and right $\ZZ[X]/(X^2)$-module structures. For the right action, this is obvious. For the left action,
\[
\overline{g}(X w_1\otimes w_2) = \overline{g}((-1)^{\Split(v_1)} Z_p\wedge w_1 \otimes w_2) = (-1)^{\Split(v_1)+(|w_1|+1)\Split(v_2)} m(Z_p\wedge w_1\wedge w_2),
\]
while
\[
X \overline{g}(w_1\otimes w_2) = X m((-1)^{|w_1|\cdot\Split(v_2)}w_1\wedge w_2) = (-1)^{\Split(v) + |w_1|\cdot \Split(v_2)} m(Z_p\wedge w_1\wedge w_2).
\]
So, since $\Split(v) = \Split(v_1)+\Split(v_2)$ the left action of $X$ is also preserved.
\end{proof}

The left and right module structures on $\oKhCx$ are related by an
automorphism of the complex. Specifically,
decompose $\oKhCx(K)$ as $C'(K)\oplus C''(K)$ where
$C'(K)=\oKhCx(K)\cdot X$ is the image of multiplication by $X$ (i.e.,
is spanned by wedge products containing $Z_p$ as a factor) and
$C''(K)$ is the complement to $C'(K)$ spanned by the empty wedge
products and wedge products of the
form $(Z_1-Z_p)\wedge Z_2\wedge\cdots\wedge Z_i$ (where the $Z_j$ are
arbitrary circles in the resolution not containing the
basepoint). Define $h\co C'\oplus C''\to C'\oplus C''$
\[
  h(w',w'')=(w',(-1)^{\Split(v)+|w''|}w''),
\]
where $v$ is the vertex of the cube that $w''$ lies over.
\begin{lemma}\label{lem:left-is-right}
  The map $h$ is a chain isomorphism satisfying $h(X\cdot w)=h(w)\cdot X$.
\end{lemma}
\begin{proof}
  This is straightforward from the definitions.
\end{proof}

\begin{corollary}
  Given based link diagrams $K_1$ and $K_2$, if we view both
  $\oKhCx(K_i)$ as right modules over the commutative ring $\ZZ[X]/(X^2)$ by
  using the action in Equation~\eqref{eq:right-act}, then there is an isomorphism
  of chain complexes of $\ZZ[X]/(X^2)$-modules
  \[
    \oKhCx(K_1\#K_2)\cong \oKhCx(K_1)\otimes_{\ZZ[X]/(X^2)}\oKhCx(K_2)
  \]
  where the action on the left-hand side is again the right action by
  $\ZZ[X]/(X^2)$ (and the basepoint is on one of the arcs
  where the connected sum occurred).
\end{corollary}
\begin{proof}
  This is immediate from Theorem~\ref{thm:odd-Kunneth} and
  Lemma~\ref{lem:left-is-right}.
\end{proof}

So, for the rest of the paper, we will only think of $\oKhCx(K)$ as a
right module over $\ZZ[X]/(X^2)$, and ignore the left action.

The K\"unneth theorems for reduced Khovanov homology are slightly simpler:
\begin{lemma}
  There are chain isomorphisms
  \begin{align}
    \rKhCx(K_1\#K_2;R)&\cong \rKhCx(K_1;R)\otimes_{R}\rKhCx(K_2;R)\\
    \rKhCxBN(K_1\#K_2;R)&\cong \rKhCxBN(K_1;R)\otimes_{R[h]}\rKhCxBN(K_2;R)\\
    \roKhCx(K_1\# K_2;R)&\cong \roKhCx(K_1;R)\otimes_{R}\roKhCx(K_2;R).
  \end{align}
  Moreover, the second isomorphism mod $h$ is the first one, and the first
  and third isomorphisms agree mod $2$.
\end{lemma}
\begin{proof}
  The result is immediate from Lemmas~\ref{lem:Kh-kunneth}
  and~\ref{lem:Kh-kunneth-BN}, Theorem~\ref{thm:odd-Kunneth}, and the
  definitions of the reduced complexes.
\end{proof}

\section{Even-odd local equivalence}\label{sec:local-equiv}
In this section, we introduce a notion of local equivalence
combining even and odd Khovanov homology. We work
over the ground ring $R=\ZZ$, so in particular
$\BNring=\ZZ[X,h]/(X^2-Xh)$.

\begin{definition}
  \label{def:local}
  A \emph{local even-odd (LEO) triple} consists of a finitely generated, bigraded
  cochain complex $C$ over $\ZZ[X]/(X^2)$, a finitely
  generated, bigraded cochain complex $D$ over $\BNring$,
  and a bigraded chain homotopy equivalence
  \[
    f\co C\otimes_{\ZZ} \ZZ/(2) \to D\otimes_{\BNring}\BNring/(2,h),
  \]
  so that:
  \begin{itemize}
  \item $C$ is freely generated over $\ZZ[X]/(X^2)$,
  \item  $D$ is freely generated over $\BNring$,
  \item the map $f$ is a homomorphism of cochain complexes over
    $\FF_2[X]/(X^2)$, and
  \item the localization $h^{-1}D=D\otimes_{\BNring}h^{-1}\BNring$ is
    homotopy equivalent to a free graded module of rank 1 over $h^{-1}\BNring$,
    supported in homological grading $0$ and odd quantum gradings.
  \end{itemize}
\end{definition}

Given a knot diagram $K$, let
$f\co \oKhCx(K)\otimes_\ZZ\FF_2\to\KhCxBN(K)\otimes_{\BNring}\FF_2$ be
the identification of the mod-2 reduction of the odd Khovanov complex
with the mod-$(2,h)$ reduction of the Bar-Natan complex (both of which
are $\KhCx(K;\FF_2)$).  It follows from the results cited in
Section~\ref{sec:review} that
\[
  \LEO(K)=(\oKhCx(K),\KhCxBN(K),f)
\]
is a LEO triple, and this is our motivating example.  A special case
is the \emph{trivial LEO triple}
\[
  \LEO(U)=(\ZZ[X]/(X^2)\{1\},\BNring\{1\},\id)
\]
associated to the unknot $U$. (The set braces indicate quantum grading
shifts.) Another LEO triple associated to $K$ that uses only the even
Khovanov homology is
\[
  \LEE(K)=(\KhCx(K),\KhCxBN(K),f),
\]
where $f$ is the obvious identification of
$\KhCxBN(K)\otimes_{\BNring}\FF_2$ with $\KhCx(K;\FF_2)$. While the data in
$\LEE(K)$ is completely determined by $\LEO(K)$, we will use
it later to put certain existing invariants into our framework.

\begin{definition}\label{def:local-map}
  Given LEO triples $(C,D,f)$ and $(C',D',f')$, a \emph{local
    map} from $(C,D,f)$ to $(C',D',f')$ consists of bigrading-preserving chain maps
  $\alpha\co C\to C'$ and $\beta\co D\to D'$ (respecting the
  module structures) so that
  \begin{itemize}
  \item the induced map $\beta\co h^{-1}D\to h^{-1}D'$ is a homotopy
    equivalence and
  \item the following diagram commutes up to homotopy:
  \begin{equation}\label{eq:local-map}
    \begin{tikzcd}
       C\arrow[r]\arrow[d,"\alpha"] & C\otimes_\ZZ\ZZ/(2)\arrow[d,"\alpha"]\arrow[r,"f"] &
       D\otimes_{\BNring} \BNring/(2,h)\arrow[d,"\beta"] & D \arrow[l]\arrow[d,"\beta"] \\
       C'\arrow[r] & C'\otimes_\ZZ\ZZ/(2)\arrow[r,"f'"] &
       D'\otimes_{\BNring} \BNring/(2,h) & D' \arrow[l].
    \end{tikzcd}
  \end{equation}
  \end{itemize}
  (The first and third squares automatically commute on the nose, so
  the condition is equivalent to homotopy commutativity of the
  second square. Also, the homotopy
  $f' \circ \alpha \sim \beta \circ f$ is required to respect the
  action of $\FF_2[X]/(X^2)$.) We say that $(C,D,f)$ and $(C',D',f')$
  are \emph{locally equivalent} if there are local maps
  $(C,D,f)\to(C',D',f')$ and $(C',D',f')\to(C,D,f)$.
\end{definition}

A key motivation for Definition~\ref{def:local-map} is the following:
\begin{proposition}\label{prop:conc-implies-local-equiv}
  If $K_0$ and $K_1$ are (smoothly) concordant, then $\LEO(K_1)$ and
  $\LEO(K_2)$ are locally equivalent, as are $\LEE(K_1)$ and
  $\LEE(K_2)$. In particular, if $K$ is slice then $\LEO(K)$ and
  $\LEE(K)$ are locally equivalent to $\LEO(U)$.
\end{proposition}
\begin{proof}
  We prove the result for $\LEO$; the statements for $\LEE$ follow
  since it is determined by $\LEO$. Suppose that $\Sigma$ is a concordance from $K_0$ to $K_1$. Fix a
  description of $\Sigma$ as a movie. Then there is a diagram:
  \begin{equation}\label{eq:Kh-local-equiv}
    \begin{tikzcd}
      \oKhCx(K_0)\arrow[r,"/(2)"]\arrow[d, "\alpha"] & \KhCx(K_0;\FF_2)\arrow[d] &
      \KhCxBN(K_0) \arrow[r]\arrow[d, "\beta"]\arrow[l,"/(2\comma h)" above] & h^{-1}\KhCxBN(K_0)\arrow[d,"\simeq"]\\
      \oKhCx(K_1)\arrow[r,"/(2)"] & \KhCx(K_1;\FF_2) &
      \KhCxBN(K_1)\arrow[l,"/(2\comma h)" above]\arrow[r] & h^{-1}\KhCxBN(K_1).
    \end{tikzcd}
  \end{equation}
  Here, all the vertical arrows are induced by $\Sigma$, and are
  bigrading-preserving; the second vertical arrow is the map on
  $\KhCx(\cdot;\FF_2)$ induced by the movie. Commutativity of the
  middle and right squares follows from the principle at the end of
  Section~\ref{sec:review}, that the maps on the even Khovanov complex
  with $R$-coefficients (for any $R$) are induced by the maps over
  $\ZZ$, and the map on the Bar-Natan complex reduces to the usual map
  on the Khovanov complex. Commutativity of the left square is a
  special case of a result of Sarkar-Scaduto-Stoffregen about the odd
  Khovanov spectrum~\cite[Lemma 5.11]{SSS20:odd-htpy}, and is implicit
  in Putyra's work~\cite{Putyra14:chron}. In particular, all of the
  squares commute exactly, not just up to homotopy.

  The fact that the vertical arrow at the right is a homotopy
  equivalence essentially follows from a result of
  Rasmussen's~\cite[Corollary 4.2]{Rasmussen10:s}, though he was
  working over $\QQ$, used a different Frobenius algebra, and only
  stated it for homology. The result for the Bar-Natan deformation, over
  arbitrary rings, is well known (e.g.,~\cite[Proposition
  3.4]{LS22:mixed} where the result is
  stated only for homology, but the proof gives the chain level statement).

  To arrange that $\alpha$ commutes with the action of $\ZZ[X]$ and
  $\beta$ commutes with the action of $\BNring$, we modify $\Sigma$ and
  its movie representative slightly. Specifically, the action of $X$
  uses a basepoint on $K_0$ and $K_1$, though up to homotopy
  equivalence the resulting complexes of modules are independent of
  the choice of basepoint. Choose $\Sigma$ and the movie representing
  it so that no Reidemeister move in the movie for $\Sigma$ crosses
  the basepoint, and the component containing the basepoint never
  disappears (dies), and choose the basepoint on $K_1$ to be the image
  of the basepoint on $K_0$ under the movie. (To arrange the condition
  about Reidemeister moves, one can replace any move that crosses the
  basepoint by a sequence of moves going around the rest of the
  diagram and through infinity. Note that this may change the isotopy
  class of $\Sigma$ in $[0,1]\times S^3$.) Then, the map associated to
  each step in the movie for $\Sigma$ commutes with the $X$-action. 
  
  Thus, $\Sigma$ induces a local map $\LEO(K_0)\to \LEO(K_1)$. Reading
  $\Sigma$ backward gives a local map $\LEO(K_1)\to\LEO(K_0)$, so the
   $\LEO(K_0)$ and $\LEO(K_1)$ are equivalent.
\end{proof}

With this motivation in hand, we return to developing some general
properties of LEO triples and local equivalence.

\begin{lemma}
  Local equivalence of LEO triples is an equivalence relation.
\end{lemma}
\begin{proof}
  This is straightforward.
\end{proof}

We can make the set of local equivalence classes of LEO triples into a
group, as follows:
\begin{definition}\label{def:tensor-prod}
  The \emph{tensor product} of LEO triples $(C,D,f)$ and $(C',D',f')$
  is
  \[
  \big(C\otimes_{\ZZ[X]/(X^2)}C'\{-1\}, \,
  D\otimes_{\BNring}D'\{-1\}, \,f\otimes f'\big).
  \]
  The \emph{dual} of $(C,D,f)$ is
  \[
  (\Hom_{\ZZ[X]/(X^2)}(C,\ZZ[X]/(X^2))\{2\}, \, \Hom_{\BNring}(D,\BNring)\{2\}, \, \fbar),
  \]
  where we grade dual complexes so that
  $\Hom(C,\ZZ)_{i,j}=\Hom(C_{-i,-j},\ZZ)$, and
  the $\{2\}$ denotes a quantum grading shift.
  The map $\fbar$ is the
  transpose (map of dual complexes) induced by the inverse homotopy
  equivalence to $f$. We will denote the dual by $(C,D,f)^*$ or $(C^*,
  D^*, f^*)$. (The latter is a slight abuse of notation.)
\end{definition}

\begin{theorem}\label{thm:local-equiv-group}
  The set of local equivalence classes of LEO triples forms an abelian
  group $\CLEO$, with addition the tensor product from
  Definition~\ref{def:tensor-prod}, identity $\LEO(U)$, and inverse
  given by the dual from
  Definition~\ref{def:tensor-prod}. Moreover, the assignment $K\mapsto
  \LEO(K)$ induces a homomorphism from the smooth
  concordance group $\mathcal{C}$ to $\CLEO$.
\end{theorem}
\begin{proof}
  For the first statement, we must verify that:
  \begin{enumerate}[label=(\arabic*)]
  \item\label{item:gp-mul-defd} The tensor product of LEO triples is a
    LEO triple.
  \item\label{item:gp-mul-equiv} Tensor product respects local
    equivalence.
  \item\label{item:gp-inv-defd} The dual of a LEO triple is a LEO triple.
  \item\label{item:gp-inv-equiv} The dual respects local equivalence.
  \item\label{item:gp-assoc} The tensor product is associative.
  \item\label{item:gp-id} The triple $\LEO(U)$ is a unit for
    tensor product.
  \item\label{item:gp-abelian} The tensor product is commutative up to
    local equivalence.
  \item\label{item:gp-inverses} For any LEO triple
    $(C,D,f)$, $(C,D,f)\otimes (C,D,f)^*$ is locally equivalent
    to $\LEO(U)$.
  \end{enumerate}
  Points~\ref{item:gp-mul-defd},~\ref{item:gp-mul-equiv},~\ref{item:gp-inv-defd},
  and~\ref{item:gp-inv-equiv} are straightforward from the
  definitions; note that for Point~\ref{item:gp-inv-equiv}, a local
  map from $(C,D,f)$ to $(C',D',f')$ induces a map
  from $(C',D',f')^*$ to $(C,D,f)^*$.
  Points~\ref{item:gp-assoc},~\ref{item:gp-id},
  and~\ref{item:gp-abelian} are immediate from the definitions. For
  Point~\ref{item:gp-inverses}, we claim there is a canonical isomorphism of LEO triples
  \[
    (C,D,f)\otimes (C,D,f)^*=(\Hom_{\ZZ[X]/(X^2)}(C,C)\{1\},\Hom_{\BNring}(D,D)\{1\},F),
  \]
  where $F(\eta)=f\circ \eta\circ \fbar$ (and $\fbar$ is the
  homotopy inverse to $f$). For the first two components, this bigraded isomorphism is immediate from the definitions, the hypothesis that $C$ and $D$ are finitely-generated complexes of free modules, and the usual isomorphism $\Hom_R(X,Y)\cong \Hom_R(X,R)\otimes Y$ for $X$ a finitely generated, free $R$-module. In particular, the $\{1\}$ grading shift comes from the $\{2\}$ grading shift for the dual and the $\{-1\}$ grading shift for the tensor product. Tracing through these identifications also shows that $F$ has the specified form. 
  
  Now, consider the map from $\LEO(U)$ to this
  tensor product defined by
  \begin{align*}
    \alpha\co \ZZ[X]/(X^2)\{1\}&\to\Hom_{\ZZ[X]/(X^2)}(C,C)\{1\} &     \beta\co \BNring\{1\}&\to \Hom_{\BNring}(D,D)\{1\}
  \end{align*}
  both of which send $1$ to the identity map.  We claim that
  Diagram~\eqref{eq:local-map} homotopy commutes. As noted
  above, it suffices to check that the second square, which has the form
  \[
    \begin{tikzcd}
      \FF_2[X]/(X^2) \arrow[r,"\id"]\arrow[d,"\alpha"] &
      \FF_2[X]/(X^2)\arrow[d,"\beta"]\\
      \Hom_{\ZZ[X]/(X^2)}(C,C)\otimes_\ZZ\ZZ/(2) \arrow[r,"F"] &
      \Hom_{\BNring}(D,D)\otimes\BNring/(2,h),
    \end{tikzcd}
  \]
  homotopy commutes. (We are suppressing the quantum grading shifts,
  and will continue to do so for the rest of the proof.)
  Since $D$ is free over $\BNring$,
  \[
    \Hom_{\BNring}(D,D)\otimes\BNring/(2,h)=\Hom_{\BNring/(2,h)}(D\otimes\BNring/(2,h),D\otimes\BNring/(2,h))
  \]
  If $k\co D\otimes\BNring/(2,h)\to D\otimes \BNring/(2,h)$ is the
  homotopy from $f\circ\fbar$ to the identity, then the homotopy
  making this square commute sends $a\in \FF_2[X]/(X^2)$ to
  $a\cdot
  k\in\Hom_{\BNring/(2,h)}(D\otimes\BNring/(2,h),D\otimes\BNring/(2,h))$
  (i.e., sends $1$ to $k$). Finally, after inverting $h$, the map
  \[
    \beta\co h^{-1}\BNring\to
    h^{-1}\Hom_{\BNring}(D,D)=\Hom_{h^{-1}\BNring}(h^{-1}D,h^{-1}D)
  \]
  sends $1$ to the identity map. So, up to homotopy, the map
  to
  \[
    \Hom_{h^{-1}\BNring}(h^{-1}\BNring,h^{-1}\BNring)\cong h^{-1}\BNring
  \]
  induced by the homotopy equivalence $h^{-1}D\simeq h^{-1}\BNring$ is
  an isomorphism. (Here, by $h^{-1}\BNring$ we really mean a free
  module of rank 1 over $h^{-1}\BNring$.) So, the map $\beta\co h^{-1}\BNring\to
    h^{-1}\Hom_{\BNring}(D,D)$ is a homotopy equivalence.

  Similarly, there is a local map from $(C,D,f)\otimes (C,D,f)^*$
  to $\LEO(U)$ given by letting
  $\alpha'\co \Hom_{\ZZ[X]/(X^2)}(C,C)\to \ZZ[X]/(X^2)$ and
  $\beta'\co \Hom_{\BNring}(D,D)\to \BNring$ be the following trace maps. On a basic tensor
  $c^\vee\otimes c'\in \Hom_{\ZZ[X]/(X^2)}(C,\ZZ[X]/(X^2))\otimes C$, we have
  $\alpha(c^\vee\otimes c')=c^\vee(c')$.
  (When checking this is a chain
  map, it is important to remember that the differential on the dual
  complex $\Hom_{\ZZ[X]/(X^2)}(C,\ZZ[X]/(X^2))$ is given by
  $d(f)=(-1)^{\gr_h(f)+1}f\circ\partial$, and the differential on the
  tensor product of course inherits the Koszul sign.) 
  The map $\beta'$ is defined
  similarly. It is clear that $\alpha'$ and $\beta'$ are chain
  maps. The two ways around the
  second square give the maps $c^\vee\otimes c' \mapsto c^\vee(c')$ and
  $c^\vee\otimes c'\mapsto c^\vee(\fbar(f(c')))$; if $\ell$ is the
  homotopy from $\fbar\circ f$ to the identity map then the
  homotopy making the second square commute is
  $c^\vee\otimes c'\mapsto c^\vee(\ell(c))$. Again, the facts that the
  evaluation map $\Hom_{h^{-1}\BNring}(h^{-1}\BNring,h^{-1}\BNring)\to
  h^{-1}\BNring$ is an isomorphism and that $h^{-1}D\simeq
  h^{-1}\BNring$ imply that the map
  $h^{-1}\Hom_{\BNring}(D,D)\to h^{-1}\BNring$ induced by $\beta'$ is
  a homotopy equivalence.

  For the second statement, in light of
  Proposition~\ref{prop:conc-implies-local-equiv}, all that remains is
  to show that if $K_0$ and $K_1$ are knots then $\LEO(K_0\#K_1)$ is
  locally equivalent to $\LEO(K_0)\otimes\LEO(K_1)$; but this is
  immediate from Lemma~\ref{lem:Kh-kunneth-BN} and Theorem~\ref{thm:odd-Kunneth}.
\end{proof}

\begin{remark}
  The conditions we have placed on a LEO triple $(C,D,f)$ are not
  strong enough to rule out some pathological behavior which does not
  occur for $\LEO(K)$ (see Example~\ref{eg:pathology}). There is other
  structure in Khovanov homology, such as the automorphism $I$ and
  chain map $T$ of Section~\ref{subsec:conjugation}, which one might
  be able to abstract to exclude such behavior.
\end{remark}

\subsection{Reduced versions}
Recall that the reduced Khovanov complexes of a based link are given
by
\begin{align*}
  \rKhCxBN(K)&=\KhCxBN(K)\otimes_{\BNring}\ZZ[h]\{-1\}\\
  \roKhCx(K)&=\oKhCx(K)\otimes_{\ZZ[X]/(X^2)}\ZZ\{-1\}\\
  \rKhCx(K)&=\KhCx(K)\otimes_{\ZZ[X]/(X^2)}\ZZ\{-1\}
\end{align*}
where $X$ acts by $0$ on $\ZZ[h]$ and $\ZZ$. 
Here is an analogue of Definitions~\ref{def:local}
and~\ref{def:local-map} motivated by the reduced
Khovanov complexes:
\begin{definition}\label{def:reduced-local-equiv}
  A \emph{reduced LEO triple} consists of finitely generated, free,
  bigraded cochain complexes $C$ over $\ZZ$ and $D$ over $\ZZ[h]$, and a
  bigraded homotopy equivalence
  \[
    f\co C\otimes_\ZZ\ZZ/(2)\to D\otimes_{\ZZ[h]}\ZZ[h]/(2,h),
  \]
  so that $h^{-1}D$ is homotopy equivalent to a free graded
  module of rank 1 over $\ZZ[h,h^{-1}]$, supported in homological grading $0$
  and even quantum gradings.

  A \emph{local map} of such triples from $(C,D,f)$ to $(C',D',f')$
  consists of bigrading-preserving chain maps $\alpha\co C\to C'$ and
  $\beta\co D\to D'$ so that
  \begin{itemize}
  \item the induced map $\beta\co h^{-1}D\to h^{-1}D'$ is a homotopy
    equivalence and
  \item the following diagram commutes up to homotopy:
  \begin{equation}\label{eq:reduced-local-map}
    \begin{tikzcd}
       C\otimes_\ZZ\ZZ/(2)\arrow[d,"\alpha"]\arrow[r,"f"] &
       D\otimes_{\ZZ[h]} \ZZ[h]/(2,h)\arrow[d,"\beta"] \\
       C'\otimes_\ZZ\ZZ/(2)\arrow[r,"f'"] &
       D'\otimes_{\ZZ[h]} \ZZ[h]/(2,h).
    \end{tikzcd}
  \end{equation}
  \end{itemize}
  We call $(C,D,f)$ and $(C',D',f')$ are \emph{locally equivalent}
  if there are local maps $(C,D,f)\to(C',D',f')$ and
  $(C',D',f')\to(C,D,f)$.
\end{definition}

This definition has many of the same formal properties as the
unreduced case:
\begin{proposition}\label{prop:rCLEO}
  Reduced LEO triples satisfy the following:
  
  \begin{enumerate}[label=(\arabic*)]
  \item\label{item:red-equiv-rel} Local
    equivalence is an equivalence relation on reduced LEO triples.
  \item\label{item:red-group} If we define the tensor product of
    reduced LEO triples $(C,D,f)$ and $(C',D',f')$ by
    $(C\otimes_{\ZZ}C',D\otimes_{\ZZ[h]}D',f\otimes f')$ and the
    inverse of $(C, D, f)$ by
    \[
      (\Hom_{\ZZ}(C,\ZZ),\Hom_{\ZZ[h]}(D,\ZZ[h]),\fbar)
    \]
    (where
    $\fbar$ is induced by the homotopy inverse to
    $f$), then the set of local equivalence
    classes of reduced LEO triples is an abelian group $\rCLEO$.
  \item\label{item:reduce} There is a homomorphism $\pi\co
    \CLEO\to\rCLEO$ given by
    \[
      (C,D,f) \mapsto
      \big(C\otimes_{\ZZ[X]/(X^2)}\ZZ\{-1\},
      \, D\otimes_{\BNring}\ZZ[h]\{-1\}, \, \wt{f} \big),
    \]
    where $\wt{f}$ is the map induced by $f$.
  \item\label{item:homom} The map $\mathcal{C}\to \rCLEO$ which sends
    $K$ to $(\roKhCx(K),\rKhCxBN(K),\id)$ is a well-defined group
    homomorphism.
  \end{enumerate}
\end{proposition}
\begin{proof}
  Point~\ref{item:red-equiv-rel} is
  immediate from the definition. The proof of
  Point~\ref{item:red-group} is similar to the proof of the first part of
  Theorem~\ref{thm:local-equiv-group}, replacing $\ZZ[X]/(X^2)$
  by $\ZZ$ and $\BNring$ by $\ZZ[h]$ throughout, and is left to the
  reader. Point~\ref{item:reduce} is immediate from the
  definitions. Point~\ref{item:homom} follows from
  Proposition~\ref{prop:conc-implies-local-equiv} and Point~\ref{item:reduce}.
\end{proof}

We give one further specialization of LEO triples, which is perhaps
the simplest formulation incorporating both the Bar-Natan deformation
and a version of the odd Khovanov complex:

\begin{definition}
  \label{def:mod-2-local}
  A \emph{two-reduced local even-odd triple} consists of finitely
  generated, free,
  bigraded cochain complexes $C$ over $\ZZ$ and $D$ over $\FF_2[h]$,
  and a
  bigraded homotopy equivalence
  \[
    f\co C\otimes_\ZZ \ZZ/(2)\to D\otimes_{\FF_2[h]}\FF_2[h]/(h)
  \]
  so that $h^{-1}D$ is homotopy equivalent to a free graded module of
  rank 1 over $\FF_2[h,h^{-1}]$, supported in homological grading $0$
  and even quantum gradings.

  \emph{Local maps} and \emph{local equivalence} of two-reduced local
  even-odd triples are defined as in
  Definition~\ref{def:reduced-local-equiv}, replacing $\ZZ[h]$ by
  $\FF_2[h]$ throughout.
\end{definition}

\begin{proposition}
  The local equivalence classes of two-reduced local even-odd triples
  form an abelian group $\rCLEO^o$, and there is a homomorphism from
  the concordance group to $\rCLEO^o$.
\end{proposition}
\noindent
The proof is the same that of Proposition~\ref{prop:rCLEO}, and is left to
the reader.

\subsection{The \texorpdfstring{$s$}{s}-invariant}\label{subsec:s-invt}

Here is a key property of LEO triples, which will allow us to define
the $s$-invariant:
\begin{lemma}\label{lem:PID_homology}
  Let $(C,D,f)$ be a LEO triple and let $\FF$ be a field. Then the
  homology of $D\otimes_\ZZ\FF$, viewed as a module over $\FF[h]$,
  decomposes as $\FF[h]\oplus\FF[h]\oplus T$ where $T$ is a torsion
  $\FF[h]$-module. Similarly, if $(\wt{C},\wt{D},f)$ is a reduced
  LEO triple, then the homology of $\wt{D}\otimes_\ZZ\FF$ decomposes as
  $\FF[h]\oplus T$ where $T$ is a torsion $\FF[h]$-module.
\end{lemma}
\begin{proof}
  We prove the unreduced case; the reduced case is similar.
  By the classification of modules over a PID, the homology of
  $D\otimes_\ZZ\FF$ is isomorphic to $\FF[h]^m\oplus T$ for some $m$
  and some torsion module $T$. Since localization is exact, it follows
  that the homology of
  $D\otimes_{\ZZ[h]}\FF[h,h^{-1}]$ is isomorphic to $\FF[h,h^{-1}]^m$. On the other
  hand, by hypothesis,
  $D\otimes_{\ZZ[h]}\ZZ[h,h^{-1}]\simeq h^{-1}\BNring\cong
  \ZZ[h,h^{-1}]^2$. So,
  $D\otimes_{\ZZ[h]}\FF[h,h^{-1}]\simeq \FF[h,h^{-1}]^2$, and $m=2$,
  as claimed.
\end{proof}
For the complexes coming from a knot, the generators of the two free
$\FF[h]$ summands of $\Kh_h(K;\FF)$ lie in quantum gradings which are
$2$ apart, and the $s$-invariant is defined to be the average of these
two gradings. (See, for instance, Sch\"utz's paper~\cite[Section
2]{Schutz:integral-s}.)  The following example shows that this is not
true in general for LEO triples.

\begin{example}\label{eg:pathology}
  Let $k$ be a non-negative integer, and $D$ the free
  $\BNring$-cochain complex concentrated in homological degrees $-1$
  and $0$ given by $D^{-1} = \BNring{\{-1\}}$,
  $D^0 = \BNring{\{1\}}\oplus \BNring{\{-1+2k\}}$ and
  $\partial\co D^{-1}\to D^0$ given by $\partial(1) = (X-h,
  h^k)$. (Recall that the quantum grading shift $\{-1+2k\}$, for
  instance, means the generator $1\in\BNring{\{-1+2k\}}$ satisfies
  $\gr_q(1) = -1+2k$.)

Since $h^k$ is an isomorphism over $h^{-1}\BNring$, we can use
Gaussian elimination~\cite[Lemma~3.2]{MR2320156} to see that $h^{-1}D$
is homotopy equivalent to a free graded module of rank 1 supported in homological degree $0$, namely $h^{-1}\BNring{\{1\}}$.
If we restrict coefficients to $\ZZ[h]$, the complex $D$ is given by
\[
\begin{tikzpicture}
\node at (0,2) {$\ZZ[h]{\{-1\}}$};
\node at (0,1) {$\ZZ[h]{\{-3\}}$};
\node at (4,3) {$\ZZ[h]{\{1\}}$};
\node at (4,2) {$\ZZ[h]{\{-1\}}$};
\node at (4,1) {$\ZZ[h]{\{-1+2k\}}$};
\node at (4,0) {$\ZZ[h]{\{-3+2k\}}$};
\draw[->] (0.9, 2.1) -- node [above, sloped] {$-h$} (3.2, 3);
\draw[->] (0.9, 2) -- node [above] {$1$} (3.1, 2);
\draw[->] (0.9, 1.9) -- node [above, sloped] {$h^k$} (2.55, 1);
\draw[->] (0.9, 1) -- node [above, sloped] {$h^k$} (2.55, 0);
\end{tikzpicture}
\]
Another Gaussian elimination shows this is homotopy equivalent to $E$ given by $E^{-1} = \ZZ[h]{\{-3\}}$, $E^0 = \ZZ[h]{\{1\}}\oplus \ZZ[h]{\{-1+2k\}}\oplus \ZZ[h]{\{-3+2k\}}$ and $\partial_E^{-1}(1) = (0, 0, h^k)$. Given a field $\FF$ and the ring homomorphism $\ZZ[h]\to \FF$ sending $h$ to $1$, we get $E\otimes_{\ZZ[h]}\FF$ is homotopy equivalent to a complex with two generators, and their quantum gradings are $1$ and $-1+2k$. In particular, the difference in quantum grading of the two generators can be arbitrarily large in general.
\end{example}

\begin{definition}\label{def:filt-notation}
  Given a bigraded chain complex $D$, let $D^{(q)}$ be the subspace of
  $D$ in quantum grading $q$. Since we are often working with graded
  rings, such as $\ZZ[X]/(X^2)$ or $\ZZ[h]$, $D^{(q)}$ is typically not a
  submodule of $D$. We will denote the homology of $D$ in quantum
  grading $q$ by
  \[
    H^{*,q}(D)=H(D^{(q)}).
  \]
  Finally, for brevity, we will write $H^{i,j}(D;\FF)$ to
  denote $H^{i,j}(D\otimes_\ZZ\FF)$.
\end{definition}

\begin{definition}\label{def:s-of-triple}
  Given a LEO triple $(C,D,f)$ and a field $\FF$, let
  \begin{align*}
    s^+_\FF(C,D,f) &= \max\{q \mid H^{0,q}(D;\FF) \to H^{0,q}(h^{-1}D;\FF) \text{ non-zero}\} - 1\\
    s^-_\FF(C,D,f) &= \max \{q \mid H^{0,q}(D;\FF) \to H^{0,q}(h^{-1}D;\FF) \text{ surjective}\} + 1.
  \end{align*}
  Given a reduced LEO triple $(\wt{C},\wt{D},f)$, let
  \begin{align*}
    s_\FF(\wt{C},\wt{D},f)&=\max \{q \mid H^{0,q}(\wt{D};\FF) \to H^{0,q}(h^{-1}\wt{D};\FF) \text{ non-zero}\}\\
                &= \max \{q \mid H^{0,q}(\wt{D};\FF) \to H^{0,q}(h^{-1}\wt{D};\FF) \text{ surjective}\}.
  \end{align*}
  Given a LEO triple $(C,D,f)$, let $s_\FF(C,D,f)$ be $s_\FF$ of the
  image of $(C,D,f)$ in $\rCLEO$.
\end{definition}
Note that Definition~\ref{def:s-of-triple} does not use $C$ or $f$,
only $D$.

\begin{lemma}
  The numbers $s^\pm_\FF(C,D,f)$ and $s_\FF(C,D,f)$ depend only on
  $\FF$ and the local equivalence class of $(C,D,f)$. Similarly,
  $s_\FF(\wt{C},\wt{D},f)$ depends only on $\FF$ and the local
  equivalence class of $(\wt{C},\wt{D},f)$.
\end{lemma}
\begin{proof}
  This is immediate from the definitions.
\end{proof}

If $(C,D,f)=\LEO(K)$ then for each field $\FF$ the elements
$s^\pm_\FF(C,D,f)$ and $s(C,D,f)$ are all equal (e.g.,~\cite[Section
2.2]{LS14:refine-s},~\cite[Section 2]{Schutz:integral-s}), and agree
with the Rasmussen invariant over $\FF$.  In contrast:

\begin{lemma}\label{lem:splus-not-homom}
  The maps $s^+_\FF$ and $s^-_\FF$ are not group homomorphisms, but
  the map $s_\FF$ is.
\end{lemma}
\begin{proof}
To see that $s^\pm_\FF$ are not group homomorphisms, consider the
complex $D$ from Example~\ref{eg:pathology}, and extend $D$
arbitrarily to a LEO triple. Direct computation gives $s^+(D)=2k-2$
and $s^-(D)=2$. On the other hand, if we let
$D^*=\Hom_{\BNring}(D,\BNring)\{2\}$ be the dual complex then direct computation
again shows that $s^+(D^*)=-2$ and $s^-(D^*)=2-2k$.  By
Part~\ref{item:gp-inverses} of Theorem~\ref{thm:local-equiv-group},
$D\otimes D^*\{-1\}$ is locally equivalent to $\LEO(U)$, so
$s^\pm(D\otimes D^*\{-1\})=s^\pm(U)=0$. For any $k\neq 2$, this shows
$s^+$ and $s^-$ are not group homomorphisms.

To see that $s$ is a group homomorphism, consider reduced LEO triples
$(\wt{C}_i,\wt{D}_i,f_i)$, $i=1,2$. Abusing notation, let $\wt{D}_i$
denote the tensor product of $\wt{D}_i$ with $\FF$.  By the K\"unneth
theorem, the free part of $\wt{D}_1\otimes_{\FF[h]}\wt{D}_2$ is the
tensor product of the free parts of $\wt{D}_1$ and $\wt{D}_2$. The
$s$-invariant is the quantum grading of a generator of this free part,
so the result follows.
\end{proof}

\begin{corollary}
  \label{cor:C-not-onto}
  The image of the smooth knot concordance group $\mathcal{C}$ in
  $\CLEO$ from Theorem~\ref{thm:local-equiv-group} is a proper
  subgroup of $\CLEO$.
\end{corollary}
\begin{proof}
As noted above, $s_\FF = s^+_\FF = s^-_\FF$ on the image of
$\mathcal{C}$, but not on all of $\CLEO$ by
Lemma~\ref{lem:splus-not-homom}.
\end{proof}

\begin{proposition}\label{prop:Zinf-summand}
  The tuple
  $\vec{s}=(s_\QQ,s_{\FF_2}-s_\QQ,s_{\FF_3}-s_{\QQ},s_{\FF_5}-s_{\QQ},\dots)$
  defines a surjective homomorphism
  $\rCLEO\to (2\ZZ)^\infty=\bigoplus_{n\in\NN}2\ZZ$.
\end{proposition}
\begin{proof}
Let $(\wt{C}, \wt{D}, f)$ be a reduced LEO-triple. To
see that $\vec{s}$ is well-defined, we need to show that for all but finitely many primes $p$ we get $s_{\FF_p}(\wt{C},\wt{D},f) = s_{\QQ}(\wt{C},\wt{D},f)$.

Set $s=s_{\QQ}(\wt{C},\wt{D},f)$ and denote by $i\colon \QQ[h]\{s\}\to H^0(\wt{D})\otimes \QQ$, $r\colon H^0(\wt{D})\otimes\QQ\to \QQ[h]\{s\}$ the grading-preserving maps with $r\circ i=\id_{\QQ[h]}$, which exist by Lemma \ref{lem:PID_homology}. Then
\[
i(1) = \sum_{m=1}^k a_m\otimes q_m,
\]
with $a_m\in H^0(\wt{D})$ and $q_m\in \QQ$. There exists a localization $P^{-1}\ZZ\subset \QQ$ of $\ZZ$ obtained by inverting finitely many primes $P$ such that each $q_m\in P^{-1}\ZZ$, $m=1,\ldots, k$.
Also, $\ZZ[h]$ is a Noetherian ring, so $H^0(\wt{D})$ is finitely
generated. After possibly adding finitely many primes to $P$, we can
assume that $r\bigl(H^0(\wt{D})\bigr)\subset P^{-1}\ZZ[h]\{s\}$. From the commutative diagram
\[
\begin{tikzcd}
P^{-1}\ZZ[h]\{s\} \ar[r, "i|"] \ar[rd, equal] & H^0(\wt{D})\otimes P^{-1}\ZZ \ar[r] \ar[d, "r|"] & H^0(\wt{D})\otimes \QQ \ar[d, "r"] \\
 & P^{-1}\ZZ[h]\{s\} \ar[hookrightarrow, r] & \QQ[h]\{s\}
\end{tikzcd}
\]
it follows that $P^{-1}\ZZ[h]\{s\}$ is a direct summand in $H^0(\wt{D})\otimes P^{-1}\ZZ$. (Here, $i|$ and $r|$ denote the restrictions of $i$ and $r$.) If $p$ is a prime not in $P$, there is a surjective ring homomorphism $P^{-1}\ZZ\to \FF_p$, and therefore $\FF_p[h]\{s\}$ is a direct summand in $H^0(\wt{D}\otimes\FF_p)$. It follows that $s_{\FF_p}(\wt{C},\wt{D},f) = s$ for all but finitely many primes.

  To see that $\vec{s}$ is surjective, we will now show that $2$ times each standard basis vector is in the
  image. The vector $(2,0,0,\dots)$ is the image of
  $(\wt{C},\wt{D},f)$ where $\wt{D}=\ZZ[h]\{2\}$ in homological
  grading $0$ and $\wt{C}$ and $f$ are any choice of extension to
  a reduced LEO triple (e.g., $\wt{C}=\ZZ\{2\}$ and $f$ the obvious
  identification).

  For the other basis vectors, let $p$ be a prime. Consider the
  complex $\wt{D}$ with $\wt{D}^{0}=\ZZ[h]\oplus\ZZ[h]\{2\}$ and
  $\wt{D}^1=\ZZ[h]\{2\}$, and $\bdy\co \wt{D}^{0}\to \wt{D}^1$ given
  by $\bdy(x,y)=hx+py$. Extend $\wt{D}$ to a reduced LEO triple
  $(\wt{C},\wt{D},f)$ arbitrarily. Then
  \[
      s_{\FF}(\wt{C},\wt{D},f)=
      \begin{cases}
        2 & \mathrm{char}(\FF)=p\\
        0 & \text{otherwise,}
      \end{cases}
    \]
    as desired.
\end{proof}

In particular, Proposition~\ref{prop:Zinf-summand} implies that
$\rCLEO$ has a $\ZZ^\infty$ summand, and since $\CLEO\onto \rCLEO$,
the group $\CLEO$ does as well. (We have not shown these complexes are
induced by knots, so this does not re-prove the existence of a
$\ZZ^\infty$ summand of $\mathcal{C}$, but it is conjectured to be
so~\cite[Conjecture 1.3]{Schutz:integral-s}, \cite[Question
6.5]{LewarkZibrowius:doubles}.)  The definition of $s_{\FF_2}(C,D,f)$
extends to two-reduced LEO triples, giving a $\ZZ$-summand of
$\rCLEO^o$. Nontriviality of the invariants introduced in
Section~\ref{sec:refined-s} implies that the map
$\vec{s}\co \rCLEO\to(2\ZZ)^\infty$ is not an isomorphism, nor is
$s_{\FF_2}\co \rCLEO^o\to\ZZ$ (cf.\ Remark~\ref{remark:nothing}).

\begin{question}\label{question:identify-group}
  Is it possible to completely determine $\rCLEO^o$? What about
  $\rCLEO$?  Also, what is the image of the smooth concordance group
  $\mathcal{C}$ in either $\rCLEO^o$ or $\rCLEO$?
\end{question}

\subsection{The graded integral \texorpdfstring{$s$}{s}-invariant}
\label{subsec:graded-s}

There is an integral version of the $s$-invariant, which is not a
homomorphism but which gives a lower bound on $s_{\FF}$ for any field
$\FF$.  As we do not use these invariants until
Section~\ref{sec:knot-case}, the reader may wish to initially
skip ahead to Section~\ref{sec:refined-s}.  Specifically,
given a $\LEO$-triple $(C,D,f)$, let
\begin{align*}
s^+_\ZZ(C,D,f) &= \max\{q \mid \exists\thinspace a\in H^{0,q}(D;\ZZ)\mbox{ with }i(a)\in H^{0,q}(h^{-1}D;\ZZ)\mbox{ primitive}\}-1,\\
s^-_\ZZ(C,D,f) &= \max\{q \mid H^{0,q}(D;\ZZ)\to H^{0,q}(h^{-1}D;\ZZ)\mbox{ surjective}\}+1.
\end{align*}
For a reduced $\LEO$-triple $(\wt{C},\wt{D},f)$, we also set
\begin{equation}\label{eq:def-sZ}
s_\ZZ(\wt{C},\wt{D},f) = \max\{q\mid H^{0,q}(\wt{D};\ZZ)\to H^{0,q}(h^{-1}\wt{D};\ZZ)\mbox{ surjective}\}.
\end{equation}
For a $\LEO$-triple $(C,D,f)$, we set $s_\ZZ(C,D,f)$ to be $s_\ZZ$ of
its image in $\rCLEO$. If $K$ is a knot, we write $s_\ZZ(K)$ for
$s_\ZZ(\LEO(K))$, and similarly for the $\pm$ decorations. We will
show in Lemma~\ref{lem:allszettsame} that the three numbers
$s^+_\ZZ(K)$, $s^-_\ZZ(K)$, and $s_\ZZ(K)$ agree.

To put these constructions in context, recall that
Sch\"utz \cite{Schutz:integral-s} defined an $s$-invariant over $\ZZ$ as a finite sequence of numbers which encodes the $E_\infty$-term of the reduced Bar-Natan--Lee--Turner spectral sequence of a knot $K$. The definition can be extended to reduced $\LEO$ triples $(C,D,f)$ as follows. Consider the commutative diagram
\[
\begin{tikzcd}
H^{0,q+2}(D) \ar[r] \ar[d, "h"] & H^{0,q+2}(h^{-1}D) \ar[d, "h", "\cong"']  \ar[r,"\cong"]& \ZZ \ar[d, "\id"] \\
H^{0,q}(D) \ar[r] & H^{0,q}(h^{-1}D) \ar[r,"\cong"] & \ZZ
\end{tikzcd}
\]
and define
\begin{align*}
H^0(h^{-1}D)_q &= \im\big(H^{0,q}(D)\to H^{0,q}(h^{-1}D)\big) \\
H^0(h^{-1}D)^{(q)} &= H^0\big(h^{-1}D)_q/h(H^0(h^{-1}D)_{q+2}\big).
\end{align*}
Since $H^{0,q}(h^{-1}D)\cong \ZZ$, both $H^0(h^{-1}D)_q$ and $H^0(h^{-1}D)^{(q)}$ are cyclic groups. Furthermore, for large $q$, $H^0(h^{-1}D)_q = 0$, and for small $q$, $H^0(h^{-1}D)_q = H^{0,q}(h^{-1}D)$. In fact,
\begin{align*}
s_\QQ(C,D,f) &= \max\{q\in 2\ZZ\mid H^0(h^{-1}D)^{(q)}\not=0\}\\
s_\ZZ(C,D,f) &= \min\{q\in 2\ZZ\mid H^0(h^{-1}D)^{(q)}\not= 0\},
\end{align*}
as can be seen straight from the definitions
(Definition~\ref{def:s-of-triple} and Formula~\eqref{eq:def-sZ}). We now define the {\em graded length of } $(C,D,f)$ as
\[
\mathrm{gl}(C,D,f) = \frac{s_\QQ(C,D,f) - s_\ZZ(C,D,f)}{2}.
\]

\begin{lemma}
Let $(C,D,f)$ be a reduced $\LEO$ triple, and $q$ an even integer. The cyclic groups $H^0(h^{-1}D)_q$ and $H^0(h^{-1}D)^{(q)}$ only depend on the local equivalence class of $(C,D,f)$.
\end{lemma}

\begin{proof}
Let $(\alpha,\beta)$ be a local map from $(C,D,f)$ to $(C',D',f')$. From the commutative diagram
\[
\begin{tikzcd}
H^{0,q}(D) \ar[r,"\beta"] \ar[d] & H^{0,q}(D') \ar[d] \\
H^{0,q}(h^{-1}D) \ar[r, "\beta", "\cong"'] & H^{0,q}(h^{-1}D')
\end{tikzcd}
\]
we get $\beta(H^0(h^{-1}D)_q)\subset H^0(h^{-1}D)_q$. Similarly, a local map $(\alpha',\beta')$ from $(C',D',f')$ to $(C,D,f)$ implies $\beta'(H^0(h^{-1}D')_q)\subset H^0(h^{-1}D)_q$. Since $\beta$ and $\beta'$ induce isomorphisms of $\ZZ=H^{0,q}(h^{-1}D)$ we get that $H^0(h^{-1}D)_q$ and $H^0(h^{-1}D')_q$ are isomorphic.
As $\beta$ commutes with $h$, we also get $H^0(h^{-1}D)^{(q)}\cong H^0(h^{-1}D')^{(q)}$.
\end{proof} 

\begin{corollary}
Let $(C,D,f)$ be a reduced $\LEO$ triple. Then $s_\ZZ(C,D,f)$ and $\mathrm{gl}(C,D,f)$ only depend on the local equivalence class of $(C,D,f)$.
\end{corollary}

The cyclic groups $H^0(h^{-1}D)^{(q)}$ can also be expressed in terms of the complex $D_{h=1}$. To see this, note that there is a filtration
\[
\cdots \subset \wt{\mathcal{F}}_q \subset \wt{\mathcal{F}}_{q-2} \subset \cdots\subset D_{h=1}
\]
given by
\[
\wt{\mathcal{F}}_q = p(D^{(q)})\subset D_{h=1},
\]
where $p\colon D\to D_{h=1}$ is projection and $D^{(q)}$ is as in
Definition~\ref{def:filt-notation}. Since $p$ restricted to
$D^{(q)}$ is injective,
we get $H^{0,q}(D)\cong H^0(\wt{\mathcal{F}}_q)$. Also, $p$ extends to $p\colon h^{-1}D\to D_{h=1}$, and restricting this to $h^{-1}D^{(q)}$ induces an isomorphism $h^{-1}D^{(q)}\to D_{h=1}$ for all even $q$. In particular, for a knot $K$, the integral $s$-invariant $s^\ZZ(K)$ of Sch\"utz \cite{Schutz:integral-s} can be recovered using $D=\rKhCx_h(K)$.

\begin{definition}
Let $(C,D,f)$ be a reduced $\LEO$ triple. If $s_\QQ(C,D,f) =
s_\ZZ(C,D,f)$, define $s^\ZZ(C,D,f) = s_\QQ(C,D,f)$. Otherwise, for
$q=s_\QQ(C,D,f)-2,\ldots, s_\ZZ(C,D,f)$, denote the cardinality of the
finite cyclic group $H^{0}(h^{-1}D)^{(q)}$ by
\[
c(q) = \# H^{0}(h^{-1}D)^{(q)},
\]
and define $s^\ZZ(C,D,f) \in \ZZ^{1+\mathrm{gl}(C,D,f)}$ as
\[
s^\ZZ(C,D,f) = \Big(s_\QQ(C,D,f), c\big(s_\QQ(C,D,f) - 2\big), \ldots, c\big(s_\QQ(C,D,f) - 2\mathrm{gl}(C,D,f)\big)\Big).
\]
\end{definition}

These definitions can be carried over to unreduced $\LEO$ triples,
although we do not necessarily get that the $H^{0}(h^{-1}D)^{(q)}$ are
finite cyclic. Also, at least for knots, this does not give better
information, as we will show in Remark~\ref{rem:no-better}.

\section{Invariants of LEO triples}\label{sec:refined-s}
In this section, we introduce some algebraic invariants of local
equivalence classes, in the spirit of the refinements of the
$s$-invariant using the Khovanov stable homotopy
type~\cite{LS14:refine-s}. Unlike the invariants in the previous
section, the ones defined here depend on both $C$ and $D$, not just
$D$. We give two classes of refinements, name the reduced variants of
both, and then discuss some examples.

\subsection{Bockstein-refined \texorpdfstring{$s$}{s}-invariants}
For $n$ a positive integer, we have the short exact sequence
\[
\begin{tikzcd}
0 \arrow[r] & \ZZ/(2)  \arrow[r] & \ZZ/(2^{n+1})  \arrow[r] & \ZZ/(2^n)  \arrow[r] & 0.
\end{tikzcd}
\]
For a LEO triple $(C,D,f)$, this gives rise to the Bockstein homomorphism
\[
\beta_n\colon H^k(C; \ZZ/(2^n)) \to H^{k+1}(C; \FF_2).
\]
Now consider the configurations
\begin{equation}\label{eq:Bockstein_s_invariant}
\begin{tikzcd}[column sep=1.75em]
\langle\check{a}\rangle \arrow[r] \arrow[hookrightarrow, d] &[1em] \langle\hat{a}\rangle \arrow[hookrightarrow, d] & \langle a \rangle \arrow[l] \arrow[r] \arrow[hookrightarrow, d] & \langle \overline{a} \rangle \arrow[hookrightarrow, d]\\
H^{-1,q}(C;\ZZ/(2^n)) \arrow[r, "f \circ \beta_n"] & H^{0,q}(D_{h=0}; \FF_2) & H^{0,q}(D;\FF_2) \arrow[l,
"p" swap] \arrow[r, "i"] & H^{0, q}(h^{-1}D;\FF_2) \\
\langle\check{a}, \check{b}\rangle \arrow[r] \arrow[hookrightarrow, u] & \langle \hat{a}, \hat{b}\rangle \arrow[hookrightarrow, u] & \langle a,b \rangle \arrow[l] \arrow[r] \arrow[hookrightarrow, u] & \langle \overline{a}, \overline{b}\rangle \arrow[equal, u]
\end{tikzcd}
\end{equation}
where $\langle \, \cdot \, \rangle$ refers to the $\FF_2$-vector space
generated by the elements.
\begin{definition}\label{def:halffull}
  Given a LEO triple $(C,D,f)$ and a positive integer $n$, we say an
  integer $q$ is {\em $\beta_n$-half-full} if there exist
  $\check{a}\in H^{-1,q}(C;\ZZ/(2^n))$ and $a\in H^{0,q}(D;\FF_2)$
  such that $f\circ \beta_n(\check{a}) = p(a)$ and $i(a)\not=0$.

  We say $q$ is {\em $\beta_n$-full} if there exist
  $\check{a}, \check{b}\in H^{-1,q}(C;\ZZ/(2^n))$ and
  $a, b\in H^{0,q}(D;\FF_2)$ such that
  $f\circ \beta_n(\check{a}) = p(a)$,
  $f\circ \beta_n(\check{b}) = p(b)$, and $i(a),i(b)$ span
  $H^{0, q}(h^{-1}D;\FF_2)$ as an $\FF_2$-vector space.
\end{definition}
So, $q$ being $\beta_n$-half-full means we have a commutative diagram as in the upper half of (\ref{eq:Bockstein_s_invariant}) and $q$ being $\beta_n$-full means we have a commutative diagram as in the lower half of (\ref{eq:Bockstein_s_invariant}). We allow $\check{a}$ and $\check{b}$ to be $0$ in this definition. In particular, if $q$ is $\beta_n$-half-full, then $q-2$ is half-full as well, using $ha\in H^{0,q-2}(D;\FF_2)$ and $0\in H^{-1,q}(C;\ZZ/(2^n))$.

\begin{definition}\label{def:s-beta-n}
  Given a LEO triple $(C,D,f)$ and a positive integer $n$, let
  \[
  r^{\beta_n}(C,D,f) = \max\{ q\in \ZZ \mid q \mbox{ is }\beta_n\mbox{-half-full}\}+1,
  \]
  and
  \[
  s^{\beta_n}(C,D,f) = \max\{ q \in \ZZ \mid q \mbox{ is }\beta_n\mbox{-full}\}+3.
  \]
\end{definition}

If $q > s^+(C,D,f)+1$, then $q$ cannot be $\beta_n$-half-full. Also,
if $q < s^+(C,D,f)$ we can find $a\in H^{0,q+2}(D;\FF_2)$ with
$i(a)\in H^0(h^{-1}D;\FF_2)$ non-zero.  Then $ha\in H^{0,q}(D;\FF_2)$
also satisfies $i(ha)\not=0$ but $p(ha)=0$ so $q$ is $\beta_n$-half-full. In particular,
\begin{equation}\label{eq:r-diff-by-2}
  0\leq r^{\beta_n}(C,D,f) - s^+_{\FF_2}(C,D,f) \leq 2.
\end{equation}
A similar argument shows that
\begin{equation}\label{eq:s-diff-by-2}
0\leq s^{\beta_n}(C,D,f) - s^-_{\FF_2}(C,D,f) \leq 2.
\end{equation}
For $k \leq n$, we have the commutative diagram
\[
\begin{tikzcd}
0 \arrow[r] & \ZZ/(2)  \arrow[r] \arrow[equal, d]& \ZZ/(2^{k+1})  \arrow[r] \arrow[d, "\cdot 2^{n-k}"] & \ZZ/(2^k)  \arrow[r] \arrow[d, "\cdot 2^{n-k}"] & 0 \\
0 \arrow[r] & \ZZ/(2)  \arrow[r] & \ZZ/(2^{n+1})  \arrow[r] & \ZZ/(2^n) \arrow[r] & 0
\end{tikzcd}
\]
Given a $\LEO$-triple, naturality implies that if $q$ is
$\beta_k$-(half-)full, then $q$ is $\beta_n$-(half-)full so
$r^{\beta_n}\geq r^{\beta_k}$ and $s^{\beta_n}\geq s^{\beta_k}$. Passing to the direct limit gives a commutative diagram
\[
\begin{tikzcd}
0 \arrow[r] & \ZZ/(2)  \arrow[r] \arrow[equal, d]& \ZZ/(2^{k+1})  \arrow[r] \arrow[d, "\cdot 2^{-k-1}"] & \ZZ/(2^k)  \arrow[r] \arrow[d, "\cdot 2^{-k}"] & 0 \\
0 \arrow[r] & \ZZ/(2)  \arrow[r] & \ZZ[\frac{1}{2}]/\ZZ  \arrow[r] & \ZZ[\frac{1}{2}]/\ZZ \arrow[r] & 0
\end{tikzcd}
\]
and a Bockstein homomorphism $\beta_\infty\colon H^k(C;\ZZ[\frac{1}{2}]/\ZZ)\to H^{k+1}(C;\FF_2)$ giving rise to invariants $r^{\beta_\infty}(C,D,f)$ and $s^{\beta_\infty}(C,D,f)$ with
\[
r^{\beta_\infty}(C,D,f) = \lim_{n\to\infty}r^{\beta_n}(C,D,f)\hspace{0.4cm}\mbox{ and }\hspace{0.4cm} s^{\beta_\infty}(C,D,f) = \lim_{n\to\infty} s^{\beta_n}(C,D,f).
\]
See Examples~\ref{eg:unknot-full}--\ref{eg:alg-example2} for some
examples of these invariants, as well as of the ones constructed later
in this section.

In the case $n=1$, we can use the isomorphism $f\colon H^{\ast,q}(C;\FF_2)\to H^{\ast,q}(D_{h=0};\FF_2)$ to produce another operation $\beta\colon H^{-1,q}(C;\FF_2)\to H^{0,q}(C;\FF_2)$ given by
\begin{equation}\label{eq:beta-sum}
\beta = \beta_1+f^{-1}\circ \beta_1\circ f.
\end{equation}
The notions of $\beta$-half-full and $\beta$-full carry over verbatim, and we can define $r^\beta(C,D,f)$ and $s^\beta(C,D,f)$ as before.

\begin{remark}
  For a knot $K$, the map $\beta$ from Equation~\eqref{eq:beta-sum}
  corresponds to the sum of the even and odd first Steenrod squares
  $\Sq^1+\Sq^1_o\colon \Kh^{-1,q}(K;\FF_2)\to Kh^{0,q}(K;\FF_2)$. (See
  also Lemma~\ref{lem:old-refinements}.) The sum of these Steenrod
  squares can also be realized via a connecting homomorphism involving
  unified Khovanov homology, as described by Putyra and Shumakovitch~\cite{PS16:mixed-Bocksteins}. In
  Section~\ref{sec: MP knots}, we use the associated refined
  $s$-invariant, $s^{\beta}$, to give an alternate
  proof that the five key knots from Manolescu-Piccirillo's paper~\cite{ManolescuPiccirillo2023}
  are not slice.
\end{remark}

\subsection{The comprehensive refinements}
\label{subsec:comprehensive}
Our second class of refinements use the change of
coefficients map $j\co H^{i,j}(C)\to H^{i,j}(C;\FF_2)$ instead of
a Bockstein homomorphism, and are based on the diagrams
\[
  \begin{tikzcd}
    H^{0,q}(C)\arrow[r,"f\circ j"] &
    H^{0,q}(D_{h=0}; \FF_2) &
    H^{0,q}(D; \FF_2) \arrow[l, "p" swap] \arrow[r, "i"] &
    H^{0,q}(h^{-1}D; \FF_2)
  \end{tikzcd}
\]
and
\[
  \begin{tikzcd}
    H^{0,q}(C)\arrow[r,"f\circ j"] &
    H^{0,q}(D_{h=0}; \FF_2) &
    H^{0,q}(D) \arrow[l, "p" swap] \arrow[r, "i"] &
    H^{0,q}(h^{-1}D).
  \end{tikzcd}
\]

\begin{definition}
  Given a LEO triple $(C,D,f)$, we say an integer $q$ is
  \emph{oddly half-full} if there exist $\check{a}\in H^{0,q}(C)$
  and $a\in H^{0,q}(D;\FF_2)$ such that $f\circ j(\check{a}) = p(a)$
  and $i(a)\not=0$.  We say $q$ is \emph{oddly full} if in
  addition there exist $\check{b} \in H^{0,q}(C)$ and $b\in H^{0,q}(D;\FF_2)$ 
  such that $f\circ j(\check{b}) = p(b)$
  and $i(a)$, $i(b)$ span $H^{0, q}(h^{-1}D; \FF_2)$ as an
  $\FF_2$-vector space.

  Similarly, $q$ is \emph{completely half-full} if there exist
  $\check{a}\in H^{0,q}(C)$ and $a\in H^{0,q}(D)$ such that
  $f\circ j(\check{a}) = p(a)$ and $i(a)$ is a nonzero primitive
  element of $H^{0, q}(h^{-1}D)$.  Moreover, $q$ is \emph{completely
    full} if in addition there exist and $\check{b} \in H^{0,q}(C)$
  $b\in H^{0,q}(D)$ such that $f\circ j(\check{b}) = p(b)$
  and $i(a)$, $i(b)$ generate $H^{0, q}(h^{-1}D)$ as a
  $\ZZ$-module.

For a $\LEO$-triple $(C,D,f)$ define
\begin{align*}
r_o(C,D,f) &= \max\{q\in \ZZ\mid q \mbox{ is oddly half-full}\}-1,\\
s_o(C,D,f) &= \max\{q\in \ZZ\mid q \mbox{ is oddly full}\}+1,\\
r_c(C,D,f) &= \max\{q\in \ZZ\mid q \mbox{ is completely half-full}\}-1,\\
s_c(C,D,f) &= \max\{q\in \ZZ\mid q \mbox{ is completely full}\}+1.
\end{align*}
\end{definition}
The shifts here are different from
Definition~\ref{def:s-beta-n}, and are chosen so that they all vanish
for the trivial element $\LEO(U)$; see Example~\ref{eg:unknot-full}.
The invariants $r_o$ and $s_o$ satisfy
\begin{align}
0 &\leq s^+_{\FF_2}(C,D,f) - r_o(C,D,f) \leq 2\label{eq:s-versus_ru}\\
0 &\leq s^-_{\FF_2}(C,D,f) - s_o(C,D,f) \leq 2.\label{eq:s-versus_su}
\end{align}
The situation for the complete invariants is more complicated; see
Theorem~\ref{thm:s-rels} in the case of knots.

\subsection{The reduced case}
All of the local equivalence invariants defined here have reduced analogues
\[
\stil^\alpha(\wt{C},\wt{D},f) = \max\{q\in 2\ZZ\mid q \mbox{ is $\alpha$-reduced-full}\}+2
\]
where $\alpha$ is any of the $\beta_n$ or $\beta$, and $\alpha$-reduced-full is defined as $\alpha$-half-full, but with $\wt{C},\wt{D}$ in place of $C,D$. Also,
\[
\stil_o(\wt{C},\wt{D},f) = \max\{q\in 2\ZZ\mid q \mbox{ is
  oddly reduced-full}\},
\]
which satisfies
$0 \leq s_{\FF_2}(\wt{C}, \wt{D}, f) - \stil_o(\wt{C},\wt{D}, f)
\leq 2$.  Define $\stil_c$ similarly, where to be completely
reduced-full we require that $i(a)$ be a generator for
$H^{0, q}(h^{-1}\wt{D})$ as a $\ZZ$-module.  Given an unreduced LEO
triple $(C,D,f)$, we will write $\stil^\alpha(C,D,f)$ or
$\stil_\gamma(C,D,f)$ to denote the invariant of its image
$\pi(C,D,f)$ in $\rCLEO$.

We will sometimes refer to the invariants $r^\alpha$, $s^\alpha$, and $\stil^\alpha$ as the \emph{Bockstein refinements} and $r_\gamma$, $s_\gamma$, and $\stil_\gamma$ as the \emph{comprehensive refinements}. (The latter terminology is justified in Theorems~\ref{thm:reduced-local-is-e},~\ref{thm:s-rels}, and~\ref{thm:local-is-e}.)

\subsection{Invariance and examples} Next, we establish the invariance
under local equivalence of the various numbers defined in the first
part of this section.

\begin{lemma}\label{lem:s-beta-descends}
  The integers $r^{\beta_n}$, $s^{\beta_n}$, $r^{\beta}$, $s^{\beta}$,
  $r_o$, $s_o$, $r_c$, and $s_c$ are invariant under local equivalence
  of $(C,D,f)$, i.e., descend to $\CLEO$; their reduced versions
  $\stil^{\beta_n}$, $\stil^{\beta}$, $\stil_o$, and $\stil_c$ descend
  to $\rCLEO$. Further, $\stil_o$ and $\stil^{\beta_n}$ descend to
  $\rCLEO^o$.
\end{lemma}
\begin{proof}
  The proofs are all essentially the same, so we only spell out a
  representative case, $r^{\beta_n}$. Suppose $(h\co C\to C',k\co D\to D')$ is a local map
  from $(C,D,f)$ to $(C',D',f')$. Then, we have a commutative diagram
  \[
    \begin{tikzcd}[column sep=1.5em]
      H^{-1,q}(C;\ZZ/(2^n)) \arrow[r, "f \circ \beta_n"]\arrow[d,"h_*"] &[1.2em] H^{0,q}(D_{h=0}; \FF_2)\arrow[d,"k_*"] & H^{0,q}(D;\FF_2) \arrow[l, "p" swap] \arrow[r, "i"]\arrow[d,"k_*"] & H^0(h^{-1}D;\FF_2)\arrow{d}{k_*}[swap]{\cong} \\
      H^{-1,q}(C';\ZZ/(2^n)) \arrow[r, "f \circ \beta_n"] & H^{0,q}(D'_{h=0}; \FF_2) & H^{0,q}(D';\FF_2) \arrow[l, "p" swap] \arrow[r, "i"] & H^0(h^{-1}D';\FF_2),
    \end{tikzcd}
  \]
  where the arrow on the right is an isomorphism.  So, if $q$ is
  $\beta_n$-half-full for $(C,D,f)$, witnessed by elements
  $\check{a}$, $\hat{a}$, $a$, and $\overline{a}$, then
  $q$ is also $\beta_n$-half-full for $(C',D',f')$, witnessed by
  $h_*(\check{a})$, $k_*(\hat{a})$, $k_*(a)$, and
  $k_*(\overline{a})$. Hence,
  $r^{\beta_n}(C',D',f')\geq r^{\beta_n}(C,D,f)$. 
  If $(C,D,f)$ and
  $(C',D',f')$ are locally equivalent, there is also a local map the
  other direction, giving the reverse inequality.
\end{proof}

We illustrate the definitions above with a few simple algebraic
examples.

\begin{example}\label{eg:unknot-full}
  Consider the trivial LEO triple
  $\LEO(U)=(\ZZ[X]/(X^2)\{1\},\BNring\{1\},\id)$. We write down the
  sequence in Diagram~\eqref{eq:Bockstein_s_invariant} for $q=-3$,
  $-1$, $1$, and $3$. For $q=-3$, it has the form
  \[
    \begin{tikzcd}[scale cd=1.0]
      H^{-1,-3}(C;\ZZ/(2^n))\arrow[r,"\id \circ \beta_n"]\arrow[equal,d] &[1em]
         H^{0,-3}(D_{h=0}; \FF_2)\arrow[equal,d] &
         H^{0,-3}(D;\FF_2) \arrow[l, "p" swap] \arrow[r, "i"]\arrow[equal,d] &
         H^{0, -3}(h^{-1}D;\FF_2)\arrow[equal,d]\\
      0 \arrow[r] & 0 &
         \FF_2\langle h^2, Xh \rangle\arrow[l]\arrow[r] &  \FF_2\langle h^2, Xh \rangle.
    \end{tikzcd}
  \]
  In particular, $-3$ is $\beta_n$-full, with
  $\check{a}=\check{b}=\hat{a}=\hat{b}=0$, $a=\overline{a}=h^2$, and
  $b=\overline{b}=Xh$. The same applies to $\beta$ (which also
  vanishes); indeed, $\beta$ will be similar to $\beta_n$ throughout
  this example. For $r_c$ and $s_c$, we are interested instead in the
  diagram
  \[
    \begin{tikzcd}
      H^{0,-3}(C)\arrow[r,"f\circ j"]\arrow[equal,d] &
         H^{0,-3}(D_{h=0}; \FF_2)\arrow[equal,d] &
         H^{0,-3}(D) \arrow[l, "p" swap] \arrow[r, "i"]\arrow[equal,d] &
         H^{0,-3}(h^{-1}D)\arrow[equal,d]\\
      0 \arrow[r] & 0 & \ZZ\langle h^2, Xh \rangle\arrow[l]\arrow[r] &
         \ZZ\langle h^2, Xh \rangle.
    \end{tikzcd}
  \]
  So $q=-3$ is completely full, by taking $\check{a}=\check{b}=0$ and
  $a=h^2$, $b=Xh$, and the same for oddly full.

  For $q=-1$, the diagram for $s^{\beta_n}$ is
  \[
    \begin{tikzcd}
      0\arrow[r] & \FF_2\langle X\rangle &[1em]
      \FF_2\langle h,X\rangle \arrow[l, "0 \, \reflectbox{$\mapsto$} \, h",swap] \arrow[r] &
      \FF_2\langle h,X\rangle,
    \end{tikzcd}
  \]
  so $q=-1$ is not full for $\beta_n$ (because we cannot find
  $\check{b}$, say) but is half-full via
  $\check{a}=\hat{a}=0$ and $a=\overline{a}=h$. So,
  $s^{\beta_n}=-3+3=0$. For the complete case, the sequence is
  \[
    \begin{tikzcd}
      \ZZ\langle X\rangle\arrow[r] & \FF_2\langle X\rangle &[1em]
      \ZZ\langle h,X\rangle \arrow[l, "0 \, \reflectbox{$\mapsto$} \, h",swap] \arrow[r] &
      \ZZ\langle h,X\rangle,
    \end{tikzcd}
  \]
  and we can take $\check{a}=0$, $a=h$, and $\check{b}=b=X$ to see
  that $q = -1$ is completely and oddly full.

  For $q=1$, the diagram for $s^{\beta_n}$ is
  \[
    \begin{tikzcd}
      0\arrow[r] & \FF_2\langle 1\rangle &
      \FF_2\langle 1\rangle \arrow[l, "\id",swap] \arrow[r] &
      \FF_2\langle 1, h^{-1}X\rangle,
    \end{tikzcd}
  \]
  so $1$ is not $\beta_n$-half-full (because one cannot find
  $\check{a}$), so $r^{\beta_n}=-1+1=0$. For the complete case, the
  sequence is
  \[
    \begin{tikzcd}
      \ZZ\langle 1\rangle\arrow[r] & \FF_2\langle 1\rangle &
      \ZZ \langle 1\rangle\arrow[l]\arrow[r] & \ZZ\langle 1, h^{-1} X\rangle
    \end{tikzcd}
  \]
  so $1$ is completely half-full but not completely full. So $s_c =
  -1+1 = 0$, and the same for $s_o$.

  For $q=3$, the diagram for $s^{\beta_n}$ is 
  \[
    \begin{tikzcd}
      0 \arrow[r] & 0 & 0\arrow[l] \arrow[r] &
      \FF_2\langle h^{-1}, h^{-2} X\rangle,
    \end{tikzcd}
  \]
  so again $3$ is not $\beta_n$-half-full (as must be the case), while
  for the complete case the sequence is similarly
  \[
    \begin{tikzcd}
      0\arrow[r] & 0 &
      0\arrow[l]\arrow[r] &  \ZZ\langle h^{-1}, h^{-2} X\rangle,
    \end{tikzcd}
  \]
  so $3$ is not completely half-full, and $r_c= r_o = 1-1=0$.
\end{example}

\begin{example}\label{eg:alg-example1}
  For this example, we look at a reduced LEO triple.  Consider the
  two-step complex $C$ over $\ZZ$ given by
  $C^{-1}=\ZZ$, $C^0=\ZZ\{2\}\oplus\ZZ$, and
  $\bdy(1)=(0,2^n)$; and the two-step complex $D$ over
  $\ZZ[h]$ given by $D^{-1}=\ZZ[h]$,
  $D^0=\ZZ[h]\{2\}\oplus\ZZ[h]$, and $\bdy(1)=(h,0)$. There is
  an evident identification of $C\otimes\FF_2$ and
  $D\otimes\FF_2$, given in grading $0$ by the $2\times2$
  identity matrix. That is:
  \[
    C:
    \begin{tikzcd}[column sep=0pt]
      \ZZ\{2\} & & \ZZ\\
      & \ZZ \arrow[ur,"2^n"] & 
    \end{tikzcd}
    \qquad
    C/2=D/(2,h):
    \begin{tikzcd}[column sep=0pt]
      \FF_2\{2\} & & \FF_2\\
      & \FF_2 &
    \end{tikzcd}
    \qquad D:
    \begin{tikzcd}[column sep=0pt]
      \ZZ[h]\{2\} & & \ZZ[h]\\
      & \ZZ[h] \arrow[ul,"h"] & 
    \end{tikzcd}
  \]
  Then $s_{\FF_2}(C,D,f)=0$. The Bockstein $\beta_m$
  vanishes for $m<n$, so $\stil^{\beta_m}(C,D,f)=0$,
  but for $p\geq n$, $0$ is $\beta_p$-reduced-full so
  $\stil^{\beta_p}=2$. For this complex $\beta=\beta_1$ (since
  $D/(h)$ has no torsion) and
  $\stil^\beta=\stil^{\beta_1}$. Also, $0$ is oddly and completely
  reduced-full, so $\stil_o=\stil_c=0$.

  If we dualize the complexes just described, we have
  $\stil^{\beta_p}(C^*,D^*,f^*)=s_{\FF_2}=0$ for all
  $p$. However, for the dual complex, $0$ is no longer oddly or
  completely reduced-full, so $\stil_o=\stil_c=-2$.

  Replacing each copy of $\ZZ$ in $C$ by $\ZZ[X]/(X^2)$ and
  each copy of $\ZZ[h]$ in $D$ by $\BNring$ gives an unreduced
  LEO triple $(C,D,f)$ with examples with
  $s^{\beta_p}(C,D,f)=r^{\beta_p}(C,D,f)=\stil^{\beta_p}(C,D,f)$
  and similarly for the other refined $s$-invariants.
\end{example}

\begin{example}\label{eg:alg-example2}  
  As another example, consider the reduced LEO triple given by
  \[
    C:\quad
    \begin{tikzcd}[column sep=0pt]
      & \ZZ\{2\} & \\
      \ZZ & \ZZ & \ZZ\{2\}\\
      & \ZZ\arrow[ul,"2"]\arrow[u,swap,"2"]
    \end{tikzcd}\qquad\qquad
    D:\quad
    \begin{tikzcd}[column sep=0pt]
      & \ZZ[h]\{2\} & \\
      \ZZ[h]\arrow[ur,"h"] & \ZZ[h] & \ZZ[h]\{2\}\arrow[ul,swap,"-2"]\\
      & \ZZ[h]\arrow[ul,"2"]\arrow[ur,swap,"h"]      
    \end{tikzcd}
  \]
  with the obvious identification of $C/(2)$ and
  $D/(2,h)$. That is, $C^{-1}=\ZZ$,
  $C^{0}=\ZZ\oplus\ZZ\oplus(\ZZ\{2\})$, and
  $C^{1}=\ZZ\{2\}$, with the only nontrivial differential the
  one from $C^{-1}$ to $C^{0}$ given by
  $\bdy(1)=(2,2,0)$, and $D$ is obtained from $C$ by
  replacing each $\ZZ$ with $\ZZ[h]$, but has differential
  $D^{-1}\to D^0$ given by $\bdy(1)=(2,0,h)$ and
  $D^0\to D^1$ given by $\bdy(a,b,c)=ha-2c$.

  The Bockstein maps on the cohomology of $D/(2,h)$ are given
  below, where the \textcolor{mediumblue}{dashed} arrows are $\beta_1$ and
  the \textcolor{darkred}{dotted} arrows are $f^{-1}\circ\beta_1\circ f$:
  \[
    \begin{tikzcd}[column sep=0pt,
      beta arrow/.style={dashed, mediumblue},
      conj arrow/.style={densely dotted, darkred}]
      & \FF_2\{2\} & \\
      \FF_2 & \FF_2 & \FF_2\{2\}\arrow[ul, beta arrow]\\
      & \FF_2\arrow[ul, beta arrow, bend left] 
             \arrow[ul, conj arrow, bend right] 
             \arrow[u, conj arrow]
    \end{tikzcd}.
  \]
  It follows that $0$ is not reduced-full for the Bocksteins $\beta_n$
  induced by $C$, or for the Bocksteins induced by $D$, but
  is for the sum $\beta$ from Equation~\eqref{eq:beta-sum}. Thus,
  $\stil^{\beta_n}(C,D,f)=s_{\FF_2}(C,D,f)=0$ while
  $\stil^\beta(C,D,f)=2$. It is also easy to check that
  $\stil_o(C,D,f)=\stil_c(C,D,f)=0$.

  If we dualize the complexes, to 
  \[
    C^*:\quad
    \begin{tikzcd}[column sep=0pt]
      & \ZZ & \\
      \ZZ\arrow[ur,"2"] & \ZZ\arrow[u,swap,"2"] & \ZZ\{-2\}\\
      & \ZZ\{2\}
    \end{tikzcd}\qquad\qquad
    D^*:\quad
    \begin{tikzcd}[column sep=0pt]
      & \ZZ[h]\\
      \ZZ[h]\arrow[ur,"2"] & \ZZ[h] & \ZZ[h]\{-2\}\arrow[ul,swap,"h"]\\
      & \ZZ[h]\{-2\}\arrow[ur,swap,"-2"]\arrow[ul,"h"] & 
    \end{tikzcd},
  \]
  we now have
  $\stil^{\beta_n}(C,D,f)=\stil^\beta(C,D,f)=\stil_o(C,D,f)=0$
  but $\stil_c(C,D,f)=-2$: any generator in bigrading
  $(0,0)$ mapping to a generator of $H^{0, 0}(h^{-1}D^*)$
  must have the form $(2a,b,0)$ with $b$ odd, but the mod-$(2,h)$
  reduction of such a class is not in the image of the homology of
  $C^*$.
\end{example}

\section{Structural results on reduced LEO groups}\label{sec:LEO-structure}
In this section, we show that the invariant $\stil_c$, applied to a
reduced LEO triple and its dual, detects the trivial local equivalence
class, and $\stil_o$ does the same for two-reduced LEO triples. In the
two-reduced case, a little more work shows that the local equivalence
group $\rCLEO^o$ is totally ordered.

\begin{theorem}\label{thm:reduced-local-is-e}
  A reduced LEO triple $(C,D,f)$ is trivial in $\rCLEO$ if and only if
  $\stil_c(C,D,f)=\stil_c(C^*,D^*,f^*)=0$.  Similarly, a two-reduced
  LEO triple $(C,D,f)$ is trivial in $\rCLEO^o$ if and only if
  $\stil_o(C,D,f)=\stil_o(C^*,D^*,f^*)=0$.
\end{theorem}
Consequently,
$(C,D,f)$ is locally equivalent to $(C',D',f)$
if and only if
\[
  \stil_c\bigl((C,D,f)\otimes(C',D',f')^*\bigr)=\stil_c\bigl((C,D,f)^*\otimes(C',D',f')\bigr)=0,
\]
and similarly for the two-reduced case.

\begin{proof}
  We prove the statement about $\rCLEO$; the statement about
  $\rCLEO^o$ is similar.
  The ``only if'' direction is Lemma~\ref{lem:s-beta-descends}. So,
  assume that
  $\stil_c(C,D,f)=\stil_c(C^*,D^*,f^*)=0$. Suppose
  that $[\check{a}]$ and $[a]$ witness the fact that $0$ is completely reduced-full, so $\check{a}\in C$ and $a\in D$ are
  cocycle representatives in quantum grading $0$. Let $\wt{a}$ denote
  the image of $\check{a}$ in $C$. Define $\alpha\co \ZZ\to C$ by
  $\alpha(1)=\wt{a}$ and $\beta\co \ZZ[h]\to D$ by
  $\beta(1)=a$. Since the image of $[a]$ in $H^{0,0}(h^{-1}D)$
  is a generator, $\beta$ induces a quasi-isomorphism
  $\ZZ[h,h^{-1}]\to h^{-1}D$; since $h^{-1}D$ is a complex of
  free modules over $\ZZ[h,h^{-1}]$, this is the same as a homotopy
  equivalence. The images $[p(a)]$ of $[a]$
  and $[f(\wt{a})]$ of $[\check{a}]$ in $H(D_{h=0};\FF_2)$
  agree, so there is an element $c\in D_{h=0}\otimes\FF_2$ so that
  $\bdy(c)=f(\wt{a})-p(a)$. Define a homotopy
  $\gamma\co \FF_2\to D_{h=0}\otimes\FF_2$ by $\gamma(1)=c$. Then
  $\gamma$ is a homotopy making Diagram~\eqref{eq:reduced-local-map}
  homotopy commute. In particular, $(\alpha,\beta)$ is a local map
  from $(\ZZ,\ZZ[h],\id)$ to $(C,D,f)$.

  To construct a local map the other direction, applying the
  discussion in the previous paragraph to $(C^*,D^*,f^*)$
  gives a local map
  $(\gamma,\delta)\co (\ZZ,\ZZ[h],\id)\to
  (C^*,D^*,f^*)$. The transpose of $\gamma$ (respectively
  $\delta$) is a map $\gamma^T\co C\to \ZZ$ (respectively
  $\delta^T \co D^*\to \ZZ[h]$), and Diagram~\eqref{eq:reduced-local-map} for
  $(\gamma,\delta)$ induces a corresponding diagram for
  $(\gamma^T,\delta^T)$ after replacing the homotopy equivalences with
  their inverses. Finally, the fact that the map
  $\delta\co \ZZ[h,h^{-1}]\to h^{-1}D^*$ is a quasi-isomorphism
  implies that the transpose to $\delta$ is also a
  quasi-isomorphism. So, $(\gamma^T,\delta^T)$ is the desired local
  map $(C,D,f)\to (\ZZ,\ZZ[h],\id)$.
\end{proof}

\begin{corollary}\label{cor:stilc-triv}
  Given a reduced LEO triple $(C,D,f)$,
  $\stil_c(C^*,D^*,f^*)=-\stil_c(C,D,f)$ if and only if $(C,D,f)$ is
  locally equivalent to $(\ZZ\{q\},\ZZ[h]\{q\},\id)$, where
  $q=\stil_c(C,D,f)$.
\end{corollary}
\begin{proof}
  The ``if'' direction is trivial; for the other, suppose
  $\stil_c(C^*,D^*,f^*)=-\stil_c(C,D,f)=-q$.
  Tensoring $(C,D,f)$ with $(\ZZ\{q\},\ZZ[h]\{q\},\id)^*$
  simply has the effect of shifting the gradings, so the result and
  its dual have $\stil_c=0$.  Applying
  Theorem~\ref{thm:reduced-local-is-e} and then tensoring by
  $(\ZZ\{q\},\ZZ[h]\{q\},\id)$ gives the result.
\end{proof}

\begin{remark}\label{rem:sc-not-enough}
  For the unreduced case, it is not true that
  \[
    r_c(C,D,f)=s_c(C,D,f)=r_c((C,D,f)^*)=s_c((C,D,f)^*)=0
  \]
  implies that $(C,D,f)$ is locally equivalent to $\LEO(U)$. Let $D$
  be the complex from Example~\ref{eg:pathology} with $k=2$ and let
  $C=D_{h=0}$.  Consider the triple $(C\{-2\},D\{-2\},\id)$ and its
  dual. Since $C$ is simply the modulo-$h$ reduction of $D$, $q$ is
  completely half-full if and only if there is an element of
  $H(D\{-2\})$ of grading $q$ which is not $h$-torsion (and similarly
  for completely full). So, the computation in
  Example~\ref{eg:pathology} shows that $1$ is completely half-full,
  and $-1$ is completely full, so
  $r_c(C\{-2\},D\{-2\},\id)=s_c(C\{-2\},D\{-2\},\id)=0$. On the other
  hand, the dual complex to $D$, with grading shifted up by $2$ and
  viewed as a complex over $\ZZ[h]$, is given by
  \[
    \begin{tikzcd}
      5 & \ZZ[h]\langle a\rangle & & \\
      3 & \ZZ[h]\langle Xa\rangle & \ZZ[h]\langle b\rangle\arrow[l]\arrow[ul,"h",swap] & \\
      1 & &\ZZ[h]\langle Xb\rangle & \ZZ[h]\langle c\rangle\arrow[uull,"h^2", bend right,swap]\\
      -1 & & & \ZZ[h]\langle Xc\rangle\arrow[uull,"h^2", bend left]
    \end{tikzcd}
  \]
  where $a,b,c$ are names for the generators of the copies of
  $\BNring$ and the numbers indicate the quantum gradings. Canceling
  $Xa$ and $b$ gives that
  \[r_c\big((C\{-2\},D\{-2\},\id)^*\big)=s_c\big((C\{-2\},D\{-2\},\id)^*\big)=0,\]
  where the relevant generators are $Xb$ and $Xc$.

  There is no local map from $\LEO(U)$ to $(C\{-2\},D\{-2\},\id)$.  A
  generator of $\BNring\{1\}$ in $\LEO(U)$ would have to map to
  $(0,a)\in D^0=\BNring\{-1\}\oplus\BNring\{1\}$ for some $a\in\ZZ$,
  but $Xh^2(0,a)=\bdy(X)$ (where $X\in D^{-1}=\BNring\{-3\}$), so
  $(0,a)$ does not generate the homology of $h^{-1}D$.

  A similar short computation shows that there is also no local map
  from $(C\{-2\},D\{-2\},\id)$ to $\LEO(U)$.  This example turns out
  to show that the strategy used in Theorem~\ref{thm:o-total-order}
  below to give a total order on $\rCLEO^o$ does not work for $\CLEO$.
\end{remark}

Recall the variant $\rCLEO^{o}$ of $\rCLEO$ from
Definition~\ref{def:mod-2-local}, which is defined like $\rCLEO$ but
where $D$ is a complex over $\FF_2[h]$ rather than $\ZZ[h]$.
We show in Theorem~\ref{thm:o-total-order} that $\rCLEO^o$ is a
totally ordered group. The main ingredient in the proof is the
following:

\begin{lemma}\label{lem:o-mirror}
  Let $(C,D,f)$ be a reduced LEO triple, with dual
  $(C^*,D^*,f^*)$.  If
  $\stil_o(C,D,f)=s_{\FF_2}(C,D,f)-2$ then
  $\stil_o(C^*,D^*,f^*)=s_{\FF_2}(C^*,D^*,f^*)$.
\end{lemma}
\begin{proof}
  Since $\ZZ$ and $\FF_2[h]$ are PIDs, $C$ is homotopy equivalent
  to a direct sum of bigrading-shifted copies of $\ZZ$,
  $\ZZ\stackrel{2^n}{\longrightarrow}\ZZ$, and
  $\ZZ\stackrel{p^n}{\longrightarrow}\ZZ$ ($p>2$ prime), while
  $D\otimes_\ZZ\FF_2$ is homotopy equivalent to the direct sum of
  a copy of $\FF_2[h]$ in bigrading $(0,q=s_{\FF_2}(C,D,f))$
  and some bigrading-shifted copies of
  $\FF_2[h]\stackrel{h^n}{\longrightarrow}\FF_2[h]$.  Since homotopy equivalences
  induce local equivalences, we may assume that $C$ and $D$
  have this form. Moreover, since $\stil_o$ depends on $C$ only
  through the image of the map $H(C)\to H(C;\FF_2)$, we may
  assume that $C$ has no summands of the form
  $\ZZ\stackrel{p^n}{\longrightarrow}\ZZ$ for $p>2$.

  A copy of $\ZZ\stackrel{2^n}{\longrightarrow}\ZZ$ where the first
  $\ZZ$ has bigrading $(i,j)$ in $C$ contributes copies of $\FF_2$
  to $H(C)$ in bigradings $(i,j)$ and $(i+1,j)$, with the second
  copy in the image of the map $H(C)\to
  H(C;\FF_2)$. Similarly, a copy of
  $\FF_2[h]\stackrel{h^n}{\lra}\FF_2[h]$ where the first $\FF_2[h]$ is in
  bigrading $(i,j)$ contributes copies of $\FF_2$ to
  $H(D_{h=0})$ in bigradings $(i,j)$ and $(i+1,j+2n)$, where the
  second copy is in the image of
  $H(D;\FF_2)\to H(D_{h=0};\FF_2)$.

  Suppose that $H^{0,q}(C\otimes\FF_2)=\FF_2^{c+d+e}$ where
  $\FF_2^c\oplus0^d\oplus0^e$ comes from summands of $\ZZ$ in
  $C^{0,q}$, $0^c\oplus\FF_2^d\oplus0^e$ comes from summands of
  the form $\ZZ\stackrel{2^n}{\longrightarrow}\ZZ$ with the second
  $\ZZ$ in bigrading $(0,q)$ (and, so, is in the image of the map
  $H(C)\to H(C;\FF_2)$), and $0^c\oplus0^d\oplus\FF_2^e$
  comes from summands of the form
  $\ZZ\stackrel{2^n}{\longrightarrow}\ZZ$ with the first $\ZZ$ in
  bigrading $(0,q)$ (and, so, is not in the image of the map
  $H(C)\to H(C;\FF_2)$). Suppose
  $H^{0,q}(D_{h=0};\FF_2)=\FF_2^{1+a+b}$ where the first $\FF_2$
  comes from the $\FF_2[h]$-summand, $0\oplus\FF_2^a\oplus0^b$ comes
  from summands $\FF_2[h]\stackrel{h^n}{\lra}\FF_2[h]$ where the second
  $\FF_2[h]$ is generated in degree $(0,q)$ (hence, is in the image
  of $H(D;\FF_2)\to H(D_{h=0};\FF_2)$), and
  $0\oplus0^a\oplus\FF_2^b$ comes from summands
  $\FF_2[h]\stackrel{h^n}{\lra}\FF_2[h]$ where the first $\FF_2[h]$ is
  generated in degree $(0,q)$ (hence, is not in the image of
  $H(D;\FF_2)\to H(D_{h=0};\FF_2)$). Here, $c+d+e=1+a+b$.

  The hypothesis that
  $\stil_o(C,D,f)=s_{\FF_2}(C,D,f)-2$ means that
  $q$ is not oddly reduced-full. So,
  $f(\{(x_1,\dots,x_c,y_1,\dots,y_d,0,\dots,0)\})$ does not intersect
  $\{(1,u_1,\dots,u_a,0,\dots,0)\}$. Let
  $E=\{(0,\dots,0,0,\dots,0,z_1,\dots,z_e)\}\subset
  H^{0,q}(C;\FF_2)$, and let
  $\pi_E\co H^{0,q}(C;\FF_2)\to E$ be the projection induced by
  our chosen basis. Let
  $A=\{(0,u_1,\dots,u_a,0,\dots,0)\}\subset H^{0,q}(D_{h=0};\FF_2)$.
  So, $\pi_E(f^{-1}(1,0,\dots,0,0,\dots,0))\not\in \pi_E(f^{-1}(A))$.

  Let $\alpha_E$ be a linear functional on $E$ such that
  $\alpha_E(f^{-1}(1,0,\dots,0,0,\dots,0))=1$ and
  $\alpha_E(\pi_E(f^{-1}(A)))=0$, and let $\alpha=\alpha_E\circ\pi_E$;
  $\alpha$ is a class in $H^{0,-q}(C^*;\FF_2)$.  Since $\alpha$
  vanishes on $\{(0,\dots,0,y_1,\dots,y_d,0,\dots,0)\}$, $\alpha$ is
  in the image of $H^{0,q}(C^*)\to H^{0,q}(C^*;\FF_2)$.
  Further, $f^*\alpha$ vanishes on $A$, so $f^*\alpha$ is in the image
  of a class $\beta\in H^{0,q}(D^*;\FF_2)$.  Finally, since
  $(f^*\alpha)(1,0,\dots,0,0,\dots,0)=1$, the image of $\beta$ in
  $H^{0,q}(h^{-1}D^*;\FF_2)$ is non-zero.
  Thus, $(0,-q)$ is oddly reduced-full, proving the result.
\end{proof}

\begin{corollary}
  \label{cor:o-mirror}
  If $(C, D, f)$ is a reduced LEO triple, then at least one of
  $\stil_o(C, D, f)$ and $\stil_o(C^*, D^*, f^*)$ is non-negative.
\end{corollary}

\begin{proof}
If $s_{\FF_2}(C, D, f) \geq 2$, then $\stil_o(C, D, f) \geq 0$ as
needed; if $s_{\FF_2}(C, D, f) \leq -2$, then
$s_{\FF_2}(C^*, D^*, f^*) = -s_{\FF_2}(C, D, f) \geq 2$, giving
$\stil_o(C^*, D^*, f^*) \geq 0$.  In the remaining case when
$\stil_{\FF_2}(C, D, f) = 0$, if $\stil_o(C, D, f) \neq 0$, then
Lemma~\ref{lem:o-mirror} gives that
$\stil_o(C^*, D^*, f^*) = s_{\FF_2}(C^*, D^*, f^*) = 0$, proving the
claim.
\end{proof}

Define a relation on $\rCLEO^o$ by declaring that
$[(C,D,f)]\geq [(C',D',f')]$ if there
is a local map from $(C',D',f')$ to
$(C,D,f)$.

\begin{theorem}\label{thm:o-total-order}
  This definition specifies a translation-invariant total order on
  $\rCLEO^o$.
  Further, the total order is characterized by
  $[(C,D,f)]\geq \pi(\LEO(U))$ if and only if
  $\stil_o(C,D,f)\geq0$.
   (Here, $\pi(\LEO(U))$ denotes the identity in $\rCLEO^o$.)
\end{theorem}

\begin{proof}
This essentially follows from Corollary~\ref{cor:o-mirror} and the fact
that compositions and tensor products of local maps are local maps. In
more detail, for the first statement, we must check that the order is
transitive, anti-symmetric, translation-invariant, and that for any
$[(C,D,f)]$ and $[(C',D',f)]$, either
$[(C,D,f)]\geq [(C',D',f)]$ or
$[(C',D',f)]\geq
[(C,D,f)]$. Transitivity follows from the fact that a
composition of local maps is a local map, and translation invariance
follows from the fact that tensoring a local map
$(C,D,f)\to (C',D',f')$ with the identity map of $(C'',D'',f'')$ gives
a local map of tensor products. Anti-symmetry is immediate: if there
is a local map from $(C,D,f)$ to $(C',D',f')$ and from $(C',D',f')$ to
$(C,D,f)$ then, by definition, $(C,D,f)$ is locally equivalent to
$(C',D',f')$.  The claim that either
$[(C,D,f)]\geq [(C',D',f')]$ or
$[(C',D',f')]\geq [(C,D,f)]$ follows
from the fact that for any $(C,D,f)$ there is either a
local map from $\pi(\LEO(U))$ to $(C,D,f)$ or
vice-versa, which in turn follows from Corollary~\ref{cor:o-mirror}
and (the proof of) Theorem~\ref{thm:reduced-local-is-e}.

For the second statement, from the proof of
Theorem~\ref{thm:reduced-local-is-e},
$\stil_o(C,D,f)\geq0$ if and only if there is a local
map from $\pi(\LEO(U))$ to $(C,D,f)$, so the property
that $\stil_o(C,D,f)\geq0$ characterizes the
non-negative elements of $\rCLEO^o$. By translation invariance, this
in turn characterizes the total order.
\end{proof}

\begin{corollary}
  Every nontrivial element of $\rCLEO^o$ has infinite order.
\end{corollary}

\begin{remark}\label{remark:rCLEO-not-ordered}
  As we will now show, the analogue of Corollary~\ref{cor:o-mirror}
  does not hold for $\stil_c$, and hence the proof of 
  Theorem~\ref{thm:o-total-order} does not generalize to $\rCLEO$.
  There is a knot $K_1$ with $s_{\FF_2} = 2$ and $s_{\FF_p} = 0$ for all
  other $p$, and a knot $K_2$ with $s_{\FF_3} = -2$ and
  $s_{\FF_p} = 0$ for all other $p$~\cite{LewarkZib,
    Schutz:integral-s}. Taking $K = K_1 \# K_2$, we claim
  that $\stil_c(K)$ and $\stil_c(\Kbar)$ are both negative.  By
  Lemma~\ref{lem:obstruction_order}, it suffices to show $s_\ZZ$ is
  negative for both.  A result of
  Sch\"utz~\cite[Corollary~4.9]{Schutz:integral-s} gives that
  $s_\ZZ \leq s_\FF$ for any field $\FF$, so we have
  $s_\ZZ(K) \leq s_{\FF_3}(K) = -2$ and
  $s_\ZZ(\Kbar) \leq s_{\FF_2}(\Kbar) = -2$, as needed.
\end{remark}

\begin{remark}
  Theorem~\ref{thm:reduced-local-is-e} reminds us of Hom's
  construction of her group $\mathcal{CFK}$ of knot Floer-like
  complexes modulo the complexes with
  $\varepsilon=0$~\cite{Hom15:inf-rank}. The role of her concordance
  invariant $\varepsilon$ is played here by
  $\bigl(s_{\FF_2}(C,D,f)-\stil_c(C,D,f),-s_{\FF_2}(C,D,f)-\stil_c((C,D,f)^*)\bigr)$,
  which takes values in $\{(0,0),(2,0),(0,-2)\}$.
  For the analogue for
  $s_o$, Theorem~\ref{thm:o-total-order} pushes this analogy further,
  showing that, like Hom's group $\mathcal{CFK}$, the version
  $\rCLEO^o$ of the even-odd local equivalence group has a total
  order. 
\end{remark}

\section{The case of knots}
\label{sec:knot-case}
This section has two goals. The simpler is to connect the refined
$s$-invariants with the concordance group and the slice genus; we do
that near the outset. The other is to explore relations between these
invariants that hold for knots but not general LEO triples. That is,
for an arbitrary LEO triple, the invariants $s^+_\FF$ and $s^-_\FF$
are typically distinct, by Lemma~\ref{lem:splus-not-homom}, but for
knots they agree, and agree with the reduced version $s_\FF$. This
fact, and the techniques underlying its proof, imply more relations
between the refined $s$-invariants of LEO triples coming from
knots. We state the key relations first, then develop the properties
of the Khovanov complexes needed to prove them in
Section~\ref{subsec:conjugation}, before giving the proofs themselves
in Section~\ref{subsec:s-rels-proof}. Finally, in
Section~\ref{sec:more-relations}, we use the structural results from
Section~\ref{sec:LEO-structure} to study the relationships between the
Bockstein and comprehensive refinements.

To start, we note:
\begin{proposition}\label{prop:knot-conc-invts}
  Let $(C,D,f)=\LEO(K)$ or $\LEE(K)$. Then
  the numbers
  $r^{\alpha}(C,D,f)$, $s^{\alpha}(C,D,f)$, and
  $\stil^{\alpha}(C,D,f)$ for $\alpha\in\{\beta_n,\beta\}$ ($1\leq
  n\leq \infty$), and the numbers $r_\gamma(C,D,f)$,
  $s_\gamma(C,D,f)$, and $\stil_\gamma(C,D,f)$ for $\gamma\in\{o, c\}$, are all concordance invariants of $K$. 
\end{proposition}
\begin{proof}
  This is immediate from
  Propositions~\ref{prop:conc-implies-local-equiv}
  and~\ref{prop:rCLEO} and Lemma~\ref{lem:s-beta-descends}. 
\end{proof}
For the Bockstein refinements, the key relations are:
\begin{theorem}
  \label{thm:bockstein_s_vals}
  For any knot $K$ and $\alpha\in\{\beta_n,\beta\}$, the invariants
  $r^\alpha$, $s^{\alpha}$, and $\stil^{\alpha}$ (applied to $\LEO(K)$
  or $\LEE(K)$) lie in $\{s_{\FF_2}(K),s_{\FF_2}(K)+2\}$. Further,
  \begin{gather*}
    s_{\FF_2}(K)\leq r^{\alpha}(\LEO(K))=\stil^{\alpha}(\LEO(K))= s^{\alpha}(\LEO(K))= s_{\FF_2}(K)+2,\\
    s^{\beta_1}(\LEO(K))\leq s^{\beta_2}(\LEO(K))\leq\cdots\leq s^{\beta_\infty}(\LEO(K)).
  \end{gather*}
\end{theorem}
(The analogue for $\LEE(K)$ is given in Lemma~\ref{lem:bockstein_s_vals}.)
The following theorem gives the analogous results for the
comprehensive refinements:
\begin{theorem}\label{thm:s-rels}
  For any knot $K$,
  \begin{align*}
    s_{\ZZ}(K)-2\leq r_c(K) &= \stil_c(K)=s_c(K)\leq s_{\ZZ}(K)\\
    s_{\FF_2}(K)-2\leq r_o(K) &= \stil_o(K)=s_o(K)\leq s_{\FF_2}(K)\\
    \stil_c(K)&\leq \stil_o(K).
  \end{align*}
  (Here, we have shortened notation by writing $K$ to mean $\LEO(K)$.)
\end{theorem}

Both theorems are proved in Section~\ref{subsec:s-rels-proof}, after
we develop some more machinery in Section~\ref{subsec:conjugation}.

\begin{remark}
  For an algebraic example illustrating that $s_{\FF_2}-2$ is not a
  lower bound for $s_c$ and the other complete invariants,
  consider the complex $D$ given by
  \[
    \begin{tikzcd}[column sep=0pt]
      \BNring\{5\} & & \BNring\{3\} & & \BNring\{1\} \\
      & \BNring\{3\} \arrow[ul, "h"] \arrow [ur, swap, "3"] & & \BNring\{1\} \arrow[ul, "h"] \arrow[ur, swap, "3"] &
    \end{tikzcd}
  \]
with the top line in homological degree $0$.
One checks that $s_{\FF_2}$ is $4$, but $H^{0, q}(D;\ZZ) \to
H^{0,q}(h^{-1}D;\ZZ)$ is surjective only for $q\leq -1$. With
$C=D\otimes_\BNring \ZZ[X]/(X^2)$ we get $s_c(C,D,\id) = r_c(C,D,\id)
= 0$. While this example is algebraic, as noted in Remark~\ref{remark:rCLEO-not-ordered}, there exist knots $K$ with $s_{\FF_2}(K)>s_{\FF_3}(K)$ (\cite{LewarkZib, Schutz:integral-s}) and, by considering connected sums of such knots, we can make the difference between $s_{\FF_2}$ and $s_c$ arbitrarily large.
Consequently, there are knots where $s_c$ and $s_o$ differ by
arbitrarily large amounts.
\end{remark}

The example of Remark~\ref{rem:sc-not-enough} showed that $r_c$ alone
does not detect local equivalence in $\CLEO$. One of the reasons
Theorem~\ref{thm:s-rels} is powerful is because for knots, $r_c$ does
suffice to detect local equivalence:

\begin{theorem}\label{thm:local-is-e}
  If $K$ is a knot so that $r_c(\LEO(K))=r_c(\LEO(\Kbar))=0$ then
  $\LEO(K)$ is locally equivalent to $\LEO(U)$.
\end{theorem}

Again, the proof is deferred to Section~\ref{subsec:s-rels-proof}.

\begin{corollary}\label{cor:CLEO-to-rCLEO-iso}
  Let $\LEO(\mathcal{C})$ be the image of the smooth concordance group
  in $\CLEO$, and $\pi\co \CLEO\to\rCLEO$ be the quotient map. Then
  the restriction of $\pi$ to $\LEO(\mathcal{C})$ is injective. In
  particular, $\LEO(K)$ is trivial in $\CLEO$ if and only if
  $\stil_c(\LEO(K))=\stil_c(\LEO(\Kbar))=0$.
\end{corollary}
\begin{proof}
  By Theorem~\ref{thm:reduced-local-is-e}, a knot maps to zero in
  $\rCLEO$ if and only if we have $\stil_c(\LEO(K))=\stil_c(\LEO(\Kbar))=0$. By
  Theorem~\ref{thm:local-is-e}, a knot maps to zero in $\CLEO$ if and
  only if $r_c(\LEO(K))=0$. By Theorem~\ref{thm:s-rels},
  $\stil_c(\LEO(K))=r_c(\LEO(K))$.
\end{proof}

\begin{corollary}\label{cor:stilc-triv-unreduced}
  Given a knot $K$, $\stil_c(\LEO(\Kbar))=-\stil_c(\LEO(K))$
  if and only if $\LEO(K)$ is locally equivalent to $\LEO(U)\{q\}$,
  where $q=\stil_c(\LEO(K))$.
\end{corollary}
\begin{proof}
  The ``if'' direction is trivial.
  For the other direction, observe that, for any even
  $q\in\ZZ$, two LEO triples $(C,D,f)$ and
  $(C',D',f')$ are locally equivalent if and only if
  $(C,D,f)\{q\}$ is locally equivalent to
  $(C',D',f')\{q\}$. Also, by definition, we have
  $\stil_c((C,D,f)\{q\}) = \stil_c(C,D,f) +q$.  Furthermore, if
  $T$ is the trefoil knot with
  $s(T)=-2$ then, by direct computation,
  $\LEO(K)$ is locally equivalent to
  $\LEO(U)\{-2\}$. So, for any knot $K$ and even integer
  $q$,
  $\LEO(K)\{-q\}=\LEO(K\#\frac{q}{2}T)$ is in the image of the smooth
  concordance group, so Corollary~\ref{cor:CLEO-to-rCLEO-iso}
  applies to $\LEO(K)\{-q\}$.

  Now, if
  $\stil_c(\LEO(K))=-\stil_c(\LEO(\Kbar))=q$ then
  $\stil_c(\LEO(K)\{-q\})=\stil_c(\LEO(\Kbar)\{q\})=0$. Hence, by
  Theorem~\ref{thm:reduced-local-is-e},
  $\pi(\LEO(K)\{-q\})$ is reduced locally equivalent to
  $\pi(\LEO(U))$. Thus, by Corollary~\ref{cor:CLEO-to-rCLEO-iso},
  $\LEO(K)\{-q\}$ is locally equivalent to
  $\LEO(U)$, and the result follows.
\end{proof}

We conclude this subsection with two simpler topological properties of
these invariants.  First, as mentioned at the beginning of
Section~\ref{sec:refined-s}, the invariants we have constructed are
inspired by the refined $s$-invariants from the Khovanov stable
homotopy type~\cite{LS14:refine-s}. There, a refined $s$-invariant was
associated to any cohomology operation, using the even Khovanov stable
homotopy type. Analogous operations were constructed using the odd
stable homotopy type by
Sarkar-Scaduto-Stoffregen~\cite{SSS20:odd-htpy}.  The relationship of
these earlier invariants to the ones of Section~\ref{sec:refined-s}
is:

\begin{lemma}\label{lem:old-refinements}
  Let $r^{\Sq^1}$ and $s^{\Sq^1}$ be the refined $s$-invariants
  associated to the operation $\Sq^1$ on even Khovanov
  homology~\cite{LS14:refine-s}, and $r^{\Sq^1_o}$ and $s^{\Sq^1_o}$
  be the refined $s$-invariants associated to the operation $\Sq^1$ on
  odd Khovanov homology~\cite{SSS20:odd-htpy}. Then
  \begin{align*}
    r^{\Sq^1}(K)&=r^{\beta_1}(\LEE(K)) & 
    s^{\Sq^1}(K)&=s^{\beta_1}(\LEE(K)) \\
    r^{\Sq^1_o}(K)&=r^{\beta_1}(\LEO(K)) & 
    s^{\Sq^1_o}(K)&=s^{\beta_1}(\LEO(K))
  \end{align*}
\end{lemma}
\begin{proof}
  Recall that the first Steenrod square $\Sq^1$ is the Bockstein
  homomorphism $\beta_1$. So, the only difference between the two
  constructions is that here we have viewed the Bar-Natan complex as a
  graded complex over $\FF_2[h]$, while the earlier papers considered
  the filtered Bar-Natan complex corresponding to $D_{h=1}$.
  If we write $\mathcal{F}_qD_{h=1}\subset D_{h=1}$ for the image of
  $D^{(q)}$ under the quotient map $D\to D_{h=1}$, we get that the quotient map restricted to the graded subcomplex $D^{(q)}$ is injective, so that $H^{0,q}(D;\FF_2)\cong H^0(\mathcal{F}_qD_{h=1};\FF_2)$. 
  Also, $\mathcal{F}_qD_{h=1}/\mathcal{F}_{q+2}D_{h=1}\cong D_{h=0}^{(q)}$, and we can define $\beta_n$-half-full and $\beta_n$-full with $H^0(\mathcal{F}_qD_{h=1};\FF_2)$ in place of $H^{0,q}(D;\FF_2)$ and $H^0(D_{h=1};\FF_2)$ in place of $H^0(h^{-1}D;\FF_2)$.
  Because of the commutative diagram 
  \[
    \begin{tikzcd}
      H^{0,q}(D_{h=0};\FF_2) \arrow[d, "\cong"] & H^{0,q}(D;\FF_2) \arrow[l] \arrow[r] \arrow[d, "\cong"] & H^{0,q}(h^{-1}D;\FF_2) \arrow[d] \\
      H^0(\mathcal{F}_qD_{h=1}/\mathcal{F}_{q+2}D_{h=1};\FF_2) & H^0(\mathcal{F}_qD_{h=1};\FF_2) \arrow[l] \arrow[r] & H^0(D_{h=1};\FF_2)
    \end{tikzcd}
  \]
  we get that $q$ being $\beta_n$-(half-)full is implied by the same in this filtered sense. To see that these notions are equivalent, note that $h^{-1}D$ is also graded, so the image of $H^{0,q}(D;\FF_2)$ is contained in $H^{0,q}(h^{-1}D;\FF_2)\cong \FF_2\oplus \FF_2$, and this image is mapped isomorphically to $H^0(D_{h=1};\FF_2)$.
\end{proof}

\begin{remark}
  It is natural to ask whether the other refined invariants from the
  Khovanov stable homotopy type, like $s^{\Sq^2}$, are invariants of
  local equivalence. This seems unlikely to us, but we do not know
  a counterexample.
\end{remark}

Second, while they are not concordance homomorphisms, the refined
$s$-invariants do give slice genus bounds:
\begin{lemma}
  Let $\Sigma$ be a smooth, connected, orientable cobordism from $K_0$
  to $K_1$, and let $s'$ be any of the invariants in
  Proposition~\ref{prop:knot-conc-invts}. Then,
  \[
    |s'(K_0)-s'(K_1)|\leq -\chi(\Sigma).
  \]
\end{lemma}
(For $s'(K)=s^{\beta_n}(\LEE(K))$, say, this is a special case of a known result~\cite[Theorem~1]{LS14:refine-s}.)
\begin{proof}
  The proof of Proposition~\ref{prop:conc-implies-local-equiv} shows
  that if there is a smooth, connected, orientable cobordism $\Sigma$
  from $K_0$ to $K_1$ then there is a local map from $\LEO(K_0)$ to
  $\LEO(K_1)$, except shifting the quantum grading by
  $\chi(\Sigma)$. Given a configuration showing that $q$ is full or
  half-full in whatever sense for $K_0$, the image of that
  configuration under the local map shows that $q+\chi(\Sigma)$ is
  full or half-full in the same sense for $K_1$, so
  $s'(K_1)\geq s'(K_0)+\chi(\Sigma)$. Reversing the cobordism, the
  same argument gives $s'(K_0)\geq s'(K_1)+\chi(\Sigma)$, proving the
  result.
\end{proof}

\begin{remark}\label{remark:nothing}
  It follows from Theorems~\ref{thm:bockstein_s_vals} and~\ref{thm:s-rels}, and the computations in
  Section~\ref{sec:computations} proving the refinements all differ from
  $s_{\FF_2}$, that none of the refined $s$-invariants discussed in
  this section are homomorphisms from $\CLEO$ (or $\rCLEO$), or even
  from the subgroup generated by knots. Their nontriviality
  (Section~\ref{sec:computations}) does imply that the homomorphism
  $\vec{s}$ from Proposition~\ref{prop:Zinf-summand} has nontrivial
  kernel. It would be interesting to construct new homomorphisms from
  $\CLEO$ to $\ZZ$ giving new concordance homomorphisms. One strategy
  (inspired by Dai-Hom-Stoffregen-Truong's
  work~\cite{DHST21:more-conc}) might be to find an explicit answer to
  Question~\ref{question:identify-group}, or to prove a structure
  theorem for some other quotient of $\CLEO$.
\end{remark}

\subsection{Conjugation on the Khovanov complexes}\label{subsec:conjugation}
The special features of the Khovanov complexes of a knot which allows
us to prove Theorem~\ref{thm:s-rels} are a conjugation action and
particular generators for $h^{-1}\Kh_h(K)$.

Define $\overline{\,\cdot\,}\colon \BNring \to \BNring$ as the $\ZZ[h]$-linear map with $\overline{1} = 1$ and $\overline{X} = h-X$. It is straightforward to check that this is a ring homomorphism with $\overline{\overline{x}} = x$ for all $x\in \BNring$. We call this automorphism the {\em conjugation of }$\BNring$.

Conjugation gives rise to a bigrading-preserving
involution $I\colon \KhCx_h(K)\to \KhCx_h(K)$ as follows. If $v$ is a vertex in the cube of resolutions for $K$, and $x = x_1\otimes x_2 \otimes \cdots \otimes x_k\in \KhCx_h(K)$ is an element over $v$, define
\[
I(x) = (-1)^{\Split(v)} \overline{x}_1\otimes \overline{x}_2 \otimes \cdots \otimes \overline{x}_k.
\] 
This commutes with the differential on $\KhCx_h(K)$ \cite[\S
2]{Schutz:integral-s} and satisfies $I(r\cdot x) = \overline{r}I(x)$
for all $r\in \BNring$. Also, if we write $\KhCx_{h=1}(K)$ for
$\KhCx_h(K)_{h=1}$, the map $I$ descends to $I\colon \KhCx_{h=1}(K)\to \KhCx_{h=1}(K)$, and the crucial property of this $I$ is that there exists $\varepsilon_q\in \{\pm 1\}$ with
\begin{equation}\label{eq:lift-filtered}
x+\varepsilon_q I(x) \in \mathcal{F}_{q+2}
\end{equation}
for all $x\in \mathcal{F}_q\subset \KhCx_{h=1}(K)$ \cite[Lemma 2.1]{Schutz:integral-s}. Here, as in the reduced case, $\mathcal{F}_q = p(\KhCx_h^q(K))$ with $p\colon \KhCx_h(K)\to \KhCx_{h=1}(K)$. This property can be lifted to $\KhCx_h(K)$:
\begin{lemma}\label{lem:chainT}
There is a chain map $T\colon \KhCx_h^q(K)\to \KhCx_h^{q+2}(K)$ over $\ZZ$ so
that
\begin{equation}\label{eq:lift-grading}
\id +\varepsilon_q I = h\cdot T. 
\end{equation}
Furthermore, at the level of homology, there is a commutative diagram
\begin{equation}\label{eq:shuma_diag}
\begin{tikzcd}
\Kh_h^{i,q}(K) \ar[r, "T"] \ar[d] & \Kh_h^{i,q+2}(K)\ar[d] \\
\Kh^{i,q}(K;\FF_2) \ar[r, "\nu"] & \Kh^{i,q+2}(K;\FF_2)
\end{tikzcd}
\end{equation}
where $\nu$ is Shumakovitch's acyclic differential on Khovanov homology over $\FF_2$ \cite[\S 3.2]{Shu14:torsion}.
\end{lemma}
\begin{proof}
  For $x\in \KhCx_h^q(K)$ we get $p(x)+\varepsilon_qI(p(x)) = p(y)$ for some $y\in \KhCx_h^{q+2}(K)$ by Formula~(\ref{eq:lift-filtered}). Since $p|_{\KhCx_h^{(q+2)}}$ is injective, this uniquely determines $y$. But we also have $x+\varepsilon_q I(x)\in \KhCx_h^q(K)$, since $I$ preserves gradings. Also, $hy\in \KhCx_h^q(K)$ and therefore
\[
p(x+\varepsilon_q I(x)) = p(y) = p(hy).
\]
Hence $x+\varepsilon_qI(x)-hy\in \KhCx^q(K)$ is in $\ker p$. But $p$
restricted to $\KhCx_h^q(K)$ is injective, so $x+\varepsilon(q)I(x)-hy=0$,
and we can define $T(x) = y\in \KhCx_h^{q+2}$ and obtain
Formula~\eqref{eq:lift-grading}. Since $\varepsilon_q$ is fixed and $I$ is a
chain map, we get that $h\cdot T$ is a chain map. Therefore $T\circ
\partial - \partial\circ T$ lands in $\ker(h\colon \KhCx_h(K) \to \KhCx_h(K))$. Hence,
since multiplication by $h$ is injective on the free complex $\KhCx_h(K)$, we
see $T$ is a chain map. 

For the commutative diagram (\ref{eq:shuma_diag}), note that after passing to $\FF_2$ coefficients, we can ignore $\varepsilon_q$. In particular, we then have a chain map $T\colon \KhCx_h(K;\FF_2)\to \KhCx_h(K;\FF_2)$ which respects the action of $\FF_2[h]$. To see that setting $h=0$ turns $T$ into $\nu$ requires us to look more carefully at the definition. A typical basis element of $\KhCx_h(K;\FF_2)$ over $\FF_2[h]$ is of the form $b = \eta_1\otimes \cdots \otimes \eta_k$ with $\eta_j\in \{1,X\}$ for all $j=1,\ldots k$. Then $I(b)$ replaces each $\eta_j = X$ with $h+X = h+\eta_j$. In particular,
\[
I(b) = b + h T_1(b) +h^2 T_2(b) +\cdots +h^k T_k(b),
\]
where $T_j(b)$ is the sum of all $\bar{\eta}_1\otimes \cdots \otimes\bar{\eta}_k$, where exactly $j$ of the $\eta_m$ equal to $X$ have been replaced by $1$, and all $k-i$ other $\bar{\eta}_m = \eta_m$. In particular, $T_1$ agrees with Shumakovitch's $\nu$ \cite[Definition 3.1.A]{Shu14:torsion} and
\[
T = \nu + h T_2+\cdots+ h^{k-1} T_k.
\]
Setting $h=0$ gives the desired result. 
\end{proof}

Let $\varphi\colon h^{-1}\BNring\{1\}\to h^{-1}\KhCx_h(K)$ be a grading-preserving $\BNring$-linear homotopy equivalence. Then $\varphi(1)$ is a cycle which represents a homology class in $h^{-1}\Kh^{0,1}_h(K)$ that generates $h^{-1}\Kh_h(K)$ as a free $h^{-1}\BNring$-module. We can also pull back the involution $I\colon h^{-1}\Kh_h(K)\to h^{-1}\Kh_h(K)$ to $I_{h^{-1}\BNring}\colon h^{-1}\BNring\{1\}\to h^{-1}\BNring\{1\}$. This is a $\ZZ[h]$-linear grading-preserving involution on $h^{-1}\BNring\{1\}$ which satisfies $I_{h^{-1}\BNring}(rx) = \overline{r}I_{h^{-1}\BNring}(x)$ for all $r\in h^{-1}\BNring$ and $x\in h^{-1}\BNring\{1\}$. This implies that $I_{h^{-1}\BNring}$ is either plus or minus the conjugation.

Define
\begin{align*}
\mathcal{O} &= (h-X) [\varphi(1)] \in h^{-1}\Kh_h^{0,-1}(K) \\
\overline{\mathcal{O}} &= I(\mathcal{O}) = \pm X [\varphi(1)] \in h^{-1}\Kh_h^{0,-1}(K).
\end{align*}
Then $\mathcal{O},\overline{\mathcal{O}}$ generate $h^{-1}\Kh_h^{0,-1}(K)\cong \ZZ^2$ as an abelian group. Observe that $(h-X)\mathcal{O} = h\mathcal{O}$ and $X\mathcal{O} = 0$. Furthermore, each combination $m\mathcal{O}+n\overline{\mathcal{O}}$ with $\gcd(n,m) = 1$ is a primitive element in $h^{-1}\Kh_h^{0,-1}(K)$, but only for $m,n\in \{\pm 1\}$  does this element generate $h^{-1}\Kh_h(K)$ as an $h^{-1}\BNring$-module. This is because $\pm h$ and $\pm (h-2X)$ are the only units of $h^{-1}\BNring$ in quantum degree $-2$ (the inverse of the latter is $\pm(h-2X)h^{-2}$).

\begin{lemma}\label{lem:minus-plus}
If $i\colon \Kh_h^{0,q-2}(K) \to h^{-1}\Kh_h^{0,q-2}(K)$ is surjective, then there exists $a\in \Kh_h^{0,q}(K)$ such that $i(a)$ generates $h^{-1}\Kh_h(K)$ as an $h^{-1}\BNring$-module. Furthermore, if $b_1,b_2\in  \Kh_h^{0,q-2}(K)$ are such that $i(b_1),i(b_2)$ generate $h^{-1}\Kh_h^{0,q-2}(K)$, then $a = T(b)$ for some linear combination $b$ of $b_1$ and $b_2$.
\end{lemma}
\begin{proof}
  There exists $j$ such that $h^j\mathcal{O}, h^j\overline{\mathcal{O}}$ generate $h^{-1}\Kh_h^{0,q-2}(K)$. By assumption, there exists a cycle $b\in \KhCx_h^{0,q-2}(K)$ such that $i(b)$ represents $h^j\mathcal{O}$. Then, $i(I(b))$ represents $h^j\overline{\mathcal{O}}$. By Formula~(\ref{eq:lift-grading}), $a = T(b)$ is a cycle in $\KhCx_h^{0,q}(K)$ such that $i(a)$ represents $h^{j-1}(\mathcal{O}\pm \overline{\mathcal{O}})$, which generates $h^{-1}\Kh_h(K)$ as an $h^{-1}\BNring$-module. The second statement follows by construction of $a$.
\end{proof}

\begin{lemma}\label{lem:plus-minus}
If there is an element $a\in \Kh_h^{0,q}(K)$ with $i(a) \in h^{-1}\Kh_h^{0,q}(K)$ primitive, then $i\co \Kh_h^{0,q-2}(K)\to h^{-1}\Kh_h^{0,q-2}(K)$ is surjective.
Furthermore, if $i(a)$ does not generate $h^{-1}\Kh_h(K)$ as an $h^{-1}\BNring$-module, then $i\colon \Kh_h^{0,q+2}(K)\to h^{-1}\Kh_h^{0,q+2}(K)$ is non-zero.
\end{lemma}

\begin{proof}
  We can write $i(a) = mh^{j-1}\mathcal{O}+nh^{j-1}\overline{\mathcal{O}}$ with $m,n\in \ZZ$ such that $\gcd(m,n) = 1$.

Denote the image of $i\colon \Kh_h^{0,q}(K)\to h^{-1}\Kh_h^{0,q}(K)\cong \ZZ\oplus \ZZ$ by $H$. Then $mh^j\mathcal{O}+nh^j\overline{\mathcal{O}}\in H$. We also get $i((h-X)a) = mh^j\mathcal{O}\in H$ and therefore also $nh^j\overline{\mathcal{O}}\in H$. Applying $I$ to $(h-X)a$ and $Xa$ also shows that $I(mh^j\mathcal{O}) = mh^j\overline{\mathcal{O}}\in H$ as well as $I(nh^j\overline{\mathcal{O}}) = nh^j\mathcal{O}\in H$. This implies
\begin{equation}\label{eq:eucl-alg}
mh^j\mathcal{O}, nh^j\mathcal{O}, nh^j\overline{\mathcal{O}}, mh^j\overline{\mathcal{O}} \in H.
\end{equation}
In particular, if $(m,n) = (1,0)$ or $(m,n)=(0,1)$ we get $H=\ZZ\oplus \ZZ$. Otherwise, we have $|m|\geq |n|$ or $|m| < |n|$. In both cases, Formula~(\ref{eq:eucl-alg}), $\gcd(m,n)=1$ and the Euclidean Algorithm imply $\mathcal{O}, \overline{\mathcal{O}}\in H$. This shows surjectivity of $i$ in quantum grading $q-2$.

By Lemma \ref{lem:minus-plus} there exists an $a'\in\Kh_h^{0,q}(K)$ such that $i(a')$ generates $h^{-1}\Kh_h(K)$ as an $h^{-1}\BNring$-module. In particular, we can assume that $i(a') = h^{j-1} (\mathcal{O}\pm \overline{\mathcal{O}})$. If we assume that $i(a)$ does not generate $h^{-1}\Kh_h(K)$ as an $h^{-1}\BNring$-module, then $|m|\not= |n|$, and for some linear combination of $a$ and $a'$ we get that $(n\pm m)h^{j-1}\mathcal{O}\not=0$ is in the image of $i$. Let $b\in \KhCx_h^q(K)$ be a cycle such that $i(b)$ represents this multiple of $h^{j-1}\mathcal{O}$. Using Formula~(\ref{eq:lift-grading}), we get that $T(b)\in \KhCx_h^{q+2}(K)$ is a cycle with $i(T(b))$ representing $(n\pm m)h^{j-2}(\mathcal{O}\pm \overline{\mathcal{O}})$, a non-zero element in $h^{-1}\Kh_h^{0,q+2}(K)$.
\end{proof}

The reduced complex is defined as
\[
\wt{\KhCx}_h(K) = \KhCx_h(K)\otimes_\BNring \ZZ[h]\{-1\}
\]
where $X$ acts on $\ZZ[h]$ as $0$. This gives rise to a short exact sequence of chain complexes
\begin{equation}\label{eq:ses-bn}
\begin{tikzcd}
0 \ar[r] & \rKhCx_h(K)\{-1\} \ar[r, "S"] & \KhCx_h(K) \ar[r, "\pi"] & \rKhCx_h(K)\{1\} \ar[r] &0
\end{tikzcd}
\end{equation}
with $S(x) = X I(x)$ (where we view $x$ as a labeling of the circles
that labels the marked circle $1$) and $\pi(x) = x\otimes 1$.
After inverting $h$ this gives rise to another short exact sequence
\begin{equation}\label{eq:ses-bn-2}
\begin{tikzcd}[scale cd=0.91, column sep=1.8em]
0 \ar[r] & h^{-1}\rKh_h^{0,q+1}(K) \ar[r, "S"] & h^{-1}\Kh_h^{0,q}(K) \ar[r, "\pi"] & h^{-1}\rKh_h^{0,q-1}(K) \ar[r] & 0
\end{tikzcd}
\end{equation}
with $\pi(h^k\overline{\mathcal{O}}) = 0$ and $\pi(h^k\mathcal{O})$ generating $h^{-1}\rKh_h^{0,q-1}(K)$. Here, $k$ is an integer ensuring the right quantum degree.

\subsection{Applications of conjugation to the refined \texorpdfstring{$s$}{s}-invariants}
\label{subsec:s-rels-proof}

Recall from Section~\ref{subsec:graded-s} that for a knot $K$ we have
\begin{align*}
s^+_\ZZ(K) &= \max\{q \mid \exists\, a\in \Kh^{0,q}_h(K)\mbox{ with }i(a)\in h^{-1}\Kh^{0,q}_h(K)\mbox{ primitive}\}-1,\\
s^-_\ZZ(K) &= \max\{q \mid \Kh_h^{0,q}(K)\to h^{-1}\Kh^{0,q}_h(K)\mbox{ surjective}\}+1,\\
s_\ZZ(K) &= \max\{q \mid \rKh_h^{0,q}(K)\to h^{-1}\rKh_h^{0,q}(K)\mbox{ surjective}\}.
\end{align*}

\begin{lemma}\label{lem:allszettsame}
Let $K$ be a knot. Then
\[
s^+_\ZZ(K) = s^-_\ZZ(K) = s_\ZZ(K).
\]
\end{lemma}

\begin{proof}
  Let $q = s^-_\ZZ(K)-1$, so that $i\co \Kh_h^{0,q}(K)\to h^{-1}\Kh_h^{0,q}(K)\cong\ZZ\oplus \ZZ$ is surjective. By Lemma \ref{lem:minus-plus}, $h^j\mathcal{O}\pm h^j\overline{\mathcal{O}}$, one of the generators of $h^{-1}\Kh_h(K)$ and hence primitive, is $i(a)$ for some $a$ in quantum degree $q+2$. Then $\pi(h^j\mathcal{O}\pm h^j\overline{\mathcal{O}}) = h^j\pi(\mathcal{O})\in h^{-1}\rKh_h^{0,q+1}(K)$ is $i(\pi(a))$, so is also in the image,
  and generates this abelian group.  It follows that
\[
s^-_\ZZ(K)  \leq s^+_\ZZ(K) \mbox{ and } s^-_\ZZ(K) \leq s_\ZZ(K). 
\]

Now let $q=s_\ZZ(K)$, so that $i\colon \rKh_h^{0,q}(K)\to h^{-1}\rKh_h^{0,q}(K)$ is surjective. Then there exists $a$ with $i(a) = h^{j-1}\pi(\mathcal{O})\in h^{-1}\rKh_h^{0,q}(K)$. Since $S(\pi(\mathcal{O})) = \overline{\mathcal{O}}$, we get $b = \pm S(a)\in \Kh_h^{0,q-1}(K)$ with $i(b) = h^{j}\overline{\mathcal{O}}$. Also, $i(I(b)) = \pm h^{j}\mathcal{O}$ and hence $i\colon \Kh_h^{0,q-1}(K)\to h^{-1}\Kh_h^{0,q-1}(K)$ is surjective. Therefore
\[
s_\ZZ(K)\leq s^-_\ZZ(K) \mbox{ and hence } s_\ZZ(K) = s^-_\ZZ(K).
\]

It remains to show that for $q = s^+_\ZZ(K)-1$ we get $i\colon \Kh_h^{0,q}(K)\to h^{-1}\Kh_h^{0,q}(K)$ is surjective. This follows from Lemma \ref{lem:plus-minus} and therefore
\[
s^+_\ZZ(K) \leq s^-_\ZZ(K). \qedhere
\]
\end{proof}

\begin{remark}\label{rem:no-better}
The graded groups $H^0(h^{-1}\rKhCx_h(K))^{(q)}$ contain $p$-torsion if and only if the analogous unreduced graded groups contain $p$-torsion \cite[Lemma 4.15]{Schutz:integral-s}. Combined with Lemma \ref{lem:allszettsame}, this shows that an unreduced graded $s$-invariant contains the same information as the $s^\ZZ(K)$ defined by Sch\"utz \cite{Schutz:integral-s}.
\end{remark}


The first step towards proving Theorem~\ref{thm:bockstein_s_vals} is:
\begin{lemma}
  \label{lem:bockstein_s_vals}
  For any knot $K$ and $\alpha\in\{\beta_n,\beta\}$, the invariants
  $r^\alpha$, $s^{\alpha}$, and $\stil^{\alpha}$ (applied to $\LEO(K)$
  or $\LEE(K)$) lie in $\{s_{\FF_2}(K),s_{\FF_2}(K)+2\}$. Further,
  \begin{gather*}
    s_{\FF_2}(K)\leq r^{\alpha}(\LEO(K))\leq\stil^{\alpha}(\LEO(K))\leq s^{\alpha}(\LEO(K))\leq s_{\FF_2}(K)+2,\\
    s^{\beta_1}(\LEO(K))\leq s^{\beta_2}(\LEO(K))\leq\cdots\leq s^{\beta_\infty}(\LEO(K)),
  \end{gather*}
  and similarly for $\LEE(K)$. 
\end{lemma}
\begin{proof}
Most of this is a rephrasing of observations from
  Section~\ref{sec:refined-s}, keeping in mind that for knots,
  $s^{\pm}_{\FF_2}(\LEO(K))=s^{\pm}_{\FF_2}(\LEE(K))=s_{\FF_2}(K)$. To
  see the inequalities between $r^\alpha$, $\stil^\alpha$, and $s^\alpha$, assume that
  $r^\alpha(\LEO(K)) = s_{\FF_2}(K)+2 \eqqcolon s+2$. Then there exists $a\in \Kh^{0,s+1}_h(K;\FF_2)$ and $\check{a}\in \Kh^{-1,s+1}_o(K;\mathbb{A})$ with $0\not = i(a)\in h^{-1}\Kh^{0,s+1}(K;\FF_2)$ and $\alpha(\check{a}) = p(a)\in \Kh^{0,s+1}(K;\FF_2)$. Here, $\mathbb{A}$ is the coefficient group appropriate for $\alpha$. Notice that $i$ cannot be surjective in quantum degree $s+1$, and therefore $i(I(a)) = i(a)$. This implies $0\not=\pi(i(a))\in h^{-1}\rKh^{0,s}_h(K;\FF_2)$ and the naturality of the Bockstein homomorphism shows that $\pi(a)$ and $\pi(\check{a})$ witness that $s$ is $\alpha$-reduced full.
  
To see that $\stil^\alpha(\LEO(K)) \leq s^\alpha(\LEO(K))$, assume that $s$ is $\alpha$-reduced-full, witnessed by $a\in \rKh^{0,s}_h(K;\FF_2)$ and $\check{a}\in \rKh^{-1,s}_o(K;\mathbb{A})$. Then $S(a)\in \Kh^{0,s-1}_h(K;\FF_2)$ and $S(\check{a})\in \Kh^{-1,s-1}_o(K;\mathbb{A})$ witness that $s-1$ is $\alpha$-half-full. Using $b = I(S(a))+S(a)\in \Kh^{0,s-1}_h(K;\FF_2)$ and $\check{b} = 0 \in \Kh^{-1,s-1}_o(K;\mathbb{A})$ we see that $s-1$ is in fact $\alpha$-full.

The proof for $\LEE(K)$ is identical.
\end{proof}

The analogous first step towards Theorem~\ref{thm:s-rels} is slightly harder:
\begin{lemma}\label{lem:obstruction_order}
  For any knot $K$,
  \[
    \begin{tikzcd}[column sep=12pt]
       s_{\ZZ}(K)-2\arrow[r,"\leq"] \arrow[dr,"\leq"] & r_c(K)\ar[dr,"\leq"]\arrow[r,"\leq"] &
      \stil_c(K)\arrow[r,"\leq"]\arrow[dr,"\leq"] & s_c(K)\arrow[dr,"\leq"] \ar[r,"\leq"]& s_\ZZ(K) \arrow[dr,"\leq"]\\
       & s_{\FF_2}(K)-2\arrow[r,"\leq"] &  r_o(K)\ar[r,"\leq"] &
       \stil_o(K)\ar[r,"\leq"] & s_o(K)\ar[r,"\leq"]
       & s_{\FF_2}(K).
    \end{tikzcd}
  \]
\end{lemma}
\begin{proof}
The fact that $r_c(K)\leq r_o(K)$, and similarly for $\stil$ and $s$,
is immediate from the definitions.  The claims 
$s_{\FF_2}(K)-2\leq r_o(K)$ and $s_o(K)\leq s_{\FF_2}(K)$ follow from
Formulas~\eqref{eq:s-versus_ru} and \eqref{eq:s-versus_su} and the
fact that for knots, $s^+_{\FF_2}=s^-_{\FF_2}=s_{\FF_2}$. The
inequality between $s_{\ZZ}$ and $s_{\FF_2}$ follows from a result of
Sch\"utz~\cite[Corollary 4.9]{Schutz:integral-s}.  (In particular,
this part also applies to $s_{\FF}$ for any field $\FF$:
$s_{\ZZ}(K)\leq s_{\FF}(K)$ for any field $\FF$.)  The fact that
$s_c(K)\leq s_{\ZZ}(K)$ is immediate from the definitions and the fact
that $s_{\ZZ}(K)=s_{\ZZ}^-(K)$ by Lemma~\ref{lem:allszettsame}.

To see that $s_{\ZZ}(K)-2\leq r_c(K)$, let $q$ be the maximal grading
so that there is an $a\in \Kh_h^{0,q}(K;\ZZ)$ with
$i(a)\in h^{-1}\Kh_h^{0,q}(K;\ZZ)$ primitive. Since
$s^+_\ZZ(K)=s_{\ZZ}(K)$ by Lemma~\ref{lem:allszettsame},
$q=s_{\ZZ}(K)+1$. Then $ha$ also has
$i(ha)\in h^{-1}\Kh_h^{0,q-2}(K;\ZZ)$ primitive, but now
$p(ha)=0\in \Kh^{0,q-2}(K;\FF_2)$, so in particular $p(a)=j(0)$,
and hence $q-2$ is completely half-full. A similar argument shows that
$s_{\ZZ}(K)-2\leq \stil_c(K)$ and the same for $s_c(K)$.

Finally, we show that $r_c(K)\leq \stil_c(K)\leq s_c(K)$.
The proof is similar to the proof of Lemma~\ref{lem:allszettsame}.

To see that $r_c(K)\leq \stil_c(K)$, recall that we showed in
Lemma~\ref{lem:obstruction_order} that
$s_{\ZZ}(K)-2\leq r_c(K), \stil_c(K), s_c(K)\leq s_{\ZZ}(K)$; hence,
if $r_c(K)=s_{\ZZ}(K)-2$ then there is nothing to show. So, let
$q=s^+_\ZZ(K)+1$ and assume that $q$ is completely half-full,
witnessed by elements $a\in \Kh_h^{0,q}(K)$ and
$\check{a}\in \oKh^{0,q}(K)$. We can write
$i(a)\in h^{-1}\Kh_h^{0,q}(K)$ as
$m h^j\mathcal{O}+nh^j\ol{\mathcal{O}}$, for some $j$, where
$\gcd(m,n) = 1$. The image $\pi(i(a))=i(\pi(a))$ in
$h^{-1}\rKh_h^{0,1}(K)$ is $mh^j\pi(\mathcal{O})$, which is
not a generator unless $m=\pm 1$. So, consider the element
$I(a)$. The image of $I(a)$ in $\oKh^{0,q}(K)$ agrees with the
image of $\pm a$, and hence of $\pm\check{a}$, and
$i(I(a))=nh^j\mathcal{O}+mh^j\ol{\mathcal{O}}$. Since $\gcd(m,n)=1$,
there are integers $\alpha,\beta$ so that
$\pi(i(\alpha a+\beta I(a)))=h^j\pi(\mathcal{O})$. Then
$\pi(\alpha a+\beta I(a))\in \rKh_h^{0,1}(K)$ and
$\pi((\alpha\pm\beta)\check{a})\in \roKh^{0,q}(K)$ witness the
fact that $q-1$ is completely reduced-full, so
$\stil_c(K)=s_{\ZZ}(K)$.

Similarly, to see that $\stil_c(K)\leq s_c(K)$, if suffices to
consider the case that $q=s_\ZZ(K)$ is completely reduced-full. Then
we can find $a\in \rKh_h^{0,q}(K)$ and
$\check{a}\in \roKh^{0,q}(K)$ which map to the same element of
$\rKh^{0,q}(K;\FF_2)$ and so that $i(a)=h^{j-1}\pi(\mathcal{O})\in
h^{-1}\rKh_h^{0,q}(K)$. Since $S(\pi(\mathcal{O})) =
\overline{\mathcal{O}}$, there is an element $b = \pm S(a)\in
\Kh_h^{0,q-1}(K)$ with $i(b) = h^{j}\overline{\mathcal{O}}$. Also,
$i(I(b)) = \pm h^{j}\mathcal{O}$. So,  $i(b)$ and $i(I(b))$ span 
$h^{-1}\Kh_h^{0,q-1}(K)$. Further, the images of $b$ and
$S(\check{a})\in \oKh^{0,q-1}(K)$ in $\Kh^{0,q-1}(K;\FF_2)$ agree,
as do the images of $I(b)$ and $S(\check{a})$, so $q-1$ is
completely full, as desired.

The analogues of these last inequalities for the $o$ variants are
similar but easier to prove, following an argument of
Sch\"utz~\cite[Proposition~5.5]{Schutz:integral-s}.
\end{proof}

We can now prove Theorem~\ref{thm:local-is-e}, that for knots $r_c$
applied to the knot and its mirror detects local equivalence:
\begin{proof}[Proof of Theorem~\ref{thm:local-is-e}]
  Write $\LEO(K)=(C,D,f)$.  Since $r_c(\LEO(K))=0$, there are elements
  $a\in H^{0,1}\!(D)$ and $\check{a}\in H^{0,1}(C)$ with
  $f\circ j(\check{a}) = p(a)\in H^{0,1}(D_{h=0};\FF_2)$ and
  $i(a)\in H^{0,1}(h^{-1}D)$ primitive. We claim that $i(a)$ generates
  $H(h^{-1}D)$ as a module over $h^{-1}\BNring$. Suppose not. By
  Lemma~\ref{lem:plus-minus}, the map
  $i\co \Kh^{0,3}_h(K)\to
  h^{-1}\Kh_h^{0,3}(K)\cong\ZZ^2$ is non-zero. Hence, by the
  universal coefficient theorem, the map
  $i\co \Kh_h^{0,3}(K;\QQ)\to h^{-1}\Kh^{0,3}_h(K;\QQ)\cong\QQ^2$ is non-zero, so
  $s_\QQ(K)\geq 2$. Hence, $s_{\QQ}(\Kbar)\leq -2$, but by construction
  $s_{\ZZ}(\Kbar)\leq s_{\QQ}(\Kbar)$, so, by
  Lemma~\ref{lem:obstruction_order}, $s_c(\Kbar)\leq -2$, a
  contradiction.

  Now, we proceed as in the proof of
  Theorem~\ref{thm:reduced-local-is-e}.
  Fix cocycle representatives for $a$ and $\check{a}$ and define
  $\alpha\co \ZZ[X]/(X^2)\{1\}\to C$ by $\alpha(1)=\check{a}$ and
  $\beta\co \BNring\{1\}\to D$ by $\beta(1)=a$. We just verified that
  $\beta$ induces a quasi-isomorphism $h^{-1}\BNring\to h^{-1}D$;
  since both sides are free, $\beta$ is also a homotopy
  equivalence. The homotopy between $f\circ \alpha$ and
  $\beta\circ \id$ is induced by the fact that
  $[p(a)]=[f\circ j(\check{a})]\in H(D_{h=0};\FF_2)$. So, $(\alpha,\beta)$
  induces a local map $\LEO(U)\to\LEO(K)$.

  Applying the same argument with $\Kbar$ in place of $K$ gives a
  local map $\LEO(U)\!\to \!\LEO(\Kbar)=\LEO(K)^*$. Dualizing gives a
  local map $\LEO(K)\to\LEO(U)^*=\LEO(U)$, concluding the proof that
  $\LEO(K)$ is locally equivalent to $\LEO(U)$.
\end{proof}

The second component of the proofs of Theorems~\ref{thm:bockstein_s_vals} and~\ref{thm:s-rels} is
showing that, in fact, the three versions each comprehensive invariant
agree, as do the three versions of the $\beta_n$-refined invariant for
$\LEO(K)$:

\begin{lemma}\label{lem:r-is-s}
Let $K$ be a knot and $1\leq n\leq \infty$. Then for $\gamma\in
\{c,o\}$ we have
\[
  r_\gamma(\LEO(K)) = s_\gamma(\LEO(K)) \hspace{0,4cm}\mbox{and}\hspace{0.4cm}  r^{\beta_n}(\LEO(K)) = s^{\beta_n}(\LEO(K)).
\]

\end{lemma}

\begin{proof}
We focus on the case $r_c(\LEO(K)) = s_c(\LEO(K))$; the other cases are similar. By Lemma \ref{lem:obstruction_order}, it suffices to show that if $s_c(\LEO(K)) = s_\ZZ(K)$, then $s_c(\LEO(K)) \leq r_c(\LEO(K))$. So assume that $q=s_\ZZ(K) -1$ is completely full and let $a,b\in \Kh_h^{0,q}(K)$ and $\check{a},\check{b}\in \Kh_o^{0,q}(K)$ witness this fact.

By Lemma \ref{lem:minus-plus}, there exist $m,n\in \ZZ$ with $i(T(ma+nb))$ primitive in $h^{-1}\Kh^{0,q}_h(K)$. By Diagram~\eqref{eq:shuma_diag}, we have $p(T(ma+nb)) = \nu(m j(\check{a})+n j(\check{b}))$. It remains to lift Shumakovitch's $\nu$ to odd Khovanov homology.

Khovanov homology with coefficients in $\FF_2$ splits~\cite[Corollary 3.2.C]{Shu14:torsion} and so does Bar-Natan homology over $\FF_2$~\cite{MR3482492}. We want to analyze the matrix for the map
\[
\tilde{\nu} \colon \rKh^{0,q-1}(K;\FF_2)\oplus \rKh^{0,q+1}(K;\FF_2)\to \rKh^{0,q+1}(K;\FF_2) \oplus \rKh^{0,q+3}(K;\FF_2)
\]
induced by $\nu$ and Shumakovitch's splitting.

The sequence (\ref{eq:ses-bn}) can be split over $\FF_2$ by using
$r\co\rKhCx_h(K;\FF_2)\{1\}\to \KhCx_h(K;\FF_2)$, $r(u\otimes 1) = u+XT(u)$. After setting $h=0$, $T$ turns into $\nu$ and because $X\nu +\nu X = \id$~\cite[\S 3.2]{Shu14:torsion}, the splitting $\tilde{r}\colon \rKhCx(K;\FF_2)\{1\}\to \KhCx(K;\FF_2)$ is given by $\tilde{r}(u\otimes 1) = \nu(Xu)$.
In particular, for $u\otimes 1 \in \rKh^{0,q-1}(K;\FF_2)$ we get $\tilde{\nu}(u\otimes 1) = \nu (\nu(Xu)) = 0$, as $\nu^2=0$~\cite[\S 3.2]{Shu14:torsion}.

For $u\otimes 1\in \rKh^{0,q+1}(K;\FF_2)$ observe that inclusion into $\Kh^{0,q}(K;\FF_2)$ is given by sending $u\otimes 1$ to $XI(u)$. As $h=0$ in our situation, $XI(u) = Xu$ by Lemma \ref{lem:chainT}. Then $\nu(Xu) = u +X\nu(u)$. Projecting $\Kh^{0,q+2}(K;\FF_2)$ to the summand $\rKh^{0,q+3}(K;\FF_2)$ is done by sending $v$ to $\nu(v)\otimes 1$, so $\nu(Xu)$ is sent to $0$ in this summand. Finally, projection from $\Kh^{0,q+2}(K;\FF_2)$ to the summand $\rKh^{0,q+1}(K;\FF_2)$ is just the standard change of coefficients $w\mapsto w\otimes 1$. Therefore $\nu(Xu) = u+X\nu(u)$ is sent to $u\otimes 1$.

It follows that $\tilde{\nu}$ is the identity on $\rKh^{0,q+1}(K;\FF_2)$ and zero between all other summands. The splitting of odd Khovanov homology with integral coefficients~\cite{ORSz13:odd} descends to the splitting over $\FF_2$ described above, so we can lift $\nu$ to odd Khovanov homology by using the same matrix as for $\tilde{\nu}$ (i.e., $(a,b)\mapsto(b,0)$). This finishes the proof.
\end{proof}

Lemmas~\ref{lem:bockstein_s_vals} and~\ref{lem:r-is-s} imply that the
reduced invariants $\stil^{\beta_n}(\LEO(K))$ and $\stil_o(\LEO(K))$
also agree with $r^{\beta_n}(\LEO(K))$ and $r_o(\LEO(K))$.

\begin{remark}\label{rem:no-examples}
  We do not know examples of knots $K$ where
  $s^{\beta_n}(\LEE(K))\neq r^{\beta_n}(\LEE(K))$ (respectively
  $s^{\beta}(\LEO(K)) \neq r^{\beta}(\LEO(K))$), and suspect they may also
  be equal in general.
\end{remark}

\begin{proof}[Proof of Theorem~\ref{thm:bockstein_s_vals}]
  This is immediate from Lemmas~\ref{lem:bockstein_s_vals} and~\ref{lem:r-is-s}.
\end{proof}

\begin{proof}[Proof of Theorem~\ref{thm:s-rels}]
  This is immediate from Lemmas~\ref{lem:obstruction_order}
  and~\ref{lem:r-is-s}.
\end{proof}

\subsection{More relations between the refined \texorpdfstring{$s$}{s}-invariants}
\label{sec:more-relations}

There are two main uses for concordance invariants: proving that a
knot $K$ is not slice (concordant to the unknot), and proving that
knots $K_1$ and $K_2$ are not concordant. Given concordance invariants
$\alpha$ and $\beta$, we say that $\alpha$ is a \emph{stronger slice
  obstruction} than $\beta$ if whenever $\beta$ detects non-sliceness
of $K$, so does $\alpha$. We say $\alpha$ and $\beta$ are
\emph{independent} if there are knots $K_1,\dots,K_4$ so that
$\alpha(K_1)=\alpha(K_2)$, $\beta(K_1)\neq\beta(K_2)$,
$\alpha(K_3)\neq\alpha(K_4)$, and $\beta(K_3)=\beta(K_4)$. It is
possible for $\alpha$ to be a stronger slice obstruction than $\beta$
but for $\alpha$ and $\beta$ to be independent; depending on the
problem, one property may be more relevant than the other. This
subsection is focused on showing that some of the concordance
invariants described above are stronger slice obstructions than
others. Independence of many of these invariants follows from the
computations in Section~\ref{sec:computations} (see also
Examples~\ref{eg:alg-example1} and~\ref{eg:alg-example2}).

We already know that some of the invariants are stronger slice
obstructions than others. For example,
Theorem~\ref{thm:bockstein_s_vals} implies that for knots with
$s_{\FF_2}(K)=0$, if $n>m$ then $s^{\beta_n}$ is a stronger slice
obstruction than $s^{\beta_m}$. On the other hand, the computations in
Section~\ref{sec:computations} show that neither $s^{\beta_\infty}$
nor $s^\beta$ is a stronger slice obstruction than the
other. Similarly, for knots with $s_{\FF_2}(K)=0$,
Lemma~\ref{lem:obstruction_order} implies that $r_c$ is a stronger
slice obstruction than the others in that lemma. (See also
Theorem~\ref{thm:local-is-e}.)

We turn now to the relationship between the two families of invariants
($s^\alpha$ and $s_\gamma$).

Knowing $H^0(C)$ as an abelian group leads to
an understanding of the image of the Bockstein homomorphism
$\beta_n$, namely, the image is generated by the $\ZZ/(2^k)$-summands
in $H^0(C)$ with $k\leq n$ via the map
$H^0(C)\otimes_\ZZ \FF_2\to H^0(C;\FF_2)$. So,
$q=s_{\FF_2}(K)$ is $\beta_n$-reduced-full if there is a non-$h$-torsion
element of $H^{0,q}(D;\FF_2)$ whose image in
$H^{0,q}(C;\FF_2)$ is in the image of the $\ZZ/(2^k)$-summands
with $k\leq n$. In particular, as noted above, if $s_{\FF_2}(K)$ is
$\beta_n$-reduced-full (so $\stil^{\beta_n}(\LEO(K))=s_{\FF_2}(K)+2$), it
is also $\beta_p$-reduced-full for $p\geq n$. This also tells us that
$q$ is oddly reduced-full, so $\stil_o(K)=s_{\FF_2}(K)$. So,
we have proved:
\begin{lemma}\label{lem:s-beta-to-s-o}
  If $\stil^{\beta_n}(\LEO(K))=s_{\FF_2}(K)+2$ then
  $\stil_o(\LEO(K))=s_{\FF_2}(K)$, and similarly for
  $\LEE(K)$.
\end{lemma}

On the other hand, if we dualize the complex, the Bockstein
homomorphism reverses direction, so the corresponding
$\ZZ/(2^k)$-summand no longer lifts to an integral class. So, it seems
reasonable to expect that, often, if
$\stil^{\beta_n}(\LEO(K))=s_{\FF_2}(K)+2$ then
$\stil_o(\Kbar)=s_{\FF_2}(\Kbar)-2$ (where $\Kbar$ denotes the
mirror). A precise results along these lines is the following:

\begin{proposition}\label{prop:e-the-best}
  The concordance invariant 
  $\svectil_c(K) = \bigl(\stil_c(\LEO(K)),\stil_c(\LEO(\Kbar))\bigr)$ is a
  stronger slice obstruction than $s_{\FF_2}(K)$ and all of
  $s^{\beta_n}$, $s^{\beta}$, $s_o$, $\stil^{\beta_n}, \stil^{\beta},
  \stil_o, r^{\beta_n}$, and $r^{\beta}$ applied to either $\LEO(K)$
  or $\LEE(K)$.
  Also, the concordance invariant
  $\bigl(\stil_o(\LEO(K)),\stil_o(\LEO(\Kbar))\bigr)$ is a
  stronger slice obstruction than $s_{\FF_2}(K)$ and
  $\stil^{\beta_n}(\LEO(K))$.
\end{proposition}
\begin{proof}
  For the first statement, by Theorems~\ref{thm:s-rels}
  and~\ref{thm:local-is-e}, if $\svectil_c(K)=(0,0)$ then $\LEO(K)$ is
  locally equivalent to $\LEO(U)$. It follows that $\LEE(K)$ is also
  locally equivalent to $\LEO(U)$, since $\LEE(K)$ is determined by
  $\LEO(K)$. So, all of the local equivalence invariants for $\LEO(K)$
  and $\LEE(K)$ must vanish. The proof of the second statement is
  similar, using two-reduced local equivalence and
  Theorem~\ref{thm:reduced-local-is-e} in place of
  Theorem~\ref{thm:local-is-e}.
\end{proof}

\begin{corollary}
  If $\stil^{\beta_n}(\LEO(K))=s_{\FF_2}(K)+2$ then
  $\stil^{\beta_n}(\LEO(\Kbar))=-s_{\FF_2}(K)$.
\end{corollary}
\begin{proof}
  By Lemma~\ref{lem:s-beta-to-s-o}, if
  $\stil^{\beta_n}(\LEO(K))=s_{\FF_2}(K)+2$ and
  $\stil^{\beta_n}(\LEO(\Kbar))=-s_{\FF_2}(K)+2$ then
  $\stil_o(\LEO(K))=s_{\FF_2}(K)$ and
  $\stil_o(\LEO(\Kbar))=s_{\FF_2}(\Kbar)=-s_{\FF_2}(K)$, but
  then Corollary~\ref{cor:stilc-triv} implies that
  $\stil^{\beta_n}(\LEO(K))=s_{\FF_2}(K)$, a contradiction.
\end{proof}

Finally, for alternating knots, the techniques in this paper do not
give new information:

\begin{proposition}
\label{prop:alt}
Let $K$ be an alternating knot. Then, $\LEO(K)$ is locally equivalent
to $\LEO(U)\{\sgn(K)\}$, where $\sgn(K)$ denote the signature of
$K$. In particular, if we let $(C,D,f) = \LEO(K)$ or $\LEE(K)$, then
\[
s^\alpha(C,D,f) = \tilde{s}^\alpha(C,D,f) = r^\alpha(C,D,f) = \sgn(K)
\]
for all $\alpha\in \{\beta_n,\beta\}$ with $1\leq n \leq \infty$, and
\[
s_\gamma(C,D,f) = \tilde{s}_\gamma(C,D,f) = r_\gamma(C,D,f) = \sgn(K)
\]
for $\gamma\in \{o, c\}$. 
\end{proposition}

\begin{proof}
  For the first statement, let $(C,D,f)=\LEO(K)$ and consider
  $r_c(C,D,f)$. By an earlier result~\cite[Lemma
  4.11]{Schutz:integral-s}, $s_\ZZ(K) = s_{\FF_2}(K)\eqqcolon s$, and
  hence by Lemma~\ref{lem:allszettsame}, there exists
  $a\in H^{0,s+1}(D)$ with $i(a)\in H^{0,s+1}(h^{-1}D)$
  primitive. Since odd Khovanov homology is torsion-free for
  alternating knots, we get that
  $f\circ j\colon H^{0,s+1}(C)\to H^{0,s+1}(D_{h=0};\FF_2)$ is
  surjective, and therefore $s+1$ is completely half-full, that is,
  $r_c(\LEO(K)) = s$. Applying the same argument to the mirror of $K$
  gives $r_c(\LEO(K)^*)=-s$. So, by
  Corollary~\ref{cor:stilc-triv-unreduced}, $\LEO(K)$ is locally
  equivalent to $\LEO(U)\{\sgn(K)\}$. Hence, their local equivalence
  invariants agree.
\end{proof}

\section{Computations}\label{sec:computations}
\begin{table}
  \begin{tabular}{crrrrr}
    \toprule
      crossings & $\svec^{\Sq^1_e}(K)$ & $\svec^{\Sq^1_o}(K)$
                     & $\svec^\beta(K)$ & $\svec^{\beta_{15}}(K)$ & $\svectil_c(K)$ \\
      \midrule
      \hphantom{0}9 & 0 & 1 & 1 & 1 & 1 \\
      10 & 0 & 2 & 2 & 2 & 2 \\
      11 & 0 & 10 & 10 & 10 & 10  \\ 
      12 & 0 & 49 & 49 & 50 & 50 \\
      13 & 0 & 286 & 285 & 297 & 297 \\
      14 & 2 & 1,718 & 1,717 & 1,797 & 1,797 \\
      15 & 41 & 11,244 & 11,239 & 11,808 & 11,819 \\
      16 & 162 & 73,814 & 73,787 & 77,873 & 77,929 \\
    \bottomrule
  \end{tabular}

  \vspace{1em}
  
  \caption{This table shows the number of prime knots with a specific
    number of crossings where each invariant is non-constant and hence
    carries more information than $s_{\FF_2}$.  The total number where
    $\svectil_c(K)$ is non-constant is 5.4\% of the 1.7 million prime
    knots with at most 16 crossings.}
  \label{tab:s-invars}
\end{table}

Many of the knot invariants in
Section~\ref{sec:knot-case} have been implemented in the latest
version of \texttt{KnotJob} \cite{knotjob}, including the strongest
among them (cf.~Corollary~\ref{cor:CLEO-to-rCLEO-iso}).  Specifically,
\texttt{KnotJob} can compute\footnote{\texttt{KnotJob} uses the
  notations `BLS-odd' for $\svec^{\beta_{15}}$, `$\Sq^1$-sum' for
  $\svec^\beta$, and `complete LS-Inv' for $\svectil_c$.}
\begin{align*}
  \svec^{\beta_{15}}(K) &= \big(r^{\beta_{15}}(\LEO(K)), \
                          -r^{\beta_{15}}(\LEO(\Kbar))
                          \big),\\
  \svec^{\beta}(K) &= \big(r^{\beta}(\LEO(K)), \
                          s^{\beta}(\LEO(K)), \ 
                         -r^{\beta}(\LEO(\Kbar)), \ 
                         -s^{\beta}(\LEO(\Kbar))\big),\\
  \svectil_c(K) &= \big(\stil_c(\LEO(K)), \ 
                   -\stil_c(\LEO(\Kbar))\big).                   
\end{align*}
The invariants
\begin{align*}
\svec^{\Sq^1_o}(K)  &= \big(r^{\beta_{1}}(\LEO(K)),\
                                          s^{\beta_{1}}(\LEO(K)), \
                                          -r^{\beta_{1}}(\LEO(\Kbar)),\
                                          -s^{\beta_{1}}(\LEO(\Kbar))\big),\\
\svec^{\Sq^1_e}(K)  &= \big(r^{\beta_{1}}(\LEE(K)),\
                           s^{\beta_{1}}(\LEE(K)), \
                          -r^{\beta_{1}}(\LEE(\Kbar)),\
                          -s^{\beta_{1}}(\LEE(\Kbar))\big)
\end{align*}
had already been implemented in an earlier version of \texttt{KnotJob}. (By Lemma~\ref{lem:r-is-s}, we now know there is some redundancy in $\svec^{\Sq^1_o}(K)$.)

For alternating knots, these tuples are constant with all entries
equal to $s_{\FF_2}(K)$ by Proposition~\ref{prop:alt}.  Also, for at
least $\svec^{\beta_n}(K)$, if the tuple is constant, then the common
value is $s_{\FF_2}(K)$; this follows from
Theorem~\ref{thm:bockstein_s_vals} and the fact that
$s_{\FF_2}(\Kbar) = -s_{\FF_2}(K)$.  Table~\ref{tab:s-invars} shows
the number of knots with a non-constant invariant with a specific
number of crossings $\leq 16$. For knots in this range, the invariant
for $\beta = \Sq^1_o+\Sq^1_e$ is very similar to the invariant for
$\Sq^1_o$. This is not surprising, as $\Sq^1_e$ has little impact for
these knots. There are only four knots with at most $15$ crossings
such that $\svec^\beta$ is non-constant while $\svec^{\Sq^1_o}$ is
constant, and for all of these $\svec^{\Sq^1_e}$ is also
non-constant. Among $16$-crossing knots, there are three with
non-constant $\svec^\beta$ for which both $\svec^{\Sq^1_o}$ and
$\svec^{\Sq^1_e}$ are constant. For these, we also have
$\svec^{\beta_{15}}$ constant. For all these knots, we found that if
any of the invariants $\svec^{\Sq^1_e}(K)$, $\svec^{\Sq^1_o}(K)$,
$\svec^\beta(K)$, and $\svec^{\beta_{15}}(K)$ are non-constant, then
so is $\svectil_c$ (cf.\ Proposition~\ref{prop:e-the-best}).

We have computed $\svec^{\beta_{15}}$ instead of
$\svec^{\beta_\infty}$ because its implementation is faster, and the
two only differ if the odd Khovanov homology has torsion of order
$>2^{15}$. In particular, for all the knots in
Table~\ref{tab:s-invars}, $\svec^{\beta_{15}}=\svec^{\beta_\infty}$.
Computation times for $\svec^{\beta_{15}}$ and $\svec^{\beta}$ are
faster than for $\svectil_c$, because the computations involve modular, rather than integer, arithmetic.



\subsection{Obstructing sliceness}

There are 352 million prime knots with at most 19
crossings~\cite{Burton2020}.  Dunfield and Gong have an ongoing
project to identify which of these knots are smoothly
slice~\cite{DunfieldGong2023}.  As of April 2023, there were only
17{,}991 knots (0.005\%) whose slice status remained unknown. In
particular, each of these knots has signature 0, an Alexander
polynomial that satisfies the Fox-Milnor criterion, all of
Herald-Kirk-Livingston's twisted Alexander
polynomials~\cite{HeraldKirkLivingston2010} that were computed are
consistent with the knot being slice, the Heegaard Floer invariants
$\tau$, $\nu$, and $\varepsilon$ are all $0$, as are the even
$s$-invariant over the fields $\FF_2$, $\FF_3$, and $\QQ$, the even
and odd $\Sq^1$-refined $s$-invariants, and Sch\"utz's $s^\ZZ$
invariant~\cite{Schutz:integral-s}.  Of these 17,991 inscrutable
knots, 826 have non-zero $\svec^{\beta_{15}}$ invariant, 64 have
non-zero $\svec^{\beta}$ invariant, and 890 have nonvanishing
$\svectil_c$ invariant.  The knots where $\svec^{\beta_{15}}$ and
$\svec^{\beta}$ are non-zero are disjoint, and $\svectil_c$ provided
no new information, but was non-zero if and only if one of
$\svec^{\beta_{15}}$ and $\svec^{\beta}$ was.  Thus these invariants
collectively obstruct sliceness of 890 knots, reducing the number of
mystery knots by 5\%.  Of these 890 knots, at least 832 are
topologically slice as they have Alexander polynomial equal to
1~\cite[Theorem~11.7B]{FreedmanQuinn}.  By
Corollary~\ref{cor:CLEO-to-rCLEO-iso}, no further slice obstructions
for these inscrutable knots can be obtained from $\LEO(K)$.

Another way to put these 890 knots in context is
that Dunfield-Gong found 1.6 million slice knots and showed 350.5
million are not even topologically slice~\cite{DunfieldGong2023}.
The smooth slice invariants
from Khovanov homology mentioned, namely
($s_{\FF_2}, s_{\FF_3}, s_\QQ, s^\ZZ, \svec^{\Sq^1},
\svec^{\Sq^1_o}$), provide additional slice obstructions for only
about 12,100 knots.  Thus, the $\svec^{\beta_{15}}$ and
$\svec^{\beta}$ invariants increase that total by 8.1\% to about
13,000.

Also, Owens and Swenton studied \emph{alternating} knots with at most 21
crossings, determining sliceness for all but 3,276 (0.0003\%) of the
1.2 billion such knots~\cite{OwensSwenton2023}.  By
Proposition~\ref{prop:alt}, the
invariants  $\svec^{\beta_{15}}$, $\svec^{\beta}$, and $\svectil_c$ give the
same information as $s_{\FF_2}$ for alternating knots, and so cannot
help resolve these remaining cases.

\subsection{Manolescu-Piccirillo knots}
\label{sec: MP knots}

Manolescu and Piccirillo~\cite{ManolescuPiccirillo2023} gave five
topologically slice knots such that if any of them were smoothly
slice, then one would obtain an exotic 4-sphere.
Nakamura~\cite{Nakamura2022} then showed that none of these
knots are smoothly
slice by using a 0-surgery homeomorphism to relate slice properties of
two knots stably after a connected sum with some 4-manifold.  We found
the $\svec^\beta$ invariant can also be used to check directly that these
five knots are not slice; in contrast, the $\svec^{\beta_{15}}$,
$\svec^{\Sq^1}$, and $\svec^{\Sq^1_o}$ invariants of these knots all
vanish.

Manolescu-Piccirillo's strategy was based on finding pairs of knots
$K$ and $K'$ with the same 0-surgery where $s_{\FF_2}(K) = 0$ and
$s_{\FF_2}(K') \neq 0$. From Dunfield-Gong's
data~\cite{DunfieldGong2023}, we found knots $K$ and $K'$ with the
same 0-surgery where $\svec^\beta(K) = 0$ and
$\svec^\beta(K') \neq 0$, but where all the invariants
($s_{\FF_2}, s_{\FF_3}, s_\QQ, s^\ZZ, \svec^{\Sq^1}, \svec^{\Sq^1_o}$)
vanish for both knots; here $K$ had 17 crossings and $K'$ had 40
crossings.
We also found analogous such pairs for $\svec^{\beta_{15}}$, with the
caveat that we could not verify that $s^\ZZ(K') = 0$ as $K'$ had 50
crossings.

\bibliographystyle{amsalphaurl}
\def\MR#1{\href{http://www.ams.org/mathscinet-getitem?mr=#1}{MR#1}}
\bibliography{slocal}
\end{document}